\documentclass[11pt]{article}

\usepackage{amscd,amsmath,amsfonts,amssymb,amsthm,mathrsfs,color,cite} 
\usepackage{dsfont} \usepackage{leftidx}
\usepackage{tikz} \usetikzlibrary{positioning, arrows.meta}
\usepackage{enumerate}
\usepackage{appendix}
\usepackage{graphicx}
\usepackage{hyperref}
\usepackage{appendix}
\usepackage[left=2.5cm, right=2.5cm, top=2.5cm, bottom=2.5cm]{geometry}
\usepackage{caption}
\usepackage{subcaption}
\usepackage{xcolor}
\usepackage[bb=boondox]{mathalfa}
\usepackage{comment}

\numberwithin{equation}{section} 

\usepackage{setspace}

\theoremstyle{plain}
\newtheorem{theorem}{Theorem}[section]
\newtheorem{lemma}[theorem]{Lemma}
\newtheorem{corollary}[theorem]{Corollary}
\newtheorem{proposition}[theorem]{Proposition}

\theoremstyle{definition}
\newtheorem{definition} 
[theorem]{Definition}

\newtheorem{remark}[theorem]{Remark}

\newtheorem{question}{Question}
\newcommand{\tr}{\mathrm{Tr}}

\begin{document}

\setcounter{tocdepth}{2}

\title{\bf{Edge statistics for random band matrices
}}

\author{
Dang-Zheng Liu\footnotemark[1] ~ and    Guangyi Zou\footnotemark[1]}
\renewcommand{\thefootnote}{\fnsymbol{footnote}}
\footnotetext[1]{School of Mathematical Sciences, University of Science and Technology of China, Hefei 230026, P.R.~China. 
dzliu@ustc.edu.cn, zouguangyi2001@mail.ustc.edu.cn}

\maketitle

\begin{abstract}
We consider Hermitian and symmetric random band matrices on the $d$-dimensional lattice $(\mathbb{Z}/L\mathbb{Z})^d$ with bandwidth $W$, focusing on local eigenvalue statistics at the spectral edge in the limit $W\to\infty$. Our analysis reveals a critical dimension $d_c=6$ and identifies the critical bandwidth scaling as $W_c=L^{(1-d/6)_+}$. 
In the Hermitian case,  we establish   the Anderson transition   for all dimensions $d<4$,   and  GUE edge universality when $d\geq 4$ under the condition $W\geq L^{1/3+\epsilon}$ for any $\epsilon>0$.   In the symmetric case, we also establish 
   parallel   but more subtle transition phenomena after tadpole diagram renormalization. 
These findings   extend Sodin's pioneering work [Ann. Math. 172, 2010],   which was limited to the one-dimensional case  and did not address the critical phenomena.
\end{abstract}

\tableofcontents

\section{Introduction}

\subsection{Random band matrices}
 
Given a $d$-dimensional lattice  
\begin{equation}
    \Lambda_L=\Big[-\frac{1}{2}L,\frac{1}{2}L\Big)^d \bigcap \mathbb{Z}^d,
\end{equation}
where   $L$ is a positive integer and    $N:=L^d$ is the size of $\Lambda_L$, we introduce a canonical representative for $x \in \mathbb{Z}^d$ as
\begin{equation}
    [x]_L: = \big(x + L\mathbb{Z}^d\big) \cap \Lambda_{L}
\end{equation}
and also a periodic distance on $\Lambda_{L}$ via   the $l^2$ norm
\begin{equation}
    \|x-y\|: = \|[x-y]_L\|_{2}.
\end{equation}
In a canonical form, a symmetric or Hermitian (correspondingly, Dyson index $\beta=1,2$) random band matrices (RBM for short) can be defined as
\begin{equation}
   H = \big(H_{xy}\big)_{x,y\in \Lambda_{L}}, \quad   H_{xy}:=\sigma_{xy}A_{xy},
\end{equation}
where $\{A_{xy}\}$ are independent random variables with mean zero and variance one, up to symmetry of $H_{xy}=\overline{H_{yx}}$. The variance $\sigma^{2}_{xy}$ depends on the distance between the two sites, and  typically decays with distance on a characteristic length scale $W$, called the bandwidth of $H$. When $W=L$ and all variances are equal, the RBM reduces to the usual Wigner matrices, acting as a mean-field model. The case $W = O(1)$ corresponds to the famous Anderson model \cite{anderson1958absence}:
\begin{equation}
    H_{\text{AM}}=-\Delta+ \lambda V,
\end{equation}
where $\Delta$ is the discrete Laplacian on $\Lambda_{L}$ and $V$ is a random potential with i.i.d. random variables $V_{x}$. The parameter $\lambda>0$ measures the strength of the disorder and is expected to be an analogy of $1/W$.     This model exhibits short-range interaction and spatial structure, distinguishing it from mean-field models.

As $W$ varies, random band matrices become more realistic interpolating models between Wigner matrices and Anderson models. The entries represent quantum transition rates between two   quantum states. More importantly, a sharp Anderson metal-insulator transition in the bulk of the spectrum is conjectured to occur at some critical bandwidth $W_c$:
\begin{itemize}
    \item[(i)] When $W\ll W_c$, $H$ has localized eigenvectors and Poisson statistics;
    \item[(ii)] When $W\gg W_c$, $H$ has delocalized eigenvectors and GOE/GUE bulk statistics;
    \item[(iii)] When $W\sim W_c$, $H$ has coexisting eigenvector states and crossover eigenvalue statistics.
\end{itemize}
Here
\begin{equation}  
    W_c = 
    \begin{cases} \sqrt{L}, &  d=1,\\
        \sqrt{\log L}, &  d=2,\\
        O(1),& d\ge 3,
    \end{cases}
\end{equation} 
as reported in \cite{bourgade2018random, casati1990scaling1, casati1990scaling, fyodorov1991scaling,spencer2011random,spencer2012susy} for simulations, non-rigorous supersymmetry, or heuristic arguments.

Establishing the Anderson metal-insulator transition for the Anderson models and random band matrices \cite{aizenman2015random,spencer2011random} remains a challenging problem. Regarding the localization of the Anderson models, significant mathematical progress includes works of Goldsheid-Molchanov-Pastur \cite{gol1977pure} in one dimension, Fr\"ohlich-Spencer \cite{frohlich1983absence} by multiscale analysis, Aizenman-Molchanov \cite{aizenman1993localization} by fractional moment method for higher dimensions, and recent works of Ding-Smart \cite{ding2020localization} and Li-Zhang \cite{li2022anderson} for Anderson Bernoulli models in dimensions 2 and 3. However, much less is known about the delocalization conjectured to occur in dimensions $d\geq 3$; see \cite{aizenman2013resonant} and references therein.

For random band matrices, there are numerous partial results on localization or delocalization, including \cite{bao2017delocalization,bourgade2020random,benaych2014largest,disertori2002density,erdHos2011quantum,erdHos2011quantum2,erdHos2013delocalization,he2019diffusion,schenker2009eigenvector,sodin2010spectral,shcherbina2017characteristic,shcherbina2021universality,shcherbina2022susy,yang2021delocalization,yang2021random,yang2022delocalization,xu2022bulk} (this is far from an exhaustive list!). See \cite{bourgade2018random,bourgade2020random,yang2021delocalization} for detailed reviews on recent progress. As far as we know, the best delocalization result for high-dimensional band matrices was obtained in a series of works \cite{yang2021delocalization,yang2022delocalization,xu2022bulk} under the assumption $d\geq 7$ and $W\gg L^{\epsilon}$.
On the other hand, for the local eigenvalue statistics in the bulk of the spectrum, the universality result was proved in \cite{bourgade2020random} when $W\gg L^{3/4}$ in dimension $d=1$, and in \cite{xu2022bulk} when $W\gg L^{ 95/(d+95)}$ in dimension $d\geq 7$. The crossover near the critical bandwidth $W\sim \sqrt{L}$ in $d=1$ has been proved for 2-point correlation functions of characteristics polynomials under Gaussian entries and a special variance profile in \cite{shcherbina2021universality,disertori2021density,shcherbina2022susy}; see \cite{shcherbina2018transfer} for more details. It is worth stressing that there are many works about Wigner-type matrices, see \cite{anderson2010introduction,bourgade2018random,erdHos2017dynamical} and references therein.

This paper focuses on the local eigenvalue statistics at the spectral edge of the spectrum, building upon fundamental developments in random matrix theory. The study of edge statistics originated with the seminal works of Tracy and Widom   \cite{tracy1994level,tracy1996orthogonal} and Forrester \cite{forrester1993spectrum}, who introduced the now ubiquitous Tracy-Widom distributions and Airy point processes for Gaussian unitary and orthogonal ensembles (GUE,  GOE)  respectively.  These universal edge statistics were   subsequently   proven to hold for Wigner matrices (see e.g. \cite{soshnikov1999universality}) with higher moment assumptions, and finally a necessary and sufficient characterization was established by Lee and Yin \cite{lee2014necessary}. Methodologically, Soshnikov \cite{soshnikov1999universality} pioneered the moment approach to prove edge universality, while Feldheim and Sodin \cite{feldheim2010universality} later developed the more sophisticated Chebyshev polynomial method, extending universality results to both Wigner matrices and sample covariance matrices. 
Furthermore, Sodin's seminal work \cite{sodin2010spectral}  made a crucial breakthrough by identifying the critical bandwidth scaling $W_c=L^{5/6}$ for one-dimensional random band matrices with unimodular entries and with a cutoff variance profile, while rigorously characterizing   edge statistics in both the extended ($W\gg W_c$, supercritical) and localized ($W\ll W_c$, subcritical) regimes. This was followed by significant advances from 
Erd{\H{o}}s and Knowles \cite{erdHos2011quantum,erdHos2011quantum2}, who employed Chebyshev-Fourier expansions to investigate bulk eigenvector localization and quantum diffusion in band matrices with i.i.d. entries and with   general variance profiles. Their subsequent work \cite{erdHos2015altshuler1, erdHos2015altshuler2} further extended this methodology to analyze mesoscopic  bulk statistics, revealing new insights into the Altshuler-Shklovskii scaling regimes.

In his seminal paper \cite{sodin2010spectral}, Sodin proposed several open problems  concerning random band matrices, two of which are stated below.

\noindent{\bf Problem 1}. {Study the eigenvalue edge statistics in the crossover regime   $W\sim L^{5/6}$ in dimension $d=1$;  see  \cite[Remark V, Sect. 9]{sodin2010spectral}}.

\noindent
{\bf Problem 2}. Establish a natural extension of eigenvalue edge statistics proved in both subcritical and supercritical regimes from dimension $d=1$ to  higher-dimensional lattices ($d>1$);  see  \cite[Remark VIII, Sect. 9]{sodin2010spectral}.

 The primary objective of this paper is to investigate edge statistics and address these two problems. Specifically, we rigorously establish phase transition and prove universality results for random band matrices  with   Gaussian profile and   unimodular entries in arbitrary dimension $d$,  subject to potential bandwidth constraints   when $d\geq 4$. 

 \begin{definition}[{\bf{Unimodular RBM}}] \label{defmodel} A symmetric/Hermitian  matrix $H=(H_{xy})_{x,y\in \Lambda_{L}}$ is called a  unimodular random band matrix with bandwidth $W\leq L/2$ and variance profile  
  \begin{equation} \label{VP}
    \sigma_{xy}^2=   \frac{1}{M}\sum_{n\in \mathbb{Z}^d}f\big(\frac{x-y+nL}{W}\big), \quad  M:=\sum_{x\in \mathbb{Z}^d} f\big(\frac{x}{W}\big),  
 \end{equation} 
 if   $H_{xy}=\sigma_{xy}A_{xy}$ and   the following three assumptions hold: 
 \begin{itemize}
     \item[(A1)] ({\bf Independence}) $ \{A_{xy}\}$ are independent up to   Hermitian   symmetry;
     
     \item[(A2)] ({\bf  Unimodularity}) all the diagonal entries  $\{A_{xx}\}$ are    uniformly distributed  
     on $S^0:=\{+1,-1\}$ and  the off-diagonal entries  $\{A_{xy}\}$  
     are   uniformly distributed 
     on $S^0$    ($\beta=1$) or   $S^1:=\{e^{i\theta}: \theta\in [0,2\pi)\}$   ($\beta=2$); 
    \item[(A3)]({\bf Gaussian 
    profile}) $f(x)$ is the $d$-dimensional Gaussian density with covariance matrix  $\Sigma>0$,
\begin{equation}
    f(x)=\frac{1}{ (\sqrt{2\pi})^{d}\sqrt{\det(\Sigma)}}e^{-\frac{1}{2} x^T\Sigma^{-1}x}.
\end{equation}

 \end{itemize} 
  \end{definition}

The Gaussian variance profile function $f$, which characterizes the bandwidth parameter $W$, serves as a canonical model for rapidly decaying profiles.
Crucially, the identities   $\sum_{y\in \Lambda_L}\sigma^2_{xy}=1$ for all $x$ and the asymptotic relation   $M= W^d(1+o(1))$  play important   roles. These conditions ensure the validity of the semicircle law in the limit  $W \to \infty$, as demonstrated in \cite{bogachev1991level}.

\subsection{Main results}

We establish a three-phase characterization of edge statistics for dimensions $d<4$ and prove the Tracy-Widom distribution for $d\geq 4$ in the regime $W\gg L^{\frac{1}{3}+\epsilon}$ for any positive $\epsilon$. These findings extend the earlier work of Sodin \cite{sodin2010spectral}, which was restricted to the supercritical and subcritical regimes for RBM on  one-dimensional lattices  with cutoff profile.

\begin{theorem}[{\bf Supercritical regime}] \label{metatheoremsup} For  the unimodular 
{\rm{RBM}} 
in Definition   \ref{defmodel},  if 
    \begin{equation} \label{supercond2}
     W\gg 
     \begin{cases} L^{1-\frac{d}{6}}, &  d<4,\\
        L^{\frac{1}{3}+\epsilon},& d\geq 4, 
        \end{cases}
    \end{equation} 
    with any  fixed  $\epsilon>0$,   
    then  as $L\to \infty$ any  $k$-point correlation function for rescaled   eigenvalues  of $H$
    \begin{equation}\{N^{\frac{2}{3}}\big(\lambda_i-2+a_{4}(\beta)\big)\}_{i=1}^N\end{equation}    converge weakly to       those  of GUE and GOE.
Here  the correction term  $a_{4}(2)=\sum_{y\in \Lambda_L}  |\sigma_{0y}|^4$ and  $a_{4}(1)=A(1)$ with $A(1)$  defined    in \eqref{equ:def_A(z)}.
\end{theorem}
 The rescaled  limits  in  the subcritical and critical regimes  are stated  only when $d<4$.

\begin{theorem}[{\bf Subcritical regime}] \label{metatheoremsub} For  the unimodular {\rm{RBM}} in Definition   \ref{defmodel}, assume that $d<4$ and  $n_i=[\tau_i W^{\frac{2 d}{6 -d}}], \tau_i \in (0,\infty)$, $i=1,\ldots, k$. 
If
    \begin{equation} \label{subcond2}
W\ll   L^{1-\frac{d}{6}},  
    \end{equation} 
    then  as $W\to \infty$  

    \begin{equation}\label{equation:Chebyshev_independent}
         \mathbb{E}\Big[\prod_{i=1}^k \mathrm{Tr}\Big( \frac{1}{n_i}\mathcal{P}_{n_i}(H)\Big)\Big]=
         \Big(\frac{N}{W^{\frac{6d}{6 -d}}}\Big)^{k} 
         \Big(\prod_{i=1}^k\big(\frac{1+(-1)^{n_i}}{2}\big) \phi_{1}^{(\mathrm{sub})}(\beta,d;\tau_i)+o(1)\Big) \end{equation}
         where $\phi_{1}^{(\mathrm{sub})}(\beta,d;\tau)$ is given in 
         Definition \ref{def:sub_critical_transform}, the polynomials $\{\mathcal{P}_{n} \}$  are given  in Definition \ref{a4Polynomial}   and should be    replaced by renormalization polynomials $\{\widetilde{\mathcal{P}}_n\}$  in  \eqref{generalpolynomial}   when  $\beta=1$ and $d>1$. 
\end{theorem}
   
   \begin{theorem}[{\bf Critical regime}] \label{metatheoremcri} 
   For  the unimodular {\rm{RBM}}   in Definition   \ref{defmodel}, 
   assume that $d<4$ and  $n_i=[\tau_i (\frac{L}{W})^{2}], \tau_i \in (0,\infty)$, $i=1,\ldots, k$  such that $\sum_{i=1}^k n_i$ is even. 
   If
    \begin{equation} \label{critcond2}
 W= \big(1+o(1)\big)\gamma  L^{1-\frac{d}{6}}
    \end{equation} 
   for   some    $\gamma\in (0,\infty)$,  
         then  as $W\to \infty$   
\begin{equation}\label{criticalChebyshev}
          \mathbb{E}\Big[\prod_{i=1}^k \mathrm{Tr}\Big( \frac{1}{n_i}\mathcal{P}_{n_i}(H)\Big)\Big]=\big(1+o(1)\big) 
         \phi_{k}^{(\mathrm{crit})}(\beta,d,\gamma;\tau_1, \ldots,\tau_k),
                               \end{equation}
where $\phi_{k}^{(\mathrm{crit})}(\beta,d,\gamma;\tau_1, \ldots,\tau_k)$ is given in Definition \ref{def:sub_critical_transform},    
       the polynomials $\{\mathcal{P}_{n} \}$  are given  in Definition \ref{a4Polynomial}   
  and should be    replaced by $\{\widetilde{\mathcal{P}}_n\}$  in  \eqref{generalpolynomial}   when  $\beta=1$ and $d>1$.
\end{theorem}

\begin{remark} In the supercritical regime as in Theorem \ref{metatheoremsup},  to establish the   Tracy-Widom distribution   we need to make   a refined correction   depending on $W$ at the spectral edge.  In dimension $d=1$, the Chebyshev polynomials as test functions,  instead of   $\mathcal{P}_{n}$  and  $\widetilde{\mathcal{P}}_n$, are sufficient to ensure the validness of Theorem \ref{metatheoremsub} and Theorem \ref{metatheoremcri}        in the subcritical and critical regimes. 

\end{remark}

\begin{remark} 
The renormalization polynomials ${\mathcal{P}_n}$ and ${\widetilde{\mathcal{P}}_n}$ exhibit double scaling limits at the spectral edge, with their limiting forms acting as integral transforms of the limiting correlation functions. These transforms coincide with those of Chebyshev polynomials when $\beta=2$ or $(\beta=1, d=1)$, but differ for $\beta=1$ in dimensions $d=1,2,3$. However, the inverse transforms appear to be uniquely determined only for $d=1,2$, a consequence of the tail decay properties of the one-point correlation function (see Appendix  \ref{sec:poly_moment_analysis} or \cite{sodin2021uniqueness}).

In the subcritical regime, the transformed correlation functions factorize into products of rescaled one-point correlation functions with scaling factor $NW^{-6d/(6-d)}$, suggesting the edge eigenvalue point process may converge to a Poisson point process. For the critical regime, we observe various interpolation patterns between Poisson and Airy point processes whose specific forms depend on both $\beta$ and $d$.
\end{remark}

\begin{remark}      ({\bf Stable  profile})  Analogous results to our three main theorems remain valid for variance profiles characterized by by the density of $d$-dimensional $\alpha$-stable distributions, with characteristic function $e^{-c_\alpha|t|^\alpha}$ where $0<\alpha<2$ and $c_\alpha>0$. In this case, the critical exponent becomes $1-\frac{d}{3\alpha}$.
For one-dimensional power-law random band matrices ($d=1$), this model was first proposed by   \cite{mirlin1996transition} and predicted in the physics literature to display multifractality and spectral statistics interpolating between Wigner-Dyson and Poisson statistics. However, rigorous mathematical results remain limited (see e.g. \cite{mirlin1996transition,evers2008anderson}).   These results will be investigated in detail in a separate paper.  
  \end{remark}

\subsection{Key ideas for proofs}

In this paper, we develop the method  of polynomial moments   by employing non-backtracking path expansions to analyze high-dimensional random band matrices, extending beyond classical Chebyshev polynomial techniques. Our approach naturally separates   into two components: combinatorial analysis and analytical treatment.
\begin{itemize}
    \item  \textbf{Combinatorial framework.}  
    This paper primarily focuses on combinatorial analysis. For Hermitian matrices, we implement non-backtracking path expansions by introducing a family of modified Chebyshev polynomials ${\mathcal{P}_n(x)}$. Then we enumerate the non-vanishing paths by classifying non-backtracking paths into distinct diagrams and count the preimage of diagrams which is connected to the evaluation of Feynman graph integrals. 
    
However, in dimensions $d \geq 2$, a significant challenge arises: the divergences of Feynman graph integrals emerge in the real symmetric case that are absent for $d=1$. To resolve these divergences, we introduce renormalization polynomials ${\widetilde{\mathcal{P}}_n(x)}$ that eliminate effects from tadpole diagrams. 

    \item  \textbf{Analytical framework.} 
 Within our analytical approach, we demonstrate that convergence of polynomial moments guarantees that of point correlation functions. We study the asymptotic behavior for $\mathcal{P}_n(x)$ and $\widetilde{\mathcal{P}}_n(x)$ in Appendix~\ref{sec:re_poly}. Furthermore, we address the fundamental questions of existence and uniqueness for the limiting point correlation functions in Section~\ref{sec:cont_thm}, where we develop a continuity theorem to characterize this convergence.
\end{itemize} 
 
  In summary, the proofs of our main results---Theorems~\ref{metatheoremsup}, \ref{metatheoremsub}, and \ref{metatheoremcri}---  integrate several fundamental ingredients as follows.
  
\begin{itemize}
    \item[(i)] \textbf{Method of polynomial moments and non-backtracking path expansion.} 
  For random band matrices with cutoff profile functions and unimodular entries, the Chebyshev polynomial $U_n(H/2)$ admits an exact combinatorial interpretation as a sum over all non-backtracking paths \cite{feldheim2010universality,sodin2010spectral,erdHos2011quantum}. However, this precise correspondence fails for general profile functions, where $U_n(H/2)$ decomposes into a sum of non-backtracking paths plus $O(W^{-d})$ remainder terms \cite{erdHos2011quantum2} - corrections that become non-negligible in dimensions $d\geq 2$.
To address this limitation, we introduce the \textbf{modified Chebyshev polynomial} $\mathcal{P}_n(x)$, which enables a refined path expansion with significantly reduced error terms (Proposition~\ref{prop:2.7}, Theorem~\ref{thm:thm2.5}). Our analysis reveals that these polynomials induce an $O(W^{-d})$ correction to the semicircle law at the spectral edge, distinct from the expected $W^{-\frac{4d}{6-d}}$ edge fluctuations in subcritical regimes. This correction, which becomes relevant for $d\geq 2$, represents a non-universal contribution depending on both the profile function and fourth moments of matrix entries (Lemma~\ref{lem:B.5}).
To our knowledge, the modified Chebyshev polynomials and the non-universal edge corrections appear for the first time in random band matrices.

    \item[(ii)] \textbf{Diagram reduction.} 
  Building upon the non-backtracking   path reduction technique first introduced in \cite{feldheim2010universality}, we present diagram reduction procedure with slight modification in Section~\ref{sec:diagram_reduction} for completeness. Our analysis requires natural extensions of this framework to handle multiple consecutive non-backtracking paths. We need to consider the expectation of $\mathbb{E} \tr(V_n)^t$ for $t>1$ rather only the case of $t=1$ treated in \cite{feldheim2010universality, sodin2010spectral} and in Section \ref{sec:real_upper_bound}, \ref{sec:4.2}, we use it to neglect error terms. We introduce `zig up' operation and the `\textbf{multi-tail diagram}' to overcome the new difficulty(Section \ref{sec:P_powers}). 
    
    \item[(iii)] \textbf{Random walks on the torus $\mathbb{T}^d$.} 
  For each diagram, we count its preimages by the diagram function from Definition~\ref{diagramf} (Proposition~\ref{coro:3.7}, Corollary~\ref{coro:2.19}). The limiting behavior of these diagram functions is connected to random walks on the torus $\mathbb{T}^d$, which exhibit three distinct regimes depending on the relative scaling between bandwidth $W$ and step number $n$ (Section~\ref{sec:diagram_function}). The crossover behavior of random walks ultimately governs the transition between subcritical and supercritical  regimes.
    \item[(iv)] \textbf{Singularity analysis for Feynman graph integrals.} The limiting behavior of diagram functions is closely connected to Feynman integrals from quantum field theory. A crucial new challenge arises in dimensions $d\geq 2$ ($\beta=1$) and $d\geq 4$ ($\beta=2$) due to the emergence of ultraviolet divergences in these Feynman integrals - a phenomenon not present in one-dimensional settings. To our knowledge, this represents the first occurrence of Feynman integral convergence problems in random matrix theory.
Building on techniques from \cite{speer1975ultraviolet,smirnov2013analytic}, we reformulate  a singularity criterion for Feynman graph integral(Proposition \ref{propositionconvergence}). We also adopt a new combinatorial argument to obtain a quantitative bound of Feynman integral(Lemma \ref{lemma:upper_bound_integral}). 
    
    \item[(v)] \textbf{Tadpole diagram renormalization.} In real symmetric case with dimensions $d\ge 2$, the tadpole diagram becomes a singular pattern and leads to a $O(W^{-d})$ correction, which is non-universal and depends on the profile function, to the edge statistics behavior of random band matrices. To see the real fluctuation after the tadpole shift, we introduce loop-free non-backtracking path expansion, see Definition \ref{non-backtracking R-power} and a new family of renormalization polynomial $\{\widetilde{\mathcal{P}}_n(x)\}$, which are totally novel in the field of random matrix theory.
    
    \item[(vi)] \textbf{Analytical properties  of    renormalization  polynomials.} The polynomials $\mathcal{P}_n(x)$ and $\widetilde{\mathcal{P}}_n(x)$ are  defined via   recursion relations and lack explicit closed-form expressions unlike the classical Chebyshev polynomials $U_n(x)$. In the supercritical regime, properly scaled versions of $\mathcal{P}_n(x)$ and $\widetilde{\mathcal{P}}_n(x)$ exhibit identical asymptotic behavior to $U_n(x)$. Surprisingly, the subcritical case reveals more intricate behavior - the limiting forms of $\widetilde{\mathcal{P}}_n(x)$ differ significantly from their supercritical counterparts.
We develop a comprehensive analytical framework for these polynomials via contour integral methods (Appendix~\ref{sec:re_poly}) and establish continuity theorems   for the limiting test functions (Appendix~\ref{sec:cont_thm}). 
    
\end{itemize}

 \subsection{Structure of the paper} 
  The rest of the paper is organized as follows.

\textbf{Section \ref{sec:moment_method_revisited}:}
This section develops the method of polynomial    moments. Section \ref{sec:poly_path} introduces the modified Chebyshev polynomial $\mathcal{P}_n(x)$ and establishes its non-backtracking path expansion $\mathcal{P}_n(H)$. 
Section \ref{sec:reduce_to_Vn} demonstrates how to eliminate error terms, leaving only the essential non-backtracking power $V_n$. 
Building on \cite{feldheim2010universality}, Section \ref{sec:diagram_reduction} presents an adapted diagram reduction procedure for both self-consistency verification and subsequent analysis. 
Finally, Section \ref{sec:count_typical} develops a both lower and upper bound for diagram preimages and restricts our counting framework to typical diagrams.
    
\textbf{Section \ref{section_diagram}:} 
Section \ref{sec:diagram_function} introduces the diagram functions $F_{\mathfrak{D}}(\{n_i\})$ and establishes their asymptotic behavior in the subcritical regime. Section \ref{subsection_singular_diagram} revisits and reformulates the divergence criterion for Feynman integrals. Finally, Section \ref{sec:diagram_func_bound} develops both upper bounds and asymptotic estimates for individual diagram functions $F_{\mathfrak{D}}$.

\textbf{Section \ref{section:asymptotic_expansion}:} 
Section \ref{sec:transform of correaltion}  presents the integral transform   for limiting correlation functions. Section \ref{sec:asy_mixed_moment} establishes the asymptotic behavior of polynomial moments $\mathbb{E}\big[\prod_{i=1}^{k}\tr \mathcal{P}_{n_i}(H)\big]$, while Section \ref{sec:P_powers} extends these results to general mixed powers $\mathbb{E}\big[\prod_{i=1}^{k}\tr (\mathcal{P}_{n_i}(H))^{t_i}\big]$ and derives the tail decay properties for extreme eigenvalues.  The Hermitian case of the main theorems is fully resolved in Section \ref{sec:beta=2_main_thm}.

\textbf{Section \ref{sec:tadpole_diagram}:} 
  Section \ref{sec:renorm_poly} 
 introduces  the renormalization  polynomials $\widetilde{\mathcal{P}}_n(x)$ and develops  the loop-free non-backtracking path expansion for $\widetilde{\mathcal{P}}_n(H)$ with cutoff parameter $R$. Section \ref{sec:real_upper_bound} simplifies the error term analysis through appropriate reductions. The negligibility of these error terms is established in Section \ref{sec:4.2}. Finally, Section \ref{sec:beta=1_main_thm} presents the proof of our main theorems for the real symmetric case.

\textbf{Section \ref{openquestions}:} This section presents   several open   questions.

For clarity of presentation, we defer two key components to appendices: the analysis of random walks on the torus $\mathbb{T}^d$ and the  analytical framework for the polynomial moment method,  
both of central importance to our results and of independent mathematical interest.
\textbf{Appendix \ref{section:heat_kernel_profile}}    focuses on   properties of  random walk on the torus $\mathbb{T}^d$, including local limit theorems and uniform upper bounds for the heat kernel.
\textbf{Appendix \ref{sec:poly_moment_analysis}} 
 develops the analytical framework for the polynomial moment method. Section \ref{sec:re_poly} examines the asymptotic behavior of both $\mathcal{P}_n(x)$ and $\widetilde{\mathcal{P}}_n(x)$, while Section \ref{sec:cont_thm} establishes the continuity theorem.

{\bf Notation.} We use  $C, C_0, C_1, C_2, \ldots,$ to denote arbitrary constants, independent of the large parameters $W$, $L$ and $N$, unless otherwise specified.

\newpage
\section{Method of polynomial moments}\label{sec:moment_method_revisited}
The method of moments, pioneered by Chebyshev in his proof of the central limit theorem for sums of independent random variables, has become an indispensable tool across probability and statistics. Its application to random matrix theory was initiated by Wigner's seminal proof of the semicircle law \cite{wigner1955characteristic}, followed by Soshnikov's breakthrough work establishing edge universality for Wigner matrices through the analysis of asymptotically large moments \cite{soshnikov1999universality}. While the use of Chebyshev polynomials in this framework dates back to \cite{bai1993limit}, recent advances by Feldheim-Sodin \cite{feldheim2010universality,sodin2010spectral,sodin2014several} and Erdős-Knowles \cite{erdHos2011quantum2} have demonstrated the remarkable power of combining moment methods with Chebyshev polynomial techniques for studying spectral properties of large random matrices. Nevertheless, tackling unimodular random band matrices on higher-dimensional lattices with $
d>1$ requires moving beyond conventional Chebyshev polynomials through the development of carefully modified polynomial sequences.

\subsection{Polynomial moments   and non-backtracking walks}\label{sec:poly_path}
	The usefulness of the method of moments  in   random matrix theory is based on  a simple observation that  expresses   traces of powers of an Hermitian matrix $H=(H_{uv})_{u,v\in \Lambda_{L}}$ as  a sum  over paths
	\begin{equation}
			\tr (H^n)=\sum_{\mathfrak{p}_{n}}\prod_{j=0}^{n-1}H_{u_ju_{j+1}},
	\end{equation}
	where the summation  is over all  paths  $\mathfrak{p}_n=(u_0,u_1,\cdots u_{n-1},u_n)$ with  $ u_n=u_0$.
		
Recall 		that  the Chebyshev polynomial of second kind is   defined as  		\begin{equation}U_n(x)=\frac{\sin\big((n+1)\theta\big)}{\sin\theta},\quad  x=\cos \theta,
 			\end{equation}
which is indeed  a polynomial of $x$ with    degree $n$ and  	    satisfies  the three-term   recurrence relation
		\begin{equation}
			U_{n}(x)=2xU_{n-1}(x)-U_{n-2}(x), \quad n=1,2,\ldots,
		\end{equation}
		with the initial values $U_{-1}= 0$ and $U_{0}= 1$. 
   Using the recurrence relation,   whenever  $|H_{uv}|^2$ are constants,  a further  key  observation   from  Feldheim and Sodin \cite{feldheim2010universality,sodin2010spectral}  (cutoff variance profile) and  Erd{\H{o}}s and Knowles \cite{erdHos2011quantum2} (general variance profile),   shows  that  the trace of   $U_{m}(H/2)$ admits 
  a non-backtracking path expansion. 

However,  once    the  lattice dimension  $d>1$,  we need to add a small shift of order $W^{-d}$  in the three-term recurrence relation for Chebyshev polynomials. This  will be used to eliminate the divergence  effect caused by the four moment.  See  Proposition \ref{prop:2.7} below,  which  is our primary result in this section.
 
\begin{definition} \label{a4Polynomial}
    Given a real number $a_4$,  the modified Chebyshev  polynomial  $\mathcal{P}_n(z)$ of degree $n$ is defined    by the   four-term recursion  
    \begin{equation}\label{equ:modify_Cheby}
        \mathcal{P}_{n}(z)=z\mathcal{P}_{n-1}(z)-\mathcal{P}_{n-2}(z)+a_4\mathcal{P}_{n-4}(z), \quad n=1,2,\ldots,
    \end{equation}
    with the standard initial conditions $\mathcal{P}_{0}(z)=1$ and  $\mathcal{P}_{-n}(z)=0$ for any integer $n>0$. It will always be assumed that   $a_4=\sum_{y\in \Lambda_L}(2\sigma_{xy}^4 - \mathbb{E}|H_{xy}|^4)=\sum_{y\in \Lambda_L}\sigma_{xy}^4$  in subsequent sections. 
\end{definition}
Next, we  need to develop a path expansion for the polynomial matrix  $\mathcal{P}_{n}(H)$ by  the similar  argument used  in \cite[Section 6]{erdHos2011quantum2}. For this, we  introduce four matrices $\Phi_1, \Phi_3, \Phi_5$ and $\Phi_7$  that are induced  from $H$,  
\begin{equation}\label{equ:2.12}
    (\Phi_1)_{xy}=H_{xy},(\Phi_3)_{xy}=-|H_{xy}|^2H_{xy},~(\Phi_5)_{xy}=2|H_{xy}|^4H_{xy},~(\Phi_7)_{xy}=-|H_{xy}|^6H_{xy}.
\end{equation}
For  the vector $\mathbf{b}^{(s)}=(b_1, \ldots, b_s)$ with nonnegative integers, put   $|\mathbf{b}^{(s)}|=\sum_{i= 1}^{s}b_i$.
By convention,  $\mathbb{1}(x\ne y)$ denotes the  indicator function  and  $\delta_{x,y}=\mathbb{1}(x=y)$ for any $x,y\in \Lambda_L$. 
\begin{definition} \label{non-backtracking power}
The $n$-th non-backtracking power of $H$, denoted by $V_n$, is defined  by
\begin{equation}\label{equ:def_V}
    \begin{aligned}(V_n)_{xy}:=\sum_{x_0,x_1,\ldots,x_n}\delta_{x,x_0}\delta_{x_n,y}\left[\prod_{i=0}^{n-2}\mathbb{1}(x_i\ne x_{i+2})\right] H_{x_0x_1}H_{x_1x_2}\cdots H_{x_{n-1}x_n}.
    \end{aligned}
\end{equation}
Similarly,  the almost non-backtracking power of $H$, denoted by   $\mathcal{V}_n$, is defined by  
\begin{equation} \label{anbp}
(\mathcal{V}_n)_{xy}:=\sum_{s\ge 0}\sum_{|\mathbf{b}^{(s)}|=n}\sum_{ x_0,x_1,\ldots,x_s}\delta_{x,x_0}\delta_{x_s,y}\left[\prod_{i=0}^{s-2}\mathbb{1}(x_i\ne x_{i+2})\right]\prod_{i=1}^{s} (\Phi_{b_i})_{x_{i-1}x_{i}},
\end{equation}
where  all $b_{i}\in \{1,3\}$.  In particular, $V_n=\mathcal{V}_n=0$ for any integer  $n<0$ and $V_1=\mathcal{V}_1=\mathbb{I}$.
 Moreover, when  $s=0$,  only the term $\delta_{x,y}$ remains nonvanishing in the sum \eqref{anbp} for $n=0$.
\end{definition}
Here it is worth stressing that,   once one   removes the factor  $|H_{xy}|^2$  in every   term from the matrix $\Phi_3$ on the right-hand side of  \eqref{anbp},  the path can still be  treated as a  non-backtracking path. That is the reason why the term  `almost' is used.   

By definition of  $\Phi_1$   and $\Phi_3$, if $(\Phi_1)_{x_{i-1}x_{i}}(\Phi_3)_{x_{i}x_{i+1}}$ is a backtracking path, that is, $x_{i+1}=x_{i-1}$, then a tadpole sub-diagram (diagram for short) with loop length 2 appears. To counter the effect of such diagrams, we have to add the subtract term $a_4\mathcal{P}_{n-4}$ as shown in \eqref{equ:modify_Cheby}. At this time  the term  $\Phi_5$ will occur. We also introduce the two shorthands  such as $\underline{\Phi_5\mathcal{V}_n}$ and $\underline{\Phi_7\mathcal{V}_n}$, defined by
\begin{equation}
    (\underline{\Phi_5 \mathcal{V}_n})_{xy}:=\sum_{s\ge 1}\sum_{|\mathbf{b}^{(s)}|=n,b_1=5}\sum_{\{x_i\}_{i=0}^s}\delta_{x_0,x}\delta_{x_s,y}\prod_{i=0}^{s-2}\mathbb{1}({x_i\ne x_{i+2}})\prod_{i=1}^{s}(\Phi_{b_i})_{x_{i-1}x_{i}},
\end{equation} 
and
\begin{equation}
    (\underline{\Phi_7 \mathcal{V}_n})_{xy}:=\sum_{s\ge 1}\sum_{|\mathbf{b}^{(s)}|=n,b_1=7}\sum_{\{x_i\}_{i=0}^s}\delta_{x_0,x}\delta_{x_s,y}\prod_{i=0}^{s-2}\mathbb{1}({x_i\ne x_{i+2}})\prod_{i=1}^{s}(\Phi_{b_i})_{x_{i-1}x_{i}}.
\end{equation} where the sum is over all $b_{i}\in \{1,3\}$ for $i>1$. 

\begin{lemma} \label{recursionV} For the matrix  $\mathcal{V}$  in Definition \ref{non-backtracking power}, 
    we have 
    \begin{equation}\label{equ:2.18}
        \mathcal{V}_n=H\mathcal{V}_{n-1}-\mathcal{V}_{n-2}+a_4\mathcal{V}_{n-4}-\underline{\Phi_5\mathcal{V}_{n-5}}-\underline{\Phi_7\mathcal{V}_{n-7}}\ (n\ge 3),
    \end{equation}
   where  \begin{equation}
        \mathcal{V}_0=\mathbb{I},~\mathcal{V}_1=H, ~\mathcal{V}_{2}=H^2-\mathbb{I}, ~\mathcal{V}_{m}=0 ~(\forall \,m<0). 
    \end{equation}
\end{lemma}
\begin{proof}

  Obviously, it's easy to verify the initial condition.

For the recursion, by definition of $\mathcal{V}_{n}$ in \eqref{anbp}, for $n\geq 2$ we have 
\begin{equation}\label{equ:HVn}
    \begin{aligned}
        (H\mathcal{V}_{n-1})_{xy}&=\sum_{x_1}H_{xx_1}(\mathcal{V}_{n-1})_{x_1y}\\
        &=\sum_{s\ge 0}\sum_{|\mathbf{b}^{(s)}|=n-1}\sum_{x_0, x_1, \ldots,x_{s+1}}\delta_{x,x_0}\delta_{x_{s+1},y}(\Phi_1)_{x_0x_1}\left[\prod_{i=1}^{s-1}\mathbb{1}(x_i\ne x_{i+2})\right]\prod_{i=1}^{s} (\Phi_{b_i})_{x_{i}x_{i+1}}.
    \end{aligned}
\end{equation}
Divide the above last  sum into two parts, according to   two cases (i) $x_0\ne x_2$ and (ii) $x_0=x_2$,   and we rewrite it as 
\begin{equation}
     (H\mathcal{V}_{n-1})_{xy}:=\Sigma_{x_0\ne x_2} + \Sigma_{x_0=x_2}. 
\end{equation}

In the first case (i),  we have \begin{equation}\label{equ:2.19}
        \Sigma_{x_0\ne x_2}=\sum_{s\ge 0}\sum_{|\mathbf{b}^{(s)}|=n-1} \sum_{x_0, x_1, \ldots,x_{s+1}}\delta_{x,x_0}\delta_{x_{s+1},y}(\Phi_1)_{x_0x_1}\left[\prod_{i=0}^{s-1}\mathbb{1}(x_i\ne x_{i+2})\right]\prod_{i=1}^{s} (\Phi_{b_i})_{x_{i}x_{i+1}}.
    \end{equation}
    Change  the vector $(b_1,\ldots,b_s)$ to  $(1,b_1,\ldots,b_s)$  by adding 1 as the first component and rewrite the new vector  as $(b_1,\ldots,b_s,b_{s+1})$ with $b_1=1$,  we further get 
     \begin{equation}\label{equ:2.20}
    \begin{aligned}
       \Sigma_{x_0\ne x_2}&= \sum_{s\ge 0}\sum_{|\mathbf{b}^{(s+1)}|=n,b_1=1} \sum_{x_0, x_1, \ldots,x_{s+1}} \delta_{x,x_0}\delta_{x_{s+1},y}\left[\prod_{i=0}^{s-1}\mathbb{1}(x_i\ne x_{i+2})\right]\prod_{i=1}^{s+1} (\Phi_{b_i})_{x_{i-1}x_{i}}\\
        &=\sum_{s\ge 1}\sum_{|\mathbf{b}^{(s)}|=n,b_1=1}\sum_{x_0, x_1, \ldots,x_{s}}\delta_{x,x_0}\delta_{x_s,y}\left[\prod_{i=0}^{s-2}\mathbb{1}(x_i\ne x_{i+2})\right]\prod_{i=1}^{s} (\Phi_{b_i})_{x_{i-1}x_{i}}.
    \end{aligned}
    \end{equation}
      
In the second case (ii),  we have 
    \begin{equation}\label{equ:2.21}
    \begin{aligned}
         \Sigma_{x_0=x_2}=&
         \sum_{s\ge 0}\sum_{|\mathbf{b}^{(s)}|=n-1} \sum_{x_0, x_1, \ldots,x_{s+1}} \delta_{x,x_0}\delta_{x_0,x_2}\delta_{x_{s+1},y}\\&\times  (\Phi_1)_{x_0x_1}(\Phi_{b_1})_{x_{1}x_{2}}\left[\prod_{i=1}^{s-1}\mathbb{1}(x_i\ne x_{i+2})\right]\prod_{i=2}^{s} (\Phi_{b_i})_{x_{i}x_{i+1}}.
    \end{aligned}
    \end{equation}
Take  a summation  over $x_1$  and we have  
\begin{equation}\label{sumx1}
\begin{aligned}
    \sum_{x_1}(\Phi_1)_{x_2x_1}(\Phi_{b_1})_{x_1x_2}\mathbb{1}(x_1\ne x_3)&=\sum_{x_1}(\Phi_1)_{x_2x_1}(\Phi_{b_1})_{x_1x_2}-\sum_{x_1}(\Phi_1)_{x_2x_1}(\Phi_{b_1})_{x_1x_2}\delta_{x_1,x_3}\\
    &=\mathbb{1}(s\ge 1)(\delta_{b_1,1}-a_4\delta_{b_1,3})- \mathbb{1}(s\ge 2)(\Phi_1)_{x_2x_3}(\Phi_{b_1})_{x_3x_2}.
\end{aligned}
\end{equation}
Thus, we see from  \eqref{equ:2.21} and   \eqref{anbp} that  the summation corresponding to $b_1=1$ 
 \begin{equation}\label{equ:2.23-1}
        \begin{aligned}
           &\sum_{s\ge 1}\sum_{|\mathbf{b}^{(s)}|=n-1,b_1=1}\sum_{x_2,\ldots,x_{s+1}}\delta_{x,x_2}\delta_{x_{s+1},y}\sum_{x_0,x_1}\delta_{x,x_0}(\Phi_1)_{x_0x_1}(\Phi_{b_1})_{x_{1}x_{2}}\left[\prod_{i=2}^{s-1}\mathbb{1}(x_i\ne x_{i+2})\right]\prod_{i=2}^{s} (\Phi_{b_i})_{x_{i}x_{i+1}}\\
           &= \sum_{s\ge 0}\sum_{|\mathbf{b}^{(s)}|=n-2}\sum_{x_0,x_1,\ldots,x_{s}}\delta_{x,x_0}\delta_{x_{s},y}
            \left[\prod_{i=0}^{s-2}\mathbb{1}(x_i\ne x_{i+2})\right]\prod_{i=1}^{s} (\Phi_{b_i})_{x_{i-1}x_{i}}\\
            &=\mathcal{V}_{n-2},
        \end{aligned}
    \end{equation}
while it is $\mathcal{V}_{n-4}$ corresponding to $b_1=3$.  Together  with the above sums,  we see from  \eqref{equ:2.21} and \eqref{sumx1} that 
 \begin{equation}\label{equ:2.23}
        \begin{aligned}
            \Sigma_{x_0=x_2}= &\mathcal{V}_{n-2}-a_4\mathcal{V}_{n-4}+B,
        \end{aligned}
    \end{equation}
   where 
   \begin{equation}
    \begin{aligned}
  B:=     
  -&\sum_{s\ge 2}\sum_{|\mathbf{b}^{(s)}|=n-1}\sum_{x_2,\ldots, x_{s+1}}\delta_{x,x_2}\delta_{x_{s+1},y}\\
  &\times (\Phi_1)_{x_2x_3}(\Phi_{b_1})_{x_3x_2}(\Phi_{b_2})_{x_2x_3}\left[\prod_{i=2}^{s-1}\mathbb{1}(x_i\ne x_{i+2})\right]\prod_{i=3}^{s} (\Phi_{b_i})_{x_{i}x_{i+1}},
   \end{aligned}
    \end{equation}
    $a_4=\sum_{y\in \Lambda_L}  |H_{xy}|^4$, and by convention $\prod_{i=3}^{s} (\Phi_{b_i})_{x_{i}x_{i+1}}=1$ when  $s=2$. 
    
    Recalling \eqref{equ:2.12},  noting that 
    \begin{equation}
        -\sum_{b_1,b_2\in \{1,3\}}(\Phi_1)_{x_2x_3}(\Phi_{b_1})_{x_3x_2}(\Phi_{b_2})_{x_2x_3}=(\Phi_3)_{x_2x_3}+(\Phi_5)_{x_2x_3}+(\Phi_7)_{x_2x_3},
    \end{equation}
     we thus obtain
    \begin{equation}\label{equ:2.25}
        \begin{aligned}
           B &=\sum_{s\ge 2}\sum_{b_1+b_2+1=\{3,5,7\}\atop |\mathbf{b}^{(s)}|=n-1}
\sum_{x_2,\ldots,x_{s+1}}\delta_{x,x_2}\delta_{x_{s+1},y}(\Phi_{b_1+b_2+1})_{x_2x_3}\left[\prod_{i=2}^{s-1}\mathbb{1}(x_i\ne x_{i+2})\right]\prod_{i=3}^{s} (\Phi_{b_i})_{x_{i}x_{i+1}}\\
            &=(\underline{\Phi_5 \mathcal{V}_{n-5}})_{xy}+(\underline{\Phi_7 \mathcal{V}_{n-7}})_{xy}\\
            &~~~+\sum_{s\ge 2}\sum_{|\mathbf{b}^{(s)}|=n-1,b_1=b_2=1}
\sum_{x_2,\ldots,x_{s+1}}\delta_{x,x_2}\delta_{x_{s+1},y}(\Phi_{3})_{x_2x_3}\left[\prod_{i=2}^{s-1}\mathbb{1}(x_i\ne x_{i+2})\right]\prod_{i=3}^{s} (\Phi_{b_i})_{x_{i}x_{i+1}},
        \end{aligned}
    \end{equation}
where the last sum is indeed equal to  
\begin{equation}\label{equ:2.22}
\begin{aligned}
\sum_{s\ge 1}\sum_{|\mathbf{b}^{(s)}|=n,b_1=3}
\sum_{x_0,\ldots,x_{s}}\delta_{x,x_0}\delta_{x_{s},y}\left[\prod_{i=0}^{s-2}\mathbb{1}(x_i\ne x_{i+2})\right]\prod_{i=1}^{s} (\Phi_{b_i})_{x_{i-1}x_{i}}.
\end{aligned}
\end{equation}
So we see that  $\eqref{equ:2.20}+\eqref{equ:2.22}$ gives $\mathcal{V}_n$ (note that $s=0$ term vanishing in the sum when $n\ge 1$). 

Finally, combining \eqref{equ:2.20}, \eqref{equ:2.23} and \eqref{equ:2.25}, we arrive at the desired recursion \eqref{equ:2.18}.  
    \end{proof}

Using  the similar notation  as in Definition \ref{non-backtracking power},  let 
$\mathbf{l}^{(s+1)}=(l_0,l_1, \ldots, l_{s})$ and   $\mathbf{c}^{(s)}=(c_1, \ldots, c_s)$, 
we can arrive at an expansion of $\mathcal{P}_{n}(H)$.
\begin{proposition}\label{prop:2.7}
  For  the modified Chebyshev  polynomial $\mathcal{P}_n(z)$    in Definition \ref{a4Polynomial},    we have
    \begin{equation}\label{equ:2.27}
        \mathcal{P}_{n}(H)=\sum_{s\ge 0}\sum_{|\mathbf{c}^{(s)}|+|\mathbf{l}^{(s+1)}|=n}\mathcal{V}_{l_0}\underline{\Phi_{c_1}\mathcal{V}_{l_1}}\ldots \underline{\Phi_{c_s}\mathcal{V}_{l_s}},
    \end{equation}
    where     the second sum is taken over all integers $l_i\ge 0$ and all $c_i\in\{5,7\}$.
\end{proposition}
\begin{proof}
    Introduce  a matrix sequence  $\mathcal{D}_n$ defined by  
    \begin{equation}
        \mathcal{D}_n:=\mathcal{P}_{n}(H)-\mathcal{V}_n,
    \end{equation}
   then  it's easy to see from Definition \ref{a4Polynomial} and Lemma \ref{recursionV}   that    \begin{equation}\label{equ:2.30}
        \mathcal{D}_n=H\mathcal{D}_{n-1}-\mathcal{D}_{n-2}+a_4\mathcal{D}_{n-4}+\underline{\Phi_5\mathcal{V}_{n-5}}+\underline{\Phi_7\mathcal{V}_{n-7}}\, (n\geq 4),
    \end{equation}
        as well as
    \begin{equation}
        \mathcal{D}_m =0, \quad  \forall m\leq 3.
    \end{equation}

    We will verify  the following expansion   
    \begin{equation}\label{equ:2.31}
        \mathcal{D}_n=\sum_{s\ge 1}\sum_{|\mathbf{c}^{(s)}|+|\mathbf{l}^{(s)}|=n}\mathcal{V}_{l_0}\underline{\Phi_{c_1}\mathcal{V}_{l_1}}\ldots \underline{\Phi_{c_s}\mathcal{V}_{l_s}}.
    \end{equation}
    By induction,     suppose that the above claim holds true up to $n-1$ and it suffices to   prove  it  for $n$  via the recursive relation \eqref{equ:2.30}. So we have 
    \begin{equation}\label{equ:2.32}
        \begin{aligned}
            H\mathcal{D}_{n-1}=\sum_{s\ge 1}\sum_{|\mathbf{c}^{(s)}|+|\mathbf{l}^{(s+1)}|=n-1}H\mathcal{V}_{l_0}\underline{\Phi_{c_1}\mathcal{V}_{l_1}}\ldots \underline{\Phi_{c_s}\mathcal{V}_{l_s}}.
        \end{aligned}
    \end{equation}
    
    By \eqref{equ:2.18}, rewrite $H\mathcal{V}_{l_0}$ as a five-term sum and we immediately see from \eqref{equ:2.32} that 
    \begin{equation}\label{equ:2.33}
        \begin{aligned}
            &H\mathcal{D}_{n-1}-\mathcal{D}_{n-2}+a_4\mathcal{D}_{n-4}\\
            &=\sum_{s\ge 1}\sum_{|\mathbf{c}^{(s)}|+|\mathbf{l}^{(s+1)}|=n-1}(\mathcal{V}_{l_0+1}+\underline{\Phi_5\mathcal{V}_{l_0-4}}+\underline{\Phi_7\mathcal{V}_{l_0-6}})\underline{\Phi_{c_1}\mathcal{V}_{l_1}}\ldots \underline{\Phi_{c_s}\mathcal{V}_{l_s}}\\
            &=\Big(\sum_{s\ge 1}\sum_{|\mathbf{c}^{(s)}|+|\mathbf{l}^{(s+1)}|=n,l_0\ge 1}+\sum_{s\ge 2}\sum_{|\mathbf{c}^{(s)}|+|\mathbf{l}^{(s+1)}|=n,l_0=0}\Big)\mathcal{V}_{l_0}\underline{\Phi_{c_1}\mathcal{V}_{l_1}}\ldots \underline{\Phi_{c_s}\mathcal{V}_{l_s}}.
        \end{aligned}
    \end{equation}
    Note that $\mathcal{V}_{m}=0 ~(\forall \,m<0)$, by adding two terms that  correspond to the special case of  $s=1$ and $l_0=0$ in the above second summation, we thus get   
    \begin{equation}\label{equ:2.34}
    \begin{aligned}
    &H\mathcal{D}_{n-1}-\mathcal{D}_{n-2}+a_4\mathcal{D}_{n-4}+\underline{\Phi_5\mathcal{V}_{n-5}}+\underline{\Phi_7\mathcal{V}_{n-7}}=\sum_{s\ge 1}\sum_{|\mathbf{c}^{(s)}|+|\mathbf{l}^{(s+1)}|=n}\mathcal{V}_{l_0}\underline{\Phi_{c_1}\mathcal{V}_{l_1}}\ldots \underline{\Phi_{c_s}\mathcal{V}_{l_s}}.
    \end{aligned}
    \end{equation}
Comparison with   the recurrence relation \eqref{equ:2.30} immediately implies  
    the desired  expansion   \eqref{equ:2.31}.   
    
    Thus, we have    completed  the proof  by mathematical induction.
\end{proof}

\subsection{\texorpdfstring{Removing $\Phi_3$, $\Phi_5$ and $\Phi_7$}{Removing Phi3, Phi5 and Phi7}}\label{sec:reduce_to_Vn}

For simplicity,  we always use the notation   $n=n_1+\cdots+n_k$ as sum of $k$ non-negative integers for any fixed positive integer $k$  in this subsection. A crucial observation is that the modified Chebyshev polynomials can be reduced to the non-backtracking powers of $H$ in Definition \ref{non-backtracking power}. 
\begin{theorem}\label{thm:thm2.5}
    If    $n W^{-d} \to 0
    $ as $W\to \infty$,  then  we have
    \begin{equation}
        \mathbb{E}\Big[\prod_{i=1}^k \tr\mathcal{P}_{n_i}(H)\Big]=\Big(1+O\big(\frac{n}{W^d}\big)\Big)\mathbb{E}\Big[\prod_{i=1}^k \tr V_{n_i}\Big].
    \end{equation}
\end{theorem}

We will see that the above theorem is indeed an immediate corollary  of Lemma \ref{PcalV} and Lemma \ref{lem:lemma2.7} below. To this end, we need   three key lemmas. 

The first one is a
 technical monotonic lemma for the trace powers.
\begin{lemma}\label{lem:2.7}
    Assume that all $ n_i-2j_i\ge 1$ with   non-negative integers $n_i$ and $j_i$, $i=1,\ldots,k$. If   $n W^{-d} \to 0$ as $W\to \infty$,  then     we have 
    \begin{equation}\label{equ:2.39}
        \mathbb{E}\Big[\prod_{i=1}^k \tr V_{n_i-2{j}_i}\Big]\le \Big(1+O\big(\frac{n}{W^d}\big)\Big)\mathbb{E}\Big[\prod_{i=1}^k \tr V_{n_i}\Big],
    \end{equation}
    and 
     \begin{equation} 
        \mathbb{E}\Big[\prod_{i=1}^k \tr \mathcal{V}_{n_i-2{j}_i}\Big]\le \Big(1+O\big(\frac{n}{W^d}\big)\Big)\mathbb{E}\Big[\prod_{i=1}^k \tr \mathcal{V}_{n_i}\Big].
    \end{equation}
\end{lemma}
\begin{proof}
    We only prove the case of $V$ since    $\mathcal{V}$ can be attacked in a similar way. 
    
    For each path $xx_1\cdots x_{n_i-2j_{i}-1}x$ in $\tr V_{n_i-2j_{i}}$, we can add $2j_i$ edges to   get a new non-backtracking path  $u_1u_2\cdots u_{j_{i}}xx_1\cdots x_{n_i-2j_{i}-1}xu_{j_{i}}\cdots u_1$. Now we sum over all possible $u_t$ with   $t=1,\ldots,j_{i}$. Note that each time $u_t$ can be chosen all except for  the restriction   $u_{t}=u_{t-2}$ 
    , which actually contributes a factor of size  $1-|H_{u_{t+1}u_{t+2}}|^2=1-O(W^{-d})$,
so if we sum over all possible $2j_1+\cdots +2j_k\le n$ edges, we then  have a factor of size at most $1-O(n W^{-d})$. 

Thus, this gives the desired estimate \eqref{equ:2.39}.
\end{proof}

The second  lemma helps us get rid of all $\Phi_{5}$ and $\Phi_{7}$ terms that appear in the  sum  of  \eqref{equ:2.27}.
\begin{lemma}  \label{PcalV} If  $n W^{-d} \to 0$ as $W\to \infty$, then  
    \begin{equation}
        \mathbb{E}\Big[\prod_{i=1}^k \tr\mathcal{P}_{n_i}(H)\Big]=\Big(1+O\big(\frac{n}{W^d}\big)\Big)\mathbb{E}\Big[\prod_{i=1}^k \tr\mathcal{V}_{n_i}\Big].
    \end{equation}
\end{lemma}
\begin{proof} 
    For simplicity, we only prove the result for $k=1$ since the proof for $k>1$ as a fixed finite number is the same. Just  consider the $s\ge 1$ terms in \eqref{equ:2.27}. The possible backtracking happens to the element $H_{xy}(\Phi_{c_i})_{yz}$ as the non-backtracking condition only have an  effect inside $\underline{\Phi_{c_i}\mathcal{V}_{l}}$. This might lead to a long backtracking as $$H_{xx_1}H_{x_1x_2}\cdots H_{x_my}(\Phi_{c_i})_{yx_m}H_{x_mx_{m-1}}\cdots H_{x_2x_1}H_{x_1z}.$$
    At this time, we can still extract the almost non-backtracking part. For each $\underline{\Phi_{c_i}\mathcal{V}_{l_i}}$,  we put a number $t_i\le l_i$ such that the  length of the backtracking path  is $t_i$. Now the path consists of an almost non-backtracking skeleton and several backtracking branches (the end edge is traversed for at least 6 times).

    Fix $t_i$, we sum over all possible label of vertices in the backtracking branches, for multiple edges (traversed more than 2 times, such as the end edges), we use the upper bound $|H_{xy}|^2=O(W^{-d})$. Each backtracking branch contributes a $O(W^{-2d})$ factor since   the end edges are traversed for at least 6 times and summing over non-multiple edges gives the weight that does not exceed 1.

    Now for the non-backtracking $\Phi_5$, we can use the upper bound $|H_{xy}|^2=O(W^{-d})$ to reduce it into $H_{xy}$ and keeps the non-backtracking property. The remaining almost non-backtracking skeleton is a sub-sum of some $\mathcal{V}_{n_i-2t}$ for some $t\ge 0$ and $\mathbb{E}[\tr\mathcal{V}_{n_i-2t}]$ can be  dominated by $\mathbb{E}[\tr\mathcal{V}_{n_i}]$ by Lemma \ref{lem:2.7}.

    Now for the path in \eqref{equ:2.27}, each $\Phi_{c_i}$ gives an $O(W^{-2d})$ factor. Summing  over $t_i$ and $l_i$  gives an $O(n^{2s})$ factor. Hence, by using  the  upper estimates except for the trivial $s=0$ term on the right-hand side of  \eqref{equ:2.27},   we arrive at 
    \begin{equation}
    \begin{aligned}
        \mathbb{E}[\tr\mathcal{P}_n(H)]
        =\mathbb{E}[\tr \mathcal{V}_n]+\sum_{s\ge 1}O\Big(\frac{n^{2s}}{W^{2sd}}\mathbb{E}[\tr \mathcal{V}_n]\Big)=\Big(1+O\big(\frac{n^2}{W^{2d}}\big)\Big)\mathbb{E}[\tr \mathcal{V}_n].
        \end{aligned}
    \end{equation}

    This completes the proof.
\end{proof}
The third   lemma helps us get rid of all $\Phi_3$ terms.
\begin{lemma}\label{lem:lemma2.7} 
   If  $n W^{-d} \to 0$ as $W\to \infty$, then     
    \begin{equation}
        \mathbb{E}\Big[\prod_{i=1}^k \tr \mathcal{V}_{n_i}\Big]=\Big(1+O\big(\frac{n}{W^d}\big)\Big)\mathbb{E}\Big[\prod_{i=1}^k \tr V_{n_i}\Big].
    \end{equation}
\end{lemma}
\begin{proof}
    For simplicity, we only verify  the case of $k=1$. Recalling the definition of  $\mathcal{V}_n$ in 
    \eqref{anbp}, let $l$ be the number of the  subscript  $i$ such that $b_i=3$ in $\mathcal{V}_n$ and label them by  $i_1,\ldots i_l$, then $s=n-2l$. Note that  $V_n$ exactly corresponds to the case of $l=0$, we  get 
    \begin{equation}
    \begin{aligned}
        &\left|\mathbb{E}[\tr\mathcal{V}_n]-\mathbb{E}[\tr V_n]\right|=\sum_{l\ge 1}\sum_{x_0,x_1,\ldots,x_{n-2l}=x_0}\sum_{0\le i_1<\cdots <i_l\le n-2l-1}\Big[\prod_{i=0}^{s-2}\mathbb{1}(x_i\ne x_{i+2})\Big]\\
        &\times  \left|\mathbb{E}\left[ H_{x_0x_1}H_{x_1x_2}\cdots (\Phi_3)_{x_{i_1}x_{i_1+1}}\cdots H_{x_{j}x_{j+1}}\cdots (\Phi_3)_{x_{i_2}x_{i_2+1}}\cdots H_{x_{n-2l-1}x_{n-2l}}\right]\right|.
        \end{aligned}
    \end{equation}
    
    Applying   the simple fact of  $|H_{xy}|^2=O(W^{-d})$  to the $\Phi_3$ terms shows  that there exists  some $C>0$    such that 
    \begin{equation}\begin{aligned}
        \left|\mathbb{E}[\tr\mathcal{V}_n]-\mathbb{E}[\tr V_n]\right|&\le \sum_{l\ge 1}\sum_{x_0,x_1,\ldots,x_{n-2l}=x_0}\sum_{0\le i_1<\cdots <i_l\le n-2l-1}\Big[\prod_{i=0}^{s-2}\mathbb{1}(x_i\ne x_{i+2})\Big] C^lW^{-ld}\\
        &\times  \mathbb{E}\Big[ H_{x_0x_1}H_{x_1x_2}\cdots H_{x_{i_1}x_{i_1+1}}\cdots H_{x_{j}x_{j+1}}\cdots H_{x_{i_2}x_{i_2+1}}\cdots H_{x_{n-2l-1}x_{n-2l}}\Big]\\
        &= \sum_{l\ge 1}\sum_{0\le i_1<\ldots <i_l\le n-2l-1}C^lW^{-ld}\sum_{x_0,x_1,x_2,\ldots,x_s=x_0}\Big[\prod_{i=0}^{s-2}\mathbb{1}(x_i\ne x_{i+2})\Big]\\
        &\times \mathbb{E}\big[ H_{x_0x_1}H_{x_1x_2}\cdots H_{x_{i_1}x_{i_1+1}}\cdots H_{x_{j}x_{j+1}}\cdots H_{x_{i_2}x_{i_2+1}}\cdots H_{x_{s-1}x_s}\big]\\
        &=\sum_{l\ge 1}\sum_{0\le i_1<\cdots <i_l\le n-2l-1}C^lW^{-ld} \mathbb{E}\left[\tr V_{n-2l} \right]\\
        &\le \sum_{l\ge 1}\frac{1}{l!}\big(CnW^{-d}\big)^l\mathbb{E}\left[\tr V_{n-2l} \right]=O\big(nW^{-d}\big)\mathbb{E}\left[\tr V_{n} \right],
        \end{aligned}
    \end{equation} where Lemma \ref{lem:2.7} has been  used in the last equality.
    
    Thus, this completes the proof in  the case of $k=1$.  
\end{proof}
\subsection{Reduction to diagrams}\label{sec:diagram_reduction}
In this subsection we  are devoted to classifying  non-backtracking paths into different diagrams in the sense of topological equivalence,  as  done in \cite{feldheim2010universality,sodin2010spectral}. 

Recall  the unimodular random band matrix $H$  in Definition   \ref{defmodel}  and the  $n$-th non-backtracking power $V_n$ associated with  $H$  in Definition \ref{non-backtracking power}, 
 for any fixed integer $k\geq 1$,  we immediately get 

 \begin{proposition} \label{nbformula}

\begin{equation}
		\mathbb{E}\!\big[(\tr V_{n_1} \tr V_{n_2}\cdots \tr V_{n_k})\big]=\Sigma_{\beta}^{(1)}(n_1,n_2,\cdots,n_k),
        \end{equation}
	where $\Sigma_{\beta}^{(1)}$	equals to a weighted sum of $k$-tuples of paths ($k$-paths for short)
			\begin{equation}\label{equ:2.45}
			\mathfrak{p}:=	\mathfrak{p}_{n_1,\cdots,n_k}=u_0^1u_1^1\cdots u_{n_1}^1,~ u_0^2u_1^2\cdots u_{n_2}^2,~\cdots, u_0^ku_1^k\cdots u_{n_k}^k, 
			\end{equation}
such that 	
			\begin{itemize}
\item[(a)] {\bf Non-backtrackings:} $u^i_{j+2}\ne u^i_{j}$, $1\le i\le k$, $ 0\le j\le n_i-2$;
			
			\item[(b)] {\bf Closed paths:} $u_{n_i}^i=u^i_0$, $1\le i\le k$;
			
			\item[(c$_\beta$)] {\bf Even multiplicity:} for any $u\in \Lambda_L$,
             \begin{equation}	\#\{(i,j) \mid u^i_j=u,u_{j+1}^i=u\} \equiv 0 \mod 2,
            \end{equation}
            and  for any $u\ne v \in \Lambda_L$,   
			 \begin{equation}\begin{cases}
			    \#\{(i,j)\mid u^i_j=u,u_{j+1}^i=v\}+\#\{(i,j)\mid u^i_j=v,u_{j+1}^i=u\} \equiv 0\mod 2, &\beta=1;\\
                \#\{(i,j) \mid u^i_j=u,u_{j+1}^i=v\}= \#\{(i,j)\mid u^i_j=v,u_{j+1}^i=u\}, & \beta=2.
			\end{cases} \end{equation}
\end{itemize}
In this case  
\begin{equation}
    \Sigma_{\beta}^{(1)}=\sum_{ \mathfrak{p}:\,(a)-(c_\beta)}
    \prod_{(u,v) \in \mathfrak{p}
    }\sigma_{uv}^{\#\{(i,j)\mid u^i_j=u,u_{j+1}^i=v\}},
\end{equation}
where the sum ranges over all $k$-paths and the product  over all distinct    directed pairs of points on the $k$-path.    
   \end{proposition}

 Given a path $\mathfrak{p}$ in \eqref{equ:2.45} satisfying (a), (b) and (c$_\beta$), we can define   a natural  directed  multigraph
$G = G(\mathfrak{p}) = 
(V, E_\text{dir})$,  where $V\subset \Lambda_L$ is the set of all distinct vertices $u^i_{j}$ on $\mathfrak{p}$ and $E_\text{dir}$ is the set of edges $(u^i_{j},u^i_{j+1})$ with multiplicities.  It is important to emphasize  that the directed  multigraph   allows multiple edges (parallel edges) between the same pair of vertices, 
as well as loops (edges that join a vertex to itself).
To address  the effect of multiple edges in $\Sigma_{\beta}^{(1)}$--for instance, an edge traversed no fewer than four times--we need to couple the $k$-path $\mathfrak{p}$ in \eqref{equ:2.45} with a matching  $\mathfrak{m}$.
This ensures that for multiple edges, the traversals are partitioned into pairs in a specified manner, as previously described in \cite[II.1\& II.3]{feldheim2010universality} as follows.
\begin{itemize}
\item For $\beta=1$, every edge $(u,v)$ can    be matched either to a coincident edge $(u,v)$ or  reverse edge $(v,u)$. In this case, if the edge is traversed   $2m$ times, there are $(2m-1)!!$ pairwise matches . 

\item For $\beta=2$, every edge $(u,v)$ can only be matched to its reverse edge   $(v,u)$. In this case, if the edge is traversed $2m$ times, there are $m!$ possible matches .  
\end{itemize}

However, for each path $\mathfrak{p}_{n_1,\cdots,n_k}$  in \eqref{equ:2.45}, we just  assign one match $\mathfrak{m}$  since we will see that  the specific choice does not matter!  So if we rewrite $\mathfrak{p}_{n_1,\cdots,n_k}$   as $(\mathfrak{p},\mathfrak{m})$, then by Proposition \eqref{nbformula} we obtain 
\begin{equation} \label{matchsum}
    \Sigma_{\beta}^{(1)}=\sum_{ (\mathfrak{p},\mathfrak{m}): \,(a)-(c_\beta)}
    \prod_{(u,v)\in \mathfrak{p}}\sigma_{uv}^{\#\{(i,j)\mid u^i_j=u,u_{j+1}^i=v\}},
\end{equation}
where  the sum ranges  over all matched  $\mathfrak{p}$   satisfying (a)-(c$_\beta$) 
in Proposition \eqref{nbformula}, and   the product is  over all distinct    directed pairs. Next, an incredible  step  due to   Feldheim and Sodin \cite{feldheim2010universality} is to construct  a mapping from the collection of matched $k$-paths  to weighted diagrams  below.  The following definition is  a slight modification   of
\cite[{Definition II.3.1}]{feldheim2010universality}, by 
 removing  the degree $3$ restriction    and changing  the range of the weight function.

\begin{definition}\label{def:diagram}Let $\beta\in \{1,2\}$.
\begin{itemize}
\item A \textbf{\textit{{$k$-diagram}}} of type $\beta$ is an (undirected) multi-graph $\bar{G} = (\bar{V}, \bar{E})$,
together with a $k$-tuple of circuits 
\begin{equation}\label{equ:2.52} \bar{\mathfrak{p}} = \bar{u}_0^1 \bar{u}_1^1 \cdots  \bar{u}_0^1 ,\,\,\,
             \bar{u}_0^2 \bar{u}_1^2 \cdots   \bar{u}_0^2 ,\,\,\,
             \cdots ,\,\,\, \bar{u}_0^k \bar{u}_1^k \cdots   \bar{u}_0^k \end{equation}
on $\bar{G}$, such that
\begin{itemize}
\item $\bar{\mathfrak{p}}$  is  {\em non-backtracking}, that is, in every circuit no edge is
followed by its reverse unless $\beta=1$ and the edge is $\bar{u}\bar{u}$;
\item For every loop $
e_0= (\bar{u}, \bar{u}) \in \bar{E}$,
\begin{equation} 
\# \left\{ (i,j) \, | \, (\bar{u}_j^i, \bar{u}_{j+1}^i) =e_0 \right\}  = 2,
     \end{equation}
and for every $
e= (\bar{u}, \bar{v}) \in \bar{E}$ with $\bar{u}\ne \bar{v}$,
\begin{equation}\begin{split}
&\# \left\{ (i,j) \, | \, (\bar{u}_j^i, \bar{u}_{j+1}^i) =e \right\} + \# \left\{ (i,j) \, | \, (\bar{u}_{j+1}^i, \bar{u}_{j}^i) = e \right\} = 2
    \quad (\beta = 1)~, \\
&\# \left\{ (i,j) \, | \, (\bar{u}_j^i,\bar{u}_{j+1}^i) = e \right\} = \# \left\{(i,j) \, | \, (\bar{u}_{j+1}^i, \bar{u}_{j}^i) = e \right\} = 1
    \quad (\beta = 2)~;
\end{split}\end{equation}
\item For $i=1,\ldots, k$, the degree of $u_0^i$ in $\bar{G}$ is always 1 and the degree of all other vertices is greater than 1. Each  $u_0^i$ is called as  a \textit{\textbf{marked point}} while the edge with endpoint $u_0^i$ is a \textbf{\textit{tail edge}}.
\end{itemize}
\item A \textbf{\textit{{weighted $k$-diagram}}} is a $k$-diagram $\bar{G}$    assigned with a weight function
$\bar{w}: \bar{E} \to \{0,1,2,\cdots\}$,  in which  $\bar{w}(e)=0$ is  possible  only when $e \in \bar{E}$ is a tail edge.
\item    Denote  by  $\mathfrak{D}=(\bar{G},\bar{\mathfrak{p}})$   as a $k$-diagram and  by the pair $(\mathfrak{D},w)$  as a weighted $k$-diagram, for short.
\end{itemize}
\end{definition}

Now we describe the construction of the mapping   from matched paths to weighted diagrams, due to Feidheim and Sodin \cite{feldheim2010universality}. 
 
\begin{definition}[{\bf Feidheim-Sodin \textit{contraction}}]\label{operation:diagram}
A mapping 
\begin{equation}
    {\varphi}:\text{\big\{matched $k$-paths\big\}} \longrightarrow \text{\big\{weighted $k$-diagrams\big\}},  \quad(\mathfrak{p}, \mathfrak{m}) \mapsto (\mathfrak{D},w),
\end{equation}
from the collection of matched paths  satisfying (a)-(c$_\beta$) into the collection of weighted diagrams (of type $\beta$) is constructed as follows.

\vspace{2mm} \noindent {\bf\em (i)} Start with the directed  multigraph
$G = G(\mathfrak{p}) = (V, E_\text{dir})$ associated with   the path $\mathfrak{p}$, where 
\[ V = \{ u\in \Lambda_L \, \mid \, \exists (i,j), \,u_{j}^i = u \}~, \,\,
   E_\text{dir} = \{ (u^{i}_{j}, u^{i}_{j+1}) \}~, \]
and merge  each pair of matched edges into a single undirected edge.

\vspace{1mm} \noindent {\bf\em (ii)} If the degree of $u_{0}^i$ is greater than 1, add a new vertex $r^i$
connected to $u_{0}^i$, and replace the corresponding circuit of $\mathfrak{p}$ with $r^i u_{0}^i u_{1}^i \cdots u_{0}^i r^i$. 

\vspace{1mm} \noindent {\bf\em (iii)} Erase all the vertices of degree 2.

\vspace{1mm} \noindent {\bf\em (iv)} Set
\[ w(\bar{e}) = \begin{cases}
     1+\# \{\text{erased vertices on} \ \bar{e}\},\\
    0, \quad \text{if $\bar{e}$ was created at  {\bf Step \em(ii)} }.
 \end{cases}\]
 Here $w(\bar{e})$ counts the number of edges in the preimage   of the edge $\bar{e}$ under the mapping $\varphi$. The  abbreviated symbol   $\varphi[(\mathfrak{p},\mathfrak{m})]= \mathfrak{D}$ will be used  by forgetting the weight ${w}$.
\end{definition}

\begin{remark}\label{remark:2.14}
    It is worth stressing    that  the immediate self-loop $\bar{e}=\bar{u}\bar{u}$ is possible in the diagram in the $\beta=1$ case of   Definition \ref{def:diagram}. However, since $\bar{e}$ is non-backtracking, we must have ${w}(\bar{e})\ge 3$ or $w(\bar{u})=1$. If $w(\bar{u})=1$, the degree of the corresponding vertex must be greater than 3 since by non-backtracking property the edge $\bar{u}$ can not be traversed twice immediately. The self-loop is of vital  importance as it is the first singular diagram occurring in high-dimensional random band matrices and is known as a tadpole diagram related to   Feynman diagram integrals in Quantum Field Theory.
\end{remark}

We need to introduce  a concept  of the typical $k$-diagrams,  which is indeed the so-called $k$-diagram  in \cite[Definition II.3.1]{feldheim2010universality}.
\begin{definition}\label{def:2.15}
    \begin{itemize}
        \item  A \textbf{\textit{typical (weighted) $k$-diagram}} $\bar{G} = (\bar{V}, \bar{E})$ is a $k$-diagram  in which   the degree of marked points in $\bar{G}$ is 1 and the degree of all the other vertices is equal to $3$. Let  $\mathscr{G}_{\beta,k,3}$ be the set of typical (weighted) $k$-diagrams  of type $\beta$, and  let $\mathscr{G}_{\beta,k,\ge 3}$ be the set of general (weighted) $k$-diagrams including typical and non-typical  diagrams. 
        \item A diagram $\mathfrak{D}=(\bar{G},\bar{\mathfrak{p}})$ is \textbf{\textit{connected}} if the  graph $\bar{G}$ is connected.  Let   $\mathscr{G}^c_{\beta,k,3}$ be the set of connected typical (weighted) $k$-diagrams of type $\beta$ and   $\mathscr{G}^c_{\beta,k,\ge 3}$ be  the set of connected general (weighted) $k$-diagrams.
    \end{itemize}
    
\end{definition}

Some    important properties of typical weighted $k$-diagrams  are listed as follows.  
\begin{lemma}
Let $\mathfrak{D}=(\bar{G},\bar{\mathfrak{p}})$ be a typical weighted $k$-diagram, then  
    \begin{itemize}
        \item[(i)] $\bar{w}(\bar{e})\ge 0$  for any  tail edge $\bar{e}$, while  $\bar{w}(\bar{e})\ge 1$  for any   other edge $\bar{e}$;
        \item[(ii)] $\bar{w}(\bar{e})\ge 3$ for any  self-loop edge $\bar{e}=\bar{u}\bar{u}$ in the $\beta=1$ case;
        \item[(iii)] For $\bar{G}=(\bar{V}, \bar{E})$,  there exists an   integer $s\ge k$ such that $|\bar{E}|=2s$ and $|\bar{V}|=3s-k$. 
    \end{itemize}
\end{lemma}
\begin{proof}
    Obviously, both (i) and (ii) hold true. See  \cite[Claim II.2.1 \& Claim II.2.2]{feldheim2010universality} for   (iii).
\end{proof}

Since the length of the $i$-th path $u_0^iu_1^i\cdots u_{n_i-1}^i u_0^i$ in \eqref{equ:2.45} is exactly $n_i$, 
the weight $w$ must satisfy a system of linear equations. For practical use, we only definite it in typical diagrams.
\begin{definition}[Linear $\mathfrak{D}$-system] \label{D-system}
  A family of non-negative integers $ \{ w({{e}}) \}_{{e}\in E}$
  is said to  satisfy a system of linear equations  associated with the typical $k$-diagram $\mathfrak{D}$  if   
		\begin{equation}  
		\mathfrak{C}(\mathfrak{D}): \quad \sum_{{e}\in {E}}c_i({e})w({{e}})=n_i,\quad i=1,...,k,  \label{EqC}
		\end{equation}
		where  $c_i(e)\in \{0,1,2\}$ is the number of times that ${e}$ occurs in the $i$-th circuit of ${\mathfrak{p}}$ and 
\begin{equation}\label{equ:2.56}
    w(e)\ge\begin{cases}
       3,~~~~\text{if $e$ is a self-loop};\\
        1,~~~~\text{if $e$ is not a tail edge}.
   \end{cases} 
\end{equation} 
Similarly, a system of   inequalities, by removing the set $T$ of  all tail edges  
\begin{equation} \label{CDtildesystem}
   \widetilde{\mathfrak{C}}(\mathfrak{D}): \quad \begin{cases}
        &\sum_{{e}\in {E}\backslash T}c_i({e})w({{e}})\le n_i,\qquad i=1,...,k,\\
        &\sum_{{e}\in {E}\backslash T}c_i({e})w({{e}})\equiv n_i\pmod 2,\qquad i=1,...,k,\\
        &w(e)\ge 3,~~~~\text{if $e\in E\backslash T$ is a self-loop},\\
        &w(e)\ge 1,~~~~\text{if $e\in E\backslash T$}.
    \end{cases}
\end{equation}
\end{definition} 
		\begin{proposition}[{\cite[Proposition II.3.3]{feldheim2010universality}}] \label{diagrambound}
For $\beta\in \{1,2\}$, let $D_{\beta,k}(s)$ be the number of typical $k$-diagrams with $|\bar{V}|=2s$ and $|\bar{E}|=3s-k$, then there is  a universal constant   $C\ge 1$ such that 
			\begin{equation}	\frac{1}{(k-1)!} 
\Big(\frac{s}{C}\Big)^{s+k-1}\le D_{\beta,k}(s)\le  \frac{1}{(k-1)!} \big(C\,s\big)^{s+k-1}.
			\end{equation}
		\end{proposition} 

\subsection{Counting weighted diagrams}\label{sec:count_typical}

Take  the Feidheim-Sodin \textit{contraction} in Definition \ref{operation:diagram}, and we get  the weighted $k$-diagram $\mathfrak{D}=(\bar{G},\bar{\mathfrak{p}})$ with  $\bar{G}=(\bar{V}, \bar{E})$.   However,  to determine the index of vertices in $\mathfrak{D}$ we need to  embed the vertex $\bar{V}$ into the lattice $\Lambda_L$. Introduce  a mapping   $\iota:\bar{V}\rightarrow \Lambda_L$ and 
denote the preimage of the edge $\bar{e}\in \bar{E}$ by $u^{\bar {e}}_0\ldots u^{\bar{e}}_{w(\bar{e})}$,  then  we derive from the weighted sum \eqref{matchsum} that    
\begin{equation}\label{equ:Sigma1}
    \begin{aligned}   
    \Sigma_{\beta}^{(1)}&=\sum_{\mathfrak{D}\in \mathscr{G}_{\beta,k,\ge 3}}\sum_{w:(\mathfrak{D},w)}
\sum_{\mathfrak{p}:\varphi[(\mathfrak{p},\mathfrak{m})]=(\mathfrak{D},w), (a)-(c_\beta)} \prod_{(u,v)\in \mathfrak{p}}\sigma_{uv}^{\#\{(i,j)\mid u^i_j=u,u_{j+1}^i=v\}}\\
&=\sum_{\mathfrak{D}\in \mathscr{G}_{\beta,k,\ge 3}}\sum_{w:(\mathfrak{D},w)}\Sigma_{\beta}^{(1)}(\mathfrak{D},w),
\end{aligned}
\end{equation}
where 
\begin{equation} \label{Dw1}
    \Sigma_{\beta}^{(1)}(\mathfrak{D},w):=\sum_{\iota:\bar{V}\rightarrow \Lambda_L} \sum_{\mathfrak{p}:\varphi[(\mathfrak{p},\mathfrak{m})]=(\mathfrak{D},w), (a)-(c_\beta)}
    \prod_{\bar{e}\in \bar{E}}\Big(\prod_{i=0}^{w(\bar{e})-1}\sigma^2_{u^{\bar{e}}_{i}u^{\bar{e}}_{i+1}}\Big).
\end{equation}

On one hand, let $p_m(u,v)$ be $m$-step transition probability associated with the doubly stochastic matrix $\sigma^2_{xy}$ in \eqref{VP}, introduce 
\begin{equation}  \label{sum2}
\Sigma_{\beta}^{(2)}(\mathfrak{D},w):=
\sum_{\iota:\bar{V}\rightarrow \Lambda_L} \prod_{\bar{e}\in \bar{E}}p_{w(\bar{e})}(\iota(u^{\bar{e}}),\iota(v^{\bar{e}})),
\end{equation}
in which  $u^{\bar{e}}$ and $ v^{\bar{e}}$ are  the endpoints of $\bar{e}$ in the diagram, and  we have indeed summed  over all other possible labeling   indexes  except those   in $\bar{V}$.
On the other hand, considering  the restriction on $\mathfrak{p}$   that every edge is traversed exactly twice,    this  path forms a self-avoiding patten of shape $\mathfrak{D}$, which means that each path corresponding to $e\in \bar{E}$ is self-avoiding and for $e_1\ne e_2\in\bar{E}$, the paths corresponding to $e_1$ and $e_2$ are not intersected. 
 Also introduce the strongly restricted sum  
\begin{equation}  
\Sigma_{\beta}^{(0)}(\mathfrak{D},w):= \sum_{\iota:\bar{V}\rightarrow \Lambda_L}
\sum_{\text{\tiny self-avoiding}\, \mathfrak{p}: \iota, w}
\ \prod_{\bar{e}\in \bar{E}}\Big(\prod_{i=0}^{w(\bar{e})-1}\sigma^2_{u^{\bar{e}}_{i}u^{\bar{e}}_{i+1}}\Big).
\end{equation}

With the preparations above, we can now proceed to prove the upper and lower bounds  for $\Sigma_{\beta}^{(1)}(\mathfrak{D},w)$.

\begin{lemma}\label{lem:2.19} 
    \begin{equation}
        \Sigma_{\beta}^{(0)}(\mathfrak{D},w)\le \Sigma_{\beta}^{(1)}(\mathfrak{D},w)\le \Sigma_{\beta}^{(2)}(\mathfrak{D},w).
    \end{equation}
\end{lemma}
\begin{proof} The lower bound is obvious.    For the upper bound, remove some restrictions in the sum of \eqref{Dw1} and  we have from \eqref{sum2}  that 
\begin{equation}
\begin{aligned}
    \Sigma_{\beta}^{(1)}(\mathfrak{D},w)
&\le \sum_{\iota:\bar{V}\rightarrow \Lambda_L}\sum_{\mathfrak{p}:  \iota, w }\prod_{\bar{e}\in \bar{E}}\Big(\prod_{i=0}^{w(\bar{e})-1}\sigma^2_{u^{\bar{e}}_{i}u^{\bar{e}}_{i+1}}\Big)=\Sigma_{\beta}^{(2)}(\mathfrak{D},w).
\end{aligned}
\end{equation} 
Here the second sum means that the index of corresponding vertices in $\mathfrak{p}$ is determined by $\iota$ and the length of steps between them is $w(\bar{e})$. 
\end{proof}

The following crucial conclusion  shows that the lower and upper bounds are asymptotically the same for any $k$-diagram in Definition \ref{def:2.15}, based on properties of random walks on the torus.
\begin{proposition}\label{coro:3.7} 
For any diagram $(\mathfrak{D},w)\in \mathscr{G}_{\beta,k,\ge 3}$ with   $\mathfrak{D}=(\bar{G},\bar{\mathfrak{p}})$ and  $\bar{G}=(\bar{E},\bar{V})$,  let  $n=\sum_{i=1}^k n_i$ and 
\begin{equation} \label{ENNotation}
    \mathcal{E}(n)=\frac{n^2}{N}+\begin{cases}
		\frac{n^{\frac{3}{2}}}{W}, & d=1;\\
		\frac{n\log n}{W^2}, & d=2;\\
		\frac{n}{W^d}, & d>2.
	\end{cases}
\end{equation}
If $\mathcal{E}(n)\to 0$ as $W\to \infty$,  then we have
\begin{equation}
        \Sigma_{\beta}^{(0)}(\mathfrak{D},w)=
        \left(1+O\big(|\bar{E}|\cdot \mathcal{E}(n)\big)\right)\, \Sigma_{\beta}^{(2)}(\mathfrak{D},w).
\end{equation}
\end{proposition}

\begin{proof}
Firstly,  we add $w(\bar{e})-1$ vertices on each edge $\bar{e}\in \bar{E}$ whenever $w(\bar{e})>1$ and denote by $\Tilde{G}=(\Tilde{E},\Tilde{V})$ the resulting new graph. Let the number of intersections  
\begin{equation}
   R=\sum_{u\ne v: (u,v)\in \Tilde{E}}  1_{\{\iota(u)=\iota(v) \}},
\end{equation}
we will  complete the proof by estimating the expected number of intersections.  

Secondly, we can prove    the upper bound for the number of intersections under the condition 
\begin{equation}\label{equ:2.64}
    \mathbb{E}[R|\iota(\bar{V})] = O(|\bar{E}|\cdot \mathcal{E}(n))
\end{equation}
as follows. 
Choose $\bar{e}_1\neq \bar{e}_2\in \bar{E}$,  by Lemma \ref{proposition:Gaussian_asymptotic_independence} we see that  the expected number of intersections between $\bar{e}_1, \bar{e}_2$ is bounded by $O(\mathcal{E}(w(\bar{e}_1)+w(\bar{e}_2)))$. 
Summing over all $(\bar{e}_1,\bar{e}_2)$, by subadditivity   of $\mathcal{E}(n)$ we  obtain  
    \begin{equation}
        \sum_{(\bar{e}_1,\bar{e}_2)}\mathcal{E}(w(\bar{e}_1)+w(\bar{e}_2))\le 2\sum_{(\bar{e}_1,\bar{e}_2)}(\mathcal{E}(w(\bar{e_1}))+\mathcal{E}(w(\bar{e_2})))\le 4|\bar{E}|\sum_{\bar{e}}\mathcal{E}(w(\bar{e}))\le 4|\bar{E}|\cdot \mathcal{E}(n).
    \end{equation}
Similarly,     for  the self-intersection of $e\in \bar{E}$, by Lemma \ref{proposition:Gaussian_self_avoiding}    the self-intersection number of each edge is also bounded by $O(\mathcal{E}(w(\bar{e}))$. Summing them up also gives rise to   \eqref{equ:2.64}. 

    Now since    $1-R\le 0$ whenever   there is any intersection, we have
    \begin{equation}
        \left(1+O\big(|\bar{E}|\cdot \mathcal{E}(n)\big)\right)\, \Sigma_{\beta}^{(2)}(\mathfrak{D},w)=\sum_{\iota(\bar{V})}\mathbb{E}[1-R|\iota(\bar{V})]\le \Sigma_{\beta}^{(0)}(\mathfrak{D},w)\le \sum_{\iota(\bar{V})}\mathbb{E}[1|\iota(\bar{V})]=\Sigma_{\beta}^{(2)}(\mathfrak{D},w),
    \end{equation}
    from which the desired result follows.
\end{proof}

In order to control the contribution from the non-typical diagrams, we need to introduce  the concept of derived diagram and partial order on diagrams.   
\begin{definition}[Partial order]
    A  weighted $k$-diagram  $(\mathfrak{D}',w')$ is said to be a derived diagram of $(\mathfrak{D},w)$,
   denoted by 
    $(\mathfrak{D}',w')\prec (\mathfrak{D},w)$, if it can be  constructed  by identifying two points $u_1,u_2$ on the edges and balancing  the weights  through one of  
    three  possible manipulations:
    \begin{itemize}
        \item[(i)] Both  $u_1$ and $u_2$ are  interior points   of the  (possibly same) edges $e_1$ and $e_2$, and in this case  $u_i$ splits the edge into two (or three) pieces such that  the  sum of their weights   equals $w(\bar{e})$;
        \item[(ii)] $u_1\in \bar{V}$ and $u_2$ is  an interior point      of   $e\in \bar{E}$ such that   $e$  is split  into two edges $e_1,e_2\in\bar{E}'$ by $u_2$  and   $w'(e_1)+w'(e_2)=w(e)$. 
        \item[(iii)] $u_1,u_2\in \bar{V}$, and in this case $w(e)=w'(e)$ for all relevant  edges.
    \end{itemize}    
\end{definition}

Under the partial order relation, the typical weighted $k$-diagrams in Definition \ref{def:2.15} lie in the set of maximal elements.

\begin{lemma}   Let $\mathcal{M}_{\mathrm{diagram}}$ be the set of maximal $k$-diagrams   under   the partial order $\prec$, we have
    \begin{equation}
        \mathscr{G}_{\beta,k,3}\subset \mathcal{M}_{\mathrm{diagram}}.
    \end{equation}
\end{lemma}
\begin{proof} By reduction,  suppose that some 
       $\mathfrak{D}\in \mathscr{G}_{\beta,k,3}$ is   a derived digram.  That is, it can be glued from another $k$-diagram $\Tilde{\mathfrak{D}}$, at least  with a corresponding vertex  $v$  glued from $u_1$ and $u_2$. Here  $u_i$ could be an endpoint or  also might lie in the interior of  some edge. 
       Then   $\text{deg}(v)=\text{deg}(u_1)+\text{deg}(u_2)\ge 1+1$, implying   $\text{deg}(v)=3$.  We further know that one   of  $\text{deg}(u_1)$ and $\text{deg}(u_2)$ is 1 and the other is 2. Without loss of generality, let  $\text{deg}(u_1)=1$, then   $u_1$ is a marked vertex.  In this case the number of degree 1 decreases to $k-1$,   so  $\mathfrak{D}$ can not   be a  $k$-diagram. That is a contradiction. 

    Thus the desired result follows.
  \end{proof}

With the partial order relation, we immediately have a better bound from Lemma \ref{lem:2.19}.
\begin{corollary}\label{coro:2.19} For any diagram $\mathfrak{D}$, we have
    \begin{equation} \label{finesum}
        \Sigma_{\beta}^{(0)}(\mathfrak{D},w)\le \sum_{(\mathfrak{D}',w'):(\mathfrak{D}',w')\prec (\mathfrak{D},w)}\Sigma_{\beta}^{(1)}(\mathfrak{D}',w')\le \Sigma_{\beta}^{(2)}(\mathfrak{D},w).
    \end{equation}
    Furthermore, if $nW^{-d}\rightarrow 0$ as $W\rightarrow \infty$, we have
    \begin{multline}\label{equ:maximal_sum}
        \Big(1+O\big(\frac{n}{W^d}\big)\Big)\sum_{\mathfrak{D}\in \mathcal{M}_{\text{diagram}}}\sum_{w(e):(\mathfrak{D},w)}\Sigma_{\beta}^{(0)}(\mathfrak{D},w)\le \mathbb{E}\Big[\prod_{i=1}^k \tr\mathcal{P}_{n_i}(H)\Big]\\ \le \Big(1+O\big(\frac{n}{W^d}\big)\Big)\sum_{\mathfrak{D}\in \mathcal{M}_{\text{diagram}}}\sum_{w(e):(\mathfrak{D},w)}\Sigma_{\beta}^{(2)}(\mathfrak{D},w).
    \end{multline}
\end{corollary}
\begin{proof}
    We first consider \eqref{finesum}. The lower bound is trivial, so we just focus on   the upper bound. 
    
    To this end,   note that $\Sigma_{\beta}^{(2)}(\mathfrak{D},w)$ contains all paths  satisfying $\iota$ and $w$, while for any path $\mathfrak{p}$ in $\Sigma_{\beta}^{(1)}(\mathfrak{D}',w')$ more restrictions--for instance,  certain vertices have same labeling-- are imposed. Also, by the definition of the  Feidheim-Sodin \textit{contraction} ${\varphi}$  in Definition \ref{operation:diagram}, each path is counted once. Hence, the summation  in the middle of \eqref{finesum}     is just taken over partial terms from those of the summation for  $\Sigma_{\beta}^{(2)}(\mathfrak{D},w)$. This completes proof of \eqref{finesum}.

      Next,  we consider \eqref{equ:maximal_sum}. 
      By Theorem \ref{thm:thm2.5}, we have
         \begin{equation}\label{equ:lemma2.21_equ:2.67}
        \mathbb{E}\Big[\prod_{i=1}^k \tr\mathcal{P}_{n_i}(H)\Big]=\Big(1+O\big(\frac{n}{W^d}\big)\Big)\mathbb{E}\Big[\prod_{i=1}^k \tr V_{n_i}\Big].\end{equation}
    By Proposition \ref{nbformula} and the equation \eqref{equ:Sigma1}, we have
    \begin{equation}\label{equ:lemma2.21_equ:2.68}
        \mathbb{E}\Big[\prod_{i=1}^k \tr V_{n_i}\Big]=\Sigma_{\beta}^{(1)}(n_1,n_2,\cdots,n_k)=\sum_{\mathfrak{D}\in \mathscr{G}_{\beta,k,\ge 3}}\sum_{w:(\mathfrak{D},w)}\Sigma_{\beta}^{(1)}(\mathfrak{D},w).
    \end{equation}
    Take    $\mathfrak{D}\in \mathcal{M}_{\text{diagram}}$ in \eqref{finesum}, we  obtain   
    \begin{equation}\label{equ:lemma2.21_equ:2.69}
        \sum_{\mathfrak{D}\in \mathcal{M}_{\text{diagram}}}\sum_{w(e):(\mathfrak{D},w)}\Sigma_{\beta}^{(0)}(\mathfrak{D},w) \le \sum_{\mathfrak{D'}\in \mathscr{G}_{\beta,k,\ge 3}}\sum_{w:(\mathfrak{D'},w)}\Sigma_{\beta}^{(1)}(\mathfrak{D'},w)\le \sum_{\mathfrak{D}\in \mathcal{M}_{\text{diagram}}}\sum_{w(e):(\mathfrak{D},w)}\Sigma_{\beta}^{(2)}(\mathfrak{D},w).
    \end{equation}
    Clearly, combination  of\eqref{equ:lemma2.21_equ:2.67}, \eqref{equ:lemma2.21_equ:2.68} and \eqref{equ:lemma2.21_equ:2.69} gives rise to     \eqref{equ:maximal_sum}. 
    
    Thus, this completes proof of the corollary.
    \end{proof}

The counting of any maximal $k$-diagram can be controlled by that of  some typical diagram, based on properties of random walks on the torus.

\begin{lemma}\label{lem:2.21}
For any weighted $k$-diagram $(\mathfrak{D},w)\in \mathcal{M}_{\text{diagram}}\backslash\mathscr{G}_{\beta,k,3}$ with graph $\bar{G}=(\bar{E},\bar{V})$, there is a weighted typical diagram $(\Tilde{\mathfrak{D}},\Tilde{w})\in \mathscr{G}_{\beta,k,3}$ with graph $\Tilde{G}=(\Tilde{E},\Tilde{V})$, such that for some constant $C>0$,  
\begin{equation}
    \Sigma_{\beta}^{(2)}(\mathfrak{D},w)\le C^{|\Tilde{E}|-|\bar{E}|}\Sigma_{\beta}^{(2)}(\Tilde{\mathfrak{D}},\Tilde{w}),
\end{equation}
where $\Tilde{w}(e)\in \{1,2\}$ for   $|\Tilde{E}|-|\bar{E}|$ (new) edges of $ \Tilde{E}$ and  $\Tilde{w}(e)=w(e)$ for   other edges $e\in \bar{E} \cap \Tilde{E}$.
\end{lemma}

\begin{proof}
Firstly, we  enlarge the diagonal terms from the diagonal matrix entries  to a self-loop with weight $w(e)=3$, as shown in Step 1 of Figure \ref{fig:diagonal_term}. In this case every diagonal term contributes a constant factor $C_1$. 

Secondly,  we need to get rid of all the vertices of degree greater than $3$. We will  use Lemma \ref{proposition:degree_operation} again  to split those   vertices of $\bar{G}$, as shown in Figure \ref{fig:degree_operation}. The crucial difficulty  is how to keep the circuit structure in   $\Tilde{G}$.

    Let $x$ be a vertex with $\text{deg}(x)>3$, we consider the edges, associated with the vertex $x$,  go in and out  each time as a set of unordered pair $S=\{(e_{i}^{(1)},e_{i}^{(2)})\}_{i=1}^{\text{deg}(x)}$ where $e_i^{(1)}$ and $e_{i}^{(2)}$ have length 1; see Figure \ref{fig:keep_circuit}. 
    We claim that after suitable sort the path set can be written as $\{(e_1,e_2),(e_2,e_3),\ldots, (e_{\text{deg}(x)},e_1)\}$. Otherwise, since each edge is traversed exactly twice, $S$ must be of the form $$S=\{(e_1,e_2),(e_2,e_3),\ldots (e_{i_1},e_1); 
    \cdots; (e_{i_{1}+\cdots+ i_{t-1}+1},e_{i_1+\cdots+ i_{t-1}+2}),\ldots, (e_{i_{1}+\cdots+ i_{t}},e_{i_1+\cdots+ i_{t-1}+1})\}.$$
    However, if $t>1$, the vertex $x$ must be a glued vertex, which is contradicted to $\mathfrak{D}\in \mathcal{M}_{\text{diagram}}$.

    Now if the path is of form $\{(e_1,e_2),(e_2,e_3),\ldots, (e_{\text{deg}(x)},e_1)\}$, we can add a vertex $y$ to decrease the degree $\text{deg}(x)$ as
    $\{(e_1,xy,e_2),(e_2,e_3),(e_3,xy,e_4),\ldots, (e_{\text{deg}(x)},e_1)\}$, whenever $\text{deg}(x)>3$. 
    Here 
      $xy$ means that  we add an edge $xy$ between the path $e_1,e_2$ and $e_3,e_4$; see Figure \ref{fig:keep_circuit} as an example. So the circuit structure  can be   conserved even  after this operation.  So continue, until the degree of $x$ is equal to 3. 
    
    Lastly,  we consider the diagram weight $\Sigma_{\beta}^{(2)}(\mathfrak{D},w)$. By Lemma \ref{proposition:degree_operation}, each splitting vertex operation contributes a constant factor $C_1$. Taking $n_3=1$ there, the weight of the new edge is $1$. The total number   of splitting vertex operations is $|\Tilde{E}|-|\bar{E}|$. Each diagonal term also contributes a constant factor $C_2$. Besides, the splitting operation  from the diagonal term of degree greater than 3  also gives  a   factor $C_1$. Hence,  after absorbing the constant $C_2$ by taking $C=C_1C_2$, we complete the proof. \end{proof}

\begin{figure}
    \centering
    \includegraphics[width=0.5\textwidth]{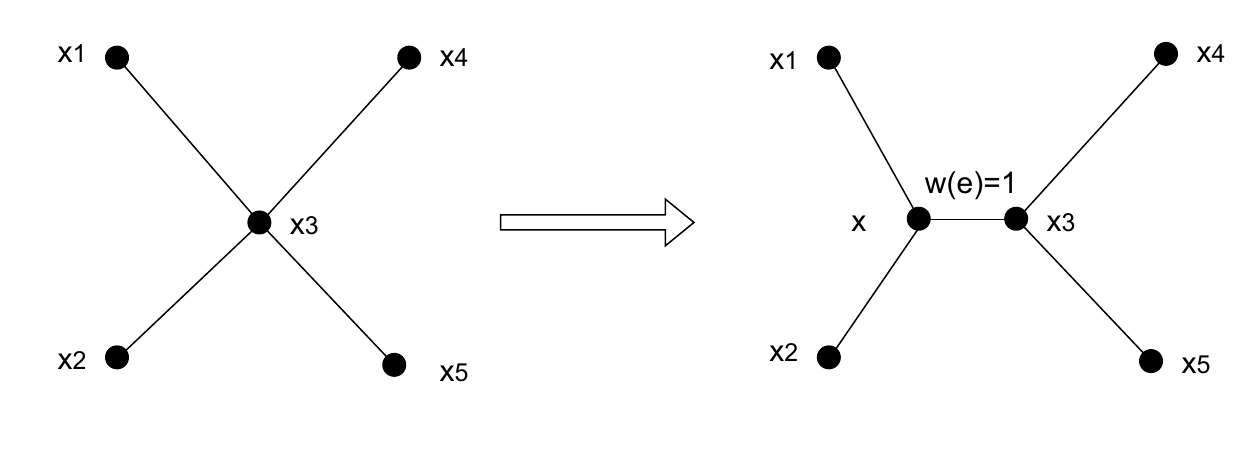}
    \caption{Split a  vertex with degree  greater than 3}
    \label{fig:degree_operation}
\end{figure}
\begin{figure}
    \centering
    \includegraphics[width=0.5\textwidth]{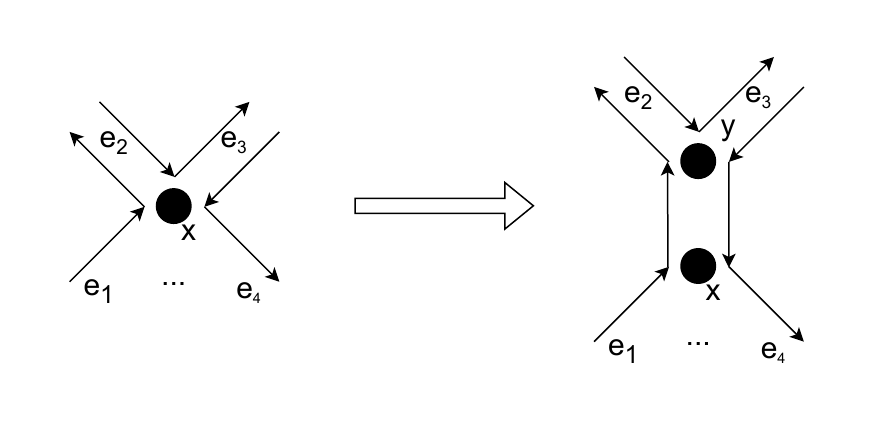}
    \caption{Keep the circuit structure}
    \label{fig:keep_circuit}
\end{figure}
\begin{figure}
    \centering
    \includegraphics[width=0.5\textwidth]{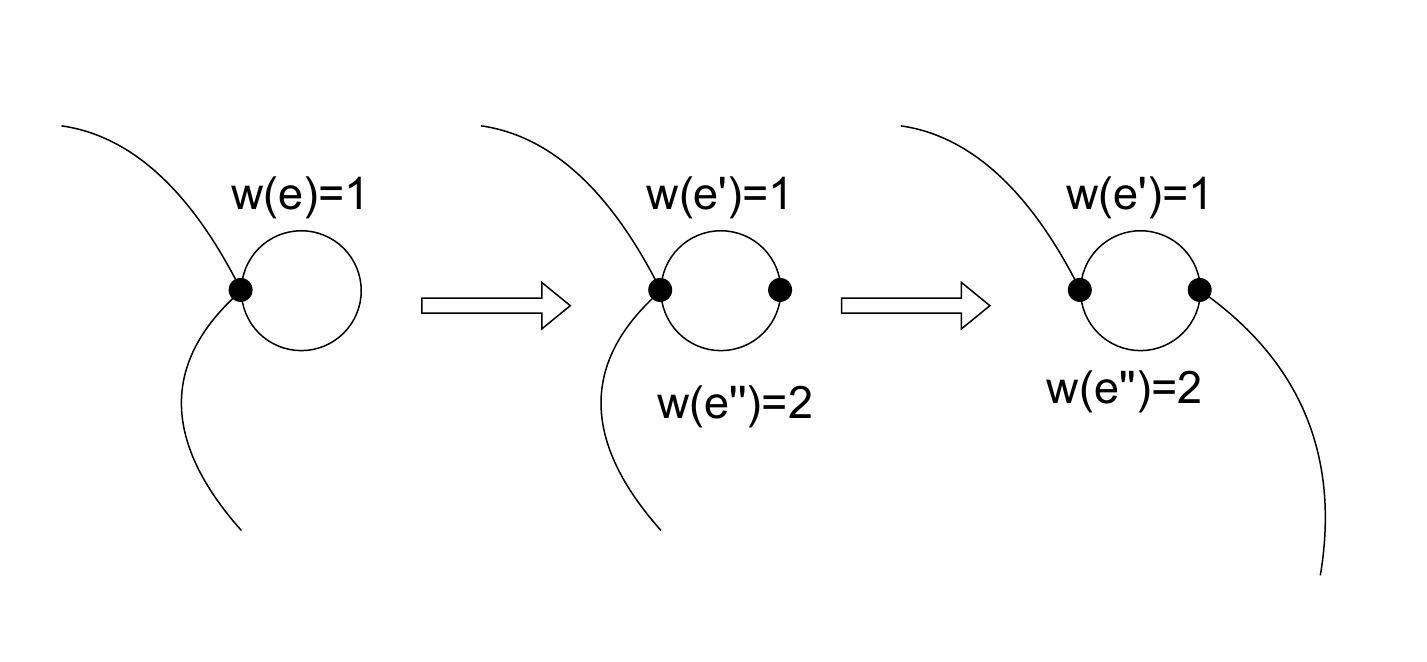}
    \caption{Split a  vertex with degree  greater than 3: diagonal case}
    \label{fig:diagonal_term}
\end{figure}

\newpage
\section{Diagram functions and integrals}\label{section_diagram}

{ \bf Notation}.  Throughout Section \ref{section_diagram}, Section \ref{section:asymptotic_expansion} and Section \ref{sec:tadpole_diagram}, we use the following conventions.  

\begin{itemize}
    \item For a given integer $k\geq 1$, $n:=\sum_{i=1}^k n_i$ where $n_1, \ldots, n_k$ are integers.
    \item We always assume  that  $ne^{-0.5c_{\Sigma}W^2}\to 0$ as $W\to \infty$, where $c_\Sigma>0$ denotes the smallest eigenvalue of the positive definite   matrix  $\Sigma$.

    \item For any graph $G=(V,E)$,  we use $|V|$ and $|E|$ to denote  the number of vertices and edges respectively for convenience. If $G$ is the underlying graph of diagram $\mathfrak{D}\in \mathscr{G}_{\beta, k,3}$, we have $|V|=2s$  and $|E|=3s-k$.

    \item We use $\Tilde{E}, \Tilde{V}$ to denote the graph corresponding to ${\mathfrak{D}}$ with all tail edges removed. For any  connected    diagram $\mathfrak{D}$ in $\mathscr{G}_{\beta, k,3}$, we have $|\Tilde{V}|=2s-k$ edges  and $|\tilde{E}|=3s-2k$ vertices.
\end{itemize}

\subsection{Diagram functions}\label{sec:diagram_function}
\begin{definition} \label{diagramf}
    For any {{connected diagram }} $\mathfrak{D}\in \mathscr{G}_{\beta,k,3}$ with graph $G=(V,E)$,   the $\mathfrak{D}$-diagram function is defined  as
    \begin{equation}\label{equ:3.1}
    \begin{aligned}
        F_{\mathfrak{D}}(\{n_i\}_{i=1}^k)&:=\frac{1}{N}\sum_{w(e):(\mathfrak{D},w)}\Sigma_{\beta}^{(2)}(\mathfrak{D},w)\\
        &=\frac{1}{N}\sum_{w(e):(\mathfrak{D},w)} 
        \sum_{\iota:V\rightarrow \Lambda_L}\prod_{e\in E}p_{w(e)}(\iota(u^{e}),\iota(v^{e})),
    \end{aligned}
    \end{equation}
    where $N=L^d$ and  the weight function $w(e)$ satisfies the $\mathfrak{D}$-system $\mathfrak{C}$ in Definition \ref{D-system}. For short,   we also use  the sum $\sum_{w(e):\mathfrak{C}}=\sum_{w(e):(\mathfrak{D},w)}$.
\end{definition}

A central objective of this paper is to characterize the asymptotic behavior of diagram functions. Building on foundational results by 
Sodin  \cite{sodin2010spectral}, who resolved this analytical challenge for both subcritical and supercritical regimes in the one-dimensional case ($d=1$), we investigate how transitions between diagram function phases determinate the statistics at spectral edges. These transitions establish a 3-fold classification of eigenvalue statistics at the   spectral edge, which is determined by the 3-fold transitions of random walk.

\begin{itemize}
 \item {\bf 
    Supercritical phase}: $n(W/L)^2\gg 1$.  The random walk 
    achieves uniform mixing over the large  torus and the edge statistics align with mean-field random matrix universality, governed by 
    the  Tracy-Widom distribution.
    
    \item {\bf 
    Subcritical phase}: $n(W/L)^2\ll 1$. The random walk
      localization induces Gaussian fluctuations within microscopic $\mathbb{R}^d$-scaled windows, and Poissonian eigenvalue statistics emerge, characteristic of weakly interacting systems.
      
    \item {\bf 
    Critical threshold}: $n(W/L)^2=\gamma>0$.  The heat kernel propagation dominates toroidal geometry and the interpolating correlation patterns arise, blending mean-field and localized characteristics.
\end{itemize}

\begin{proposition}[Subcritical case]\label{proposition:graph_polynomial}
    If  $ne^{-0.5c_\Sigma W^2} \ll 1$  and   $n\ll (\frac{L}{W})^2$ as $W\to \infty$, then  

\begin{equation}\label{equation:graph_polynomial}
\begin{aligned}
  & F_\mathfrak{D}(\{n_i\}_{i=1}^k)=\big(1+O(ne^{-0.5c_\Sigma W^2})\big)\cdot (\det(W^2{\Sigma}) )^{\frac{-|E|+|V|-1}{2}}\sum_{w(e): \mathfrak{C}}\Big(\sum_{T\in Str(\mathfrak{D}) } \prod_{e\notin T} w(e)\Big)^{-\frac{d}{2}},
\end{aligned}
\end{equation}
 where   $Str(\mathfrak{D})$ denotes the set of spanning trees  of $\mathfrak{D}$.
\end{proposition}
\begin{proof} Applying the local central limit   theorem   shown  in Theorem \ref{prop:gaussian_llt} first,   then using  Poisson summation formula and  the determinant of the Laplacian matrix,  we can proceed as in   {\cite[Eq. (4.5)]{sodin2010spectral}} to obtain the desired result. 
\end{proof}
    
       In dimension   $d=1$, Sodin  \cite{sodin2010spectral} has proved that  in the subcritical case  the summation on the right-hand side of \eqref{equation:graph_polynomial} tends  to a finite  integral, while in the supercritical case of  $n \gg (L/W)^2$,
    \begin{equation}
        F_\mathfrak{D}(\{n_i\})=(1+o(1))C_{\mathfrak{D}}n^{|E|-k}N^{|V|-|E|-1}.
    \end{equation}
Here $C_{\mathfrak{D}}$ is a constant given  in Definition \ref{D-system}. However, in  dimension $d>1$ the summation \eqref{equation:graph_polynomial} may not be approximated by   finite integrals. 
 We give a simple example that explains the above phenomena as follows.

\begin{figure}
\centering
\includegraphics[width=0.5\textwidth]{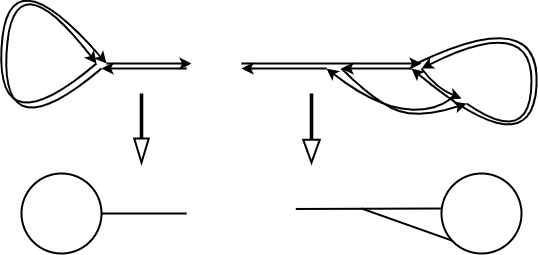}
\caption{Example of simplest diagram for $\beta =1$ and $\beta=2$}
\label{fig:singular_diagram}
\end{figure}

\textbf{Singularity problem.} Consider   the sum associated with the  loop diagram, which is usually called   a tadpole diagram (the left hand of figure \ref{fig:singular_diagram}) in   Quantum Field Theory (QFT for short). Its diagram function  reduces to  \begin{equation}\label{equ:singularity}
    ~W^{-d}\sum_{n_2=3}^{n}n_2^{-\frac{d}{2}}\sim \begin{cases}
    \frac{1}{2} W^{-1} \sqrt{n}, &\quad d=1,\\
    W^{-2} \log n, &\quad d=2,\\
    C_d W^{-d},&\quad d\ge 3.
    \end{cases}
\end{equation}
The singularity is closely related to 
the divergence of Feynman graph  integrals   and  is studied 
   in  Section \ref{subsection_singular_diagram}.

\subsection{Graph integrals}\label{subsection_singular_diagram}

The sum of multivariate polynomials on the right-hand side of \eqref{equation:graph_polynomial} 
is closely related to a multivariate integral on Feynman graph in QFT, and we have to attack the same   convergence  issue. 
A well-known approach to analyze convergence of Feynman integrals in QFT is to decompose the initial integration domain  into appropriate subdomains (sectors) and  then introduce  new variables  in each sector such  that the integrand  function 
properly factorizes;  see  Hepp \cite{hepp1966proof}  and Speer \cite{speer1975ultraviolet}. A very effective way to analyze   
the singularity of Feynman diagrams is the the maximal UV (ultra-violet) forest method; see e.g. \cite[Chapter 4]{smirnov2013analytic} and \cite{speer1975ultraviolet} for more details.  The maximal UV forest method gives the correct singularity criterion for our diagrams but without effective bound. In section \ref{sec:diagram_func_bound}, we use a combinatorial argument as illustrated  in Lemma \ref{lemma:upper_bound_integral} below to establish an effective bound.
 
Here we give a brief review  about  the maximal UV  forest method,  see e.g.  \cite[Chapter 4]{smirnov2013analytic} and \cite{speer1975ultraviolet} for more details.  

\begin{definition}\label{definition:graph_polynomial}
    For   a connected graph $G=(V,E)$, a homogeneous polynomial of Feynman parameters  $\{\alpha_e\}_{e\in  {E}}$
    \begin{equation}
        \mathcal{U}_{G}(\{\alpha_e\})=\sum_{T\in \mathrm{Str}(G)}\prod_{e \notin   T}\alpha_e
    \end{equation}
   is called as  first \textit{Symanzik polynomial}, where  $\mathrm{Str}(G)$  is the set of  spanning trees  of $G$.  
\end{definition}

The first \textit{Symanzik polynomial}  is also referred to as the 
\textit{Kirchhoff polynomial} or \textit{graph polynomial}. Consequently, the summation for the diagram $\mathfrak{D}$ in Proposition \ref{proposition:graph_polynomial}  can be rewritten as   
\begin{equation}
    \sum_{w(e): \mathfrak{C}}\Big(\sum_{  T  \in \mathrm{Str}(\mathfrak{D}) } \prod_{e\notin T} w(e)\Big)^{-\frac{d}{2}}=\sum_{w(e): \mathfrak{C}}\big(\mathcal{U}_{\mathfrak{D}}(\{w(e)\})\big)^{-\frac{d}{2}}.
\end{equation}
This summation can be approximated by a Feynman graph integral $I_{G}(D)$,  defined as   in \eqref{D-integral} below,  
   over a bounded domain 
 $D $ containing the origin. A fundamental question is to determine whether the graph integral converges.
   As has been known since the early days of QFT, Feynman graph integrals often suffer from divergences. These divergences arise because the integrals are ill-defined due to  UV singularities at the origin.
    
We require additional graph-theoretical concepts; for details, see e.g.    \cite[Chapter 4]{smirnov2013analytic}. 
\begin{definition} An articulation vertex of a graph $ G$
is a vertex whose deletion disconnects $G$.
    Any connected graph with no articulation vertices
is said to be 
one-vertex-irreducible (1VI). Otherwise, it is called
one-vertex-reducible (1VR).     A single line is treated as a 1VI graph
since it    can not be decomposed   into two parts.
   \end{definition}

 In  QFT, a UV forest is a collection of subgraphs within a Feynman diagram that plays a crucial role in analyzing and resolving  UV divergences. Specifically, a UV forest is defined as follows.
   \begin{definition}
   Given a graph $G$, a set $f$ composed of    1VI subgraphs is said to be   a UV forest if the following two conditions  hold:

    \begin{itemize}
        \item[(i)] for any pair $\gamma,\gamma' \in f$ ,  either $\gamma\subset \gamma', \gamma'\subset \gamma$ or $\gamma$ and $\gamma'$  have   no common edge;
        \item[(ii)] if $\gamma_1,...,\gamma_m\in f$ ($m>1$) and each pair $\gamma_i, \gamma_j$  have  no common edge, the subgraph $\cup_{i=1}^m \gamma_i$ as the union is 1VR.        
    \end{itemize}  
A maximal UV forest is a UV forest that is not contained in any other UV forest. Denote by  $\mathcal{F}$ the set of maximal UV forest. 
\end{definition}

As shown in \cite[Chapter 4]{smirnov2013analytic}, 
for a given maximal UV forest $f\in \mathcal{F}$, we can introduce a corresponding sector  $D_f$ such that the following facts hold. 
\begin{itemize}
        \item[(i)] 
        the intersection of two different sectors has  measure zero;
        \item[(ii)]   the union of all the
sectors is exactly   the whole integration domain of the alpha parameters;
        \item[(iii)]   for a given sector $D_f$,  there exist   new variables labelled by the elements of $f$   \begin{equation}
    \alpha_e=\prod_{\gamma\in f:e\in \gamma}t_\gamma,
\end{equation}
such that   the corresponding  Jacobian  equals to $\prod_\gamma t_\gamma^{L(\gamma)-1}$
and  the Symanzik polynomial  takes the form 
\begin{equation}
    \mathcal{U}(\{\alpha_e\})=\prod_{\gamma\in f}t_\gamma^{h(\gamma)}[1+P_f],
\end{equation} 
where  $P_f$  is a  polynomial with non-negative coefficients,  $h(\gamma):=L(\gamma)-V(\gamma)+1$,   $L(\gamma)$ and $V(\gamma) $ are respectively   the number of loops, edges and vertices in the graph. 
    \end{itemize}

With the above preparation, we can state the following convergence  rule for Feynman  graph integrals (see \cite[Chapter 4]{smirnov2013analytic} for a detailed proof). 
\begin{proposition}[{\cite{smirnov2013analytic,speer1975ultraviolet}}]\label{propositionconvergence} 
Given a connected graph $G=(V,E)$, let   $D=[0,1]^{|E|}$,    then  the graph integral 
\begin{equation} \label{D-integral}
       I_{G}(D):=\int_D    \big(\mathcal{U}_{G}(\{\alpha_e\})\big)^{-\frac{d}{2}}   \prod_{e\in {E}} d\alpha_e
\end{equation}
 is finite if and only if   the  discriminant
  \begin{equation}\label{equ:criterion_singularity}
    \Delta_{G}:=|E(\gamma)| - \frac{d}{2}\big(|E(\gamma)|-   |V(\gamma)|+1\big)>0 
    \end{equation}
    for every     $\mathrm{1VI}$ subgraph $\gamma$ of $G$, where $E(\gamma)$ and $V(\gamma) $ are respectively   the set of     edges and vertices.
\end{proposition}

Next,   we introduce the notion of singularity or regularity for a one-vertex irreducible (1VI) graph. This concept will heavily depend on the dimension $d\geq 0$.

\begin{definition}\label{def:4.8}  For any real $d\geq 0$, 
a one-vertex irreducible (1VI) graph  $G$ is said to be primitively 
$d$-singular   (or simply   singular) if its discriminant satisfies 
$\Delta\leq 0$, but no proper 1VI subgraph of  $G$ satisfies this condition.
Conversely, a graph is said to be $d$-regular (or regular) if none of its 1VI subgraphs is singular.
   \end{definition} 

Apply Proposition \ref{propositionconvergence} and  we can easily classify 
the Feynman diagrams  in  the union  $G_{3}:=\cup_{k\geq 1} \mathscr{G}^{c}_{1,k,3}$, where  
   $\mathscr{G}^c_{1,k,3}$ is the set of connected typical   $k$-diagrams of type 1 as   in    Definition \ref{def:2.15}.
 First,  for the  1VI graphs  that has   vertex of degree 1,  there are only  two simple diagrams: a line or a tadpole.  Since any diagram of $G_{3}$    has only  vertex of degree  not   greater than 3,  so does   a  subgraphs. Let $V_2$ and $V_3$ be respectively the number of vertices of degree 2 and 3 for a subgraph, the above singular criterion for the 1VI subgraphs  can be   immediately  reformulated as 
 
\begin{equation}\label{criterion:gaussian_case}
    \Delta(V_2,V_3)=V_2+\big(\frac{3}{2}-\frac{d}{4}\big)V_3-\frac{d}{2}\leq 0.
\end{equation}
See Table \ref{singular_V2V3_table} for explicit pairs.  
\begin{table}
    \centering
    \begin{tabular}{|c| c|}
        \hline
        $d$ & $(V_2, V_3)$ \\
        \hline
        1 & None   \\
                2 & (1,0),(0,1)\\
                3 & (1,0),(0,1),(0,2)\\
                4 & (1,0),(1,1),(2,0),(0,1),(0,2),(0,3),(0,4)\\
                5 &$V_2+\frac{1}{4}V_3\le \frac{5}{2}$\\
                6 &  All $V_2=1,2,3$ \& even $V_3$\\
        \hline
    \end{tabular}
    \caption{Possible vertices $(V_2, V_3)$
    with $\Delta(V_2,V_3)\leq 0$}
    \label{singular_V2V3_table}
\end{table}

 However, not all pairs $(V_2, V_3)$ listed in Table \ref{singular_V2V3_table} can actually form a subgraph of any diagram from $G_3$. A straightforward observation reveals that any subgraph must have at least one vertex with a degree less than 3, and that $V_3$ must be even, as must $2V_2 + 3V_3$. Taking the tadpole diagram into consideration, all possible singular graphs are listed in Table \ref{singular_diagram_table} according to the dimension $d$.
This is entirely consistent with all the graphs of overall divergences in the $\phi^3$ model; see \cite[Chapter 5.7]{collins1984renormalization}. It is particularly interesting that the singular criterion is also equivalent to Hall's theorem, as demonstrated in Lemma \ref{lemma:upper_bound_integral} below.

If $d<6$, then only a finite set of singular graphs, illustrated in  Table \ref{singular_diagram_table},  have overall divergences.    So we  expect to have   renormalizability as in   $\phi^3$ model.  When $d>6$, there are too many singular graphs, so $\phi^3$ theory is not renormalizable; see \cite[Chapter 5.7]{collins1984renormalization}.

\begin{table}
    \centering
    \begin{tabular}{|c| c|}
        \hline
        $d$ & Additional   singular patterns          \\
        \hline
        1 & None  \\
                2 & Tadpole\\
                3 & None\\
                4 &          $(2,0), (1,2)$\\
                5 &   $(2,2),(1,4)$\\
                6 & All $V_2=1,2,3$ \& even $V_3$\\
        \hline
    \end{tabular}
    \caption{Singular patterns in  typical diagrams  }
    \label{singular_diagram_table}
\end{table}

\begin{figure}
\centering
\includegraphics[width=0.7\textwidth]{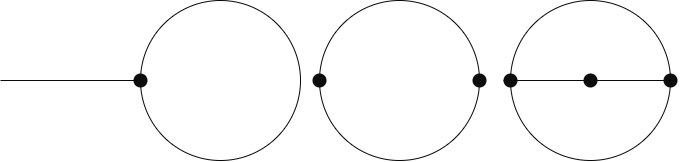}
\caption{Singular patterns for $d\le 4$}
\label{singular_diagram_ks<3}  
\end{figure}

\begin{figure}
\centering
\includegraphics[width=0.5\textwidth]{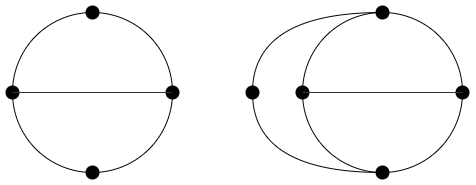}
\caption{Additional singular patterns at $d= 5$, exactly at $d={14}/{3}$ (left) and   $d=5$ (right).}
\label{singular_diagram_5d}
\end{figure}

\subsection{Asymptotics for regular diagram functions}\label{sec:diagram_func_bound}
Let  $\mathfrak{C}(\mathfrak{D})$ be  the system of linear equations  in Definition \ref{D-system},   
    the ratio  between   the number of integer solutions in the restricted  domain and the domain  volume   can be  approximated  by a constant defined by   
\begin{equation} \label{CD}
    C_{\mathfrak{D}}:=\lim_{n_1,\ldots,n_k \to \infty}  \frac{\#\big\{\{w(e)\}_{e\in E}| \text{ all } w({e})\in \mathbb{Z}_{+} \text{ and } \{w({e})\} \text{ satisfying } \mathfrak{C}(\mathfrak{D})\big\}}{\mathrm{Vol}\big\{\{ w(e)\}| \text{ all } w(e)\in \mathbb{R}_+ \text{ and } \{w(e)\}  \text{ satisfying }  \mathfrak{C}(\mathfrak{D})\big\}},
\end{equation} 
where $\mathbb{Z}_{+}$ and $\mathbb{R}_{+}$ denote the sets of nonnegative integers  and nonnegative real numbers, respectively. To illustrate the aforementioned constants, we consider two simple linear equations:
(1) $x+y=n$, (2) $x+2y=n$. For these cases, we respectively observe
        \begin{equation}
            \frac{\#\{(x,y)\in\mathbb{Z}^2_+|x+y=n\}}{\text{Vol}\{(x,y)\in \mathbb{R}^2_+|x+y=n\}}=\frac{n-1}{\sqrt{2}n}\rightarrow\frac{\sqrt{2}}{2},
        \end{equation}
and   \begin{equation}
            \frac{\#\{(x,y)\in\mathbb{Z}^2_+|x+2y=n\}}{\text{Vol}\{(x,y)\in \mathbb{R}^2_+|x+2y=n\}}=\frac{[\frac{n}{2}]}{\frac{1}{2}\sqrt{5}n}\rightarrow\frac{\sqrt{5}}{5}.
        \end{equation}

Similar to the linear system    in Definition \ref{D-system},  for   $\tau_i\in (0,\infty)$
we also introduce a simplex formed by systems of linear equations in $\mathbb{R}_{+}^{|E|}$
  \begin{equation} 
        \mathfrak{C}(\tau): \quad  \sum_{e\in {E}}c_i(e)\alpha_e=\tau_j, \qquad i=1,...,k,
    \end{equation}
    and  
 another  simplex formed by systems of linear inequalities  in $\mathbb{R}_{+}^{|E|-k}$      
  \begin{equation}
        \widetilde{\mathfrak{C}}(\tau): \quad  \sum_{e\in {E}\backslash T_i}c_i(e)\alpha_e\le \tau_i,  \qquad i=1,...,k,
    \end{equation}
where      $c_i(e)\in \{0,1,2\}$ and $T_i$ is  the tail edge in the $i$-circuit.

The central results of this section are the following two theorems about   limits and upper bounds  for diagram functions. Their proofs will be left to the end of this section.
\begin{theorem}\label{thm:homogeneous_theorem}
       
  For  any  connected and  regular   diagram $\mathfrak{D}$ in $\mathscr{G}_{\beta, k,3}$, assume that 
      $\sum_{i=1}^k n_i$ is even   and   $\tau_i\in (0,\infty)$ for $i=1,\ldots,k$.
    \begin{itemize}
        \item[(i)] 
({\bf Supercritical regime})
    If  $W\gg L^{(1-\frac{d}{6})_+}$  and all $n_i\sim \tau_i N^{\frac{1}{3}}$, then
    \begin{equation}
        F_\mathfrak{D}(\{n_i\}_{i=1}^k)=
        (1+o(1))C_{\mathfrak{D}}\mathrm{Vol}_{ \mathfrak{C}(\tau)}\big(\big\{\tau_i\big\}\big)N^{\frac{k}{3}-1}.
    \end{equation}

 \item[(ii)]({\bf Subcritical regime})     
    If $1\ll W\ll L^{1-\frac{d}{6}}$ ($d<6$) and all  $  n_i\sim \tau_i W^{\frac{2d}{6-d}}$, then  \begin{equation}
    \begin{aligned}
 F_\mathfrak{D}(\{n_i\}_{i=1}^k)& 
 &= (1+o(1))C_{\mathfrak{D}}\det({\Sigma})^{\frac{k-s-1}{2}}{{W^{\frac{2d(k-3)}{6-d}}}}\int_{\mathfrak{C}(\tau)}  \big(\mathcal{U}_{\mathfrak{D}}(\{\alpha_e\})\big)^{-\frac{d}{2}} d\alpha.
\end{aligned}
    \end{equation}
   
\item[(iii)] ({\bf Critical regime}) 
       If  $  W\sim \gamma L^{1-\frac{d}{6}}$ ($d<6$) for some $\gamma>0$ and  
    all  $n_i\sim \tau_i (\frac{L}{W})^2$, then 
    \begin{equation}
        F_\mathfrak{D}(\{n_i\}_{i=1}^k)=(1+o(1))C_{\mathfrak{D}}\big(\frac{L}{W}\big)^{2(|E|-k)}N^{-(|E|-|V|+1)}\int_{{\mathfrak{C}(\tau)}}\int_{x_v\in \mathbb{T}^d}\prod_{e\in \Tilde{E}}\theta(x_{e1}-x_{e_2},\alpha_e\Sigma ) d\alpha dx.
    \end{equation}
 \end{itemize}
        
\end{theorem}

\begin{theorem}\label{thm:upper_bound_diagram_function}
  Assume that $ne^{-0.5c_{\Sigma}W^2}\to 0$ as $W\to \infty$.
 For  any $d\ge 0$ and any  connected  diagram $\mathfrak{D}$ in $\mathscr{G}_{\beta, k,3}$ 
  with $2s$ edges  and $3s-k$ vertices, let  $d'\in  [0,d]$ such that  $\mathfrak{D}$ is  $d'$-regular, then there are positive constants $C$ and $ C_1$  independent of on $n_1, \ldots,n_k$ and $W$   
      such that
    \begin{equation}
    \begin{aligned}
        F_{\mathfrak{D}}(\{n_i\}_{i=1}^k)&\le \frac{(Cn)^{3s-2k}}{(3s-2k)!}N^{k-s-1}+C_1\frac{(Cn)^{(3-\frac{d'}{2})s+(\frac{d'}{2}-2)k-\frac{d'}{2}}}{((3-\frac{d'}{2})s+(\frac{d'}{2}-2)k-\frac{d'}{2})!}W^{(k-s-1)d}.
        \end{aligned}
    \end{equation}
\end{theorem}
        
To obtain Theorem \ref{thm:upper_bound_diagram_function}, we first sum  up over the $k$ marked points   and  thus see that   the restriction $ {\mathfrak{C}}(\mathfrak{D})$   should be   replaced by $\widetilde{\mathfrak{C}}(\mathfrak{D})$ in Definition \ref{D-system}. In this case, 
let $\Tilde{V},\Tilde{E}$ denote the resulting vertex set and edge set by removing tail edges and marked points, then  the diagram function in Definition \ref{diagramf} can be rewritten as 
\begin{equation} \label{diagramf-3}
    F_{\mathfrak{D}}(\{n_i\}_{i=1}^k)=\frac{1}{N}\sum_{w(e): \widetilde{\mathfrak{C}}(\mathfrak{D})}\, \sum_{\iota: \Tilde{V}\rightarrow \Lambda_L}\prod_{e\in \Tilde{E}}p_{w(e)}(\iota(u^e),\iota(v^e)),
\end{equation}
where the product over all non-tail edges and summation  over the vertices expect for those of degree 1 in diagram $\mathfrak{D}$. At  this time, we have $|\Tilde{V}|=2s-k$ and  $|\Tilde{E}|=3s-2k$.
        
With the help of  the heat kernel bound in Proposition \ref{prop:gaussian_hkb},  
  we have   an upper bound estimate for the diagram function $F_{\mathfrak{D}}.$

    \begin{lemma}\label{coro:diagram_function_upper_bound} 
        If $ne^{-0.5c_{\Sigma}W^2}\to 0$ as $W\to \infty$, 
        then for  any $d\ge 0$ and any  connected  diagram $\mathfrak{D}$ in $\mathscr{G}_{\beta, k,3}$ 
       we have           \begin{equation}\label{equ:3.20}
            F_{\mathfrak{D}}(\{n_i\}_{i=1}^k)\le C^{| \Tilde{E}|}\sum_{w(e):\widetilde{\mathfrak{C}}(\mathfrak{D})}\sum_{E_1\subset \Tilde{E}}\mathcal{U}_{E_1}(\{w(e)\})^{-\frac{d}{2}} N^{ -  |\Tilde{E}| +  |\Tilde{V}|-1}
                \Big(\frac{N}{W^d}\Big)^{ |E_1|  - |V_1|+c(E_1)}
            \end{equation}
           for some  constant   $C>0$ only depending on $d$ and $k$.  Here $E_1$ is any edge subset of $\Tilde{E}$ that can be treated as  a subgraph 
            with the vertex set $ V_1 $, and   $c(E_1)$ denotes the number of connected components of $E_1$. 
    \end{lemma}
\begin{proof}

Let $x_v=(x_v^{(1)},\ldots,x_v^{(d)})\in \Lambda_L$,  it's easy to see from \eqref{diagramf-3} that 
\begin{align*}
            F_{\mathfrak{D}}(\{n_i\}_{i=1}^k)
            &= \frac{1}{N}\sum_{w(e):\Tilde{\mathfrak{C}}(\mathfrak{D})}\sum_{x_{v}\in \Lambda_L, v\in \Tilde{V}}\prod_{e\in \Tilde{E}}p_{w(e)}({x}_{u^e},{x}_{v^e}),
        \end{align*} where the second sum is over all $\Lambda_L$-value of  all vertices.

Use   the  uniform  upper bound of the heat kernel in Proposition \ref{prop:gaussian_hkb} and we obtain
        \begin{align} \label{Fdupper}
            F_{\mathfrak{D}}(\{n_i\}_{i=1}^k)&\le \frac{1}{N} \sum_{w(e): \Tilde{\mathfrak{C}}(\mathfrak{D}) }\sum_{x_{v}\in \Lambda_L, v\in \Tilde{V}}\prod_{e\in \Tilde{E}}\prod_{i=1}^d C_{1}\Big(
            \frac{1}{\sqrt{w(e)W^2}}   e^{-C_{2}\frac{({x}_{u^e}^{(i)}-{x}_{v^e}^{(i)})^2}{w(e)W^2}}+\frac{1}{L}\Big)\nonumber\\
            &= \frac{1}{N} \sum_{w(e):\Tilde{\mathfrak{C}}(\mathfrak{D}) }\Big(\sum_{x_{v}^{(1)}\in \mathbb{Z}/L\mathbb{Z}}\prod_{e\in \Tilde{E}} C_{1}\bigg(
            \frac{1}{\sqrt{w(e)W^2}}   e^{-C_{2}\frac{({x}_{u^e}^{(1)}-{x}_{v^e}^{(1)})^2}{w(e)W^2}}+\frac{1}{L}\Big)\bigg)^d,
            \end{align}
            where in the last equality  the summation  over independent coordinates has been used.  
            
            By expanding  the product and introducing  any  subset that consists of some edges from $\Tilde{E}$, we have 
            \begin{equation}
                \prod_{e\in \Tilde{E}}   \Big(
            \frac{1}{\sqrt{w(e)W^2}}   e^{-C_{2}\frac{({x}_{u^e}^{(1)}-{x}_{v^e}^{(1)})^2}{w(e)W^2}}+\frac{1}{L}\Big)=\sum_{E_1\subset \Tilde{E}}L^{|E_1|-|\tilde{E}|}\prod_{e\in  E_1} 
            \frac{1}{\sqrt{w(e)W^2}}   e^{-C_{2}\frac{({x}_{u^e}^{(1)}-{x}_{v^e}^{(1)})^2}{w(e)W^2}}.
            \end{equation}
             Furthermore, we first  take summation   over $x_v^{(1)}\in \mathbb{Z}/L\mathbb{Z}$ where $v$ is chosen  either one   vertex in 
            $\Tilde{V} \backslash V_{1}$ or  any selected  one vertex from each connected component of $E_1$
            ,  then change the range of all other vertices of  $E_1$ (denoted by $V_1^*$) from $\mathbb{Z}/L\mathbb{Z}$ to $\mathbb{Z}$ and     approximate  the sum by  a Gaussian integral  
            , similar to Proposition \ref{proposition:graph_polynomial},  and  obtain   
            \begin{equation} \label{uppereach}
            \begin{aligned}  
            &\sum_{x_{v}^{(1)}\in \mathbb{Z}/L\mathbb{Z}}\sum_{E_1\subset \Tilde{E}}L^{|E_1|-|\tilde{E}|}\prod_{e\in  E_1} 
            \frac{1}{\sqrt{w(e)W^2}}   e^{-C_{2}\frac{({x}_{u^e}^{(1)}-{x}_{v^e}^{(1)})^2}{w(e)W^2}} \\
            &=\sum_{E_1\subset \Tilde{E}}\sum_{x_{v}^{(1)}\in \mathbb{Z}, v\in V_1^*}L^{|E_1|-|\tilde{E}|+|\tilde{V}|
            -|V_1|+c(E_1)
            }\prod_{e\in  E_1} 
            \frac{1}{\sqrt{w(e)W^2}}   e^{-C_{2}\frac{({x}_{u^e}^{(1)}-{x}_{v^e}^{(1)})^2}{w(e)W^2}}\\
               &\leq \sum_{E_1\subset \Tilde{E}} (C_3 W)^{-|E_1|+ |V_1|-c(E_1)}\mathcal{U}_{E_1}(\{w(e)\})^{-\frac{1}{2}}L^{|E_1|-|\Tilde{E}|+|\Tilde{V}|-|V_1|+c(E_1)},
            \end{aligned}   \end{equation}
           where $c(E_1)$ is the number of connected components of $E_1$, and $\mathcal{U}_{E_1}(\{w(e)\})$ is the product of $\mathcal{U}$ associated with   connected components of $E_1$ ($\mathcal{U}_{E_1}=1$ when    a connected component is  a tree).  

           Combine \eqref{Fdupper}-\eqref{uppereach}, in  the Hölder's  inequality   $$\big(\sum_{i=1}^q x_i\big)^d\le q^{d-1}\sum_{i=1}^q x_{i}^d, \quad x_1, \ldots, x_q\ge 0,$$ 
           take $x_i$ as the right-hand side of \eqref{uppereach}  for every $E_1$ and $q$ as $2^{|\Tilde{E}|}$  and we thus    arrive at the desired result.
\end{proof}
With the same notation as in Lemma \ref{coro:diagram_function_upper_bound},  we can loose the restriction on  $w(e)$ from $\Tilde{\mathfrak{C}}(\mathfrak{D})$ to a single simplex $\sum_{e\in \Tilde{E}}w(e)\le {n}/{|\Tilde{E}|}$.
\begin{lemma}\label{proposition4.21}
If $ne^{-0.5c_{\Sigma}W^2}\to 0$ as $W\to \infty$, then  there is a constant $C>0$ such that  

 \begin{equation}\label{ue1}
    \begin{aligned}
        &\sum_{w(e): \Tilde{\mathfrak{C}}(\mathfrak{D}) }\mathcal{U}_{E_1}(\{w(e)\})^{-\frac{d}{2}} \le  C^{ |\Tilde{E}|}\sum_{\forall e, w(e)\le \frac{n}{ |\Tilde{E}|}}\mathcal{U}_{E_1}(\{w(e)\})^{-\frac{d}{2}},
    \end{aligned}
    \end{equation}
and
    \begin{equation} \label{ue2}
          F_{\mathfrak{D}}(\{n_i\}_{i=1}^k)\le C^{|\Tilde{ E}|}\sum_{\forall e, w(e)\le \frac{n}{ |\Tilde{E}|}}\sum_{E_1\subset \Tilde{E}}\mathcal{U}_{E_1}(\{w(e)\})^{-\frac{d}{2}} N^{ -  |\Tilde{E}| +  |\Tilde{V}|-1}
                \Big(\frac{N}{W^d}\Big)^{ |E_1|  - |V_1|+c(E_1)}.
             \end{equation}
\end{lemma}
\begin{proof} Add together all $k$ inequalities  in \eqref{CDtildesystem},   
   loose the restriction   and we thus have 
\begin{equation}
    \sum_{w(e):\Tilde{\mathfrak{C}}(\mathfrak{D})}\mathcal{U}_{E_1}(\{w(e)\})^{-\frac{d}{2}} \le\sum_{\sum_{e\in \Tilde{E}} w(e)\le n} \mathcal{U}_{E_1}(\{w(e)\})^{-\frac{d}{2}}.
\end{equation}

Note that   $\mathcal{U}_{E_1}(\{w(e)\})^{-d/2}$ is decreasing on variables $\{w(e)\}$, we can use the summation over integer lattice  points  lying in the  hypercube    region 
\begin{equation}
     \sum_{\forall e\in \Tilde{E}, w(e)\le \frac{n}{ |\Tilde{E}|}}\mathcal{U}_{E_1}(\{w(e)\})^{-\frac{d}{2}}
\end{equation}
to control the sums over other  integer lattice regions.  For this, we need to estimate the number of regions to be covered.   According to the homogeneity of the polynomial
$\mathcal{U}_{E_1}$,  it   is asymptotically equivalent to count the number of unit boxes in the simplex     $$\Big\{x_e \in\mathbb{R}_+ |\sum_{e\in \Tilde{E} }x_e\le  |\Tilde{E}|\Big\}.$$

    We can pick all integer lattice  points in a larger simplex  $\sum_{e}x_e \le 2  |\Tilde{E}|$, and make a cover with the closest-to-original vertex to be lattice  integer points. Then this cover contains all integer solutions of  $\sum_{e}x_e\le  |\Tilde{E}|.$  By calculating the volume of the simplex,   we see that the number of the unit
    boxes   is bounded   by  
    \begin{equation}
    \#\big\{w_e\in \mathbb{Z}_{+}|\sum_{e}w_e \le 2  |\Tilde{E}| \big\}\le \frac{(3 | \Tilde{E}| )^{ |\Tilde{E}|}}{( |\Tilde{E}|-1)!}\le C^{ |\Tilde{E}|}
    \end{equation}
    for some $C>0$. This implies  the first part \eqref{ue1}.

    The second part \eqref{ue2} is a consequence of \eqref{ue1}.   \end{proof}

Clearly, by changing a sum into an integral we have  the following proposition where $m=n/|\tilde{E}|$, cf. Lemma \ref{proposition4.21}. 

               \begin{lemma}
               \label{proposition4.22}
       For  any regular connected diagram    $G_0=(V_0, E_0)$, for sufficiently large $m$  and for some $C>0$ we have 
            \begin{equation}
            \begin{aligned}
            \sum_{\forall e, w(e)\le m}\mathcal{U}_{G_0}(\{w(e)\})^{-\frac{d}{2}}
            &\le Cm^{ |E_0|-\frac{d}{2}(  |E_0| -  |V_0|+1)}\int_{\alpha_e\le 1}\mathcal{U}_{G_0}(\{\alpha_e\})^{-\frac{d}{2}}d\alpha.
            \end{aligned}
            \end{equation}
        \end{lemma}

In order to  obtain  an upper bound estimate  for graph integrals, we need to estimate the number of spanning trees.         \begin{lemma} \label{spanningnumber}
                The number of spanning trees in   any connected  graph  $G_0=(V_0, E_0)$ in which every vertex has   at most  3
 is  equal or less than      $18^{ |E_0|}.$
        \end{lemma}
        \begin{proof}
        Note that  each vertex has degree at most 3,   calculate the determinant  of the Laplacian matrix denoted by $A$ and we use Hadamard's inequality to obtain 
        \begin{equation}
            |\det A|\le \prod_{i=1}^{V_0} \|\mathbf{v}_i\|\le \sqrt{18}^{|V_0|}\le 18^{|E_0|}, 
        \end{equation}
       where  $\mathbf{v}_i$ represent all  row vectors of $A$.   This thus   completes the proof.
        \end{proof}

        In  order to deal with the subcritical and  critical cases,  we need   an   upper  bound estimate for any regular graph integrals.  
        \begin{lemma}\label{lemma:upper_bound_integral}       Assume that $d<6$,  then   for 
         any  connected and  regular   diagram
          $G_0=(V_0, E_0)$
          in $\mathscr{G}_{\beta, k,3}$ there is a constant $\delta>0$,  depending  on $d$ but not on $G_0$, such that      
            \begin{equation}
                I_{G_0}=\int_{\alpha_e\le 1}\mathcal{U}_{\mathfrak{S}}(\{\alpha_e\})^{-\frac{d}{2}}d\alpha\le 18^{ |E_0|}\delta^{-  |E_0| +  |V_0|-1}. 
            \end{equation}
        \end{lemma}
        
        \begin{proof} Our strategy is to decompose the integration domain  into different sectors according to minimal spanning trees and then to prove that the corresponding integral in each sector can be bounded   by $\epsilon^{-  |E_0| +  |V_0|-1}$. 

            Align  a   weight $\alpha_e$ for   every edge, we then  have a unique minimal spanning tree $T_0$ except for a set of  measure zero.   The corresponding sector  can be  chosen  as  $$D_{T_0}=\big \{ \forall  
            e\notin T_0, \alpha_e\ge \max\{\alpha_{e'}:e\cup (T_0\!\setminus e') \,\text{is a spanning tree}, e'\in T_{0}\} \big \}.$$
            Clearly,   for  such a sector we have $$\mathcal{U}(\{\alpha_e\})\ge \prod_{e\notin T_0}\alpha_e.$$ 

            Considering all spanning trees,  we know that  the union of all those  sectors is exactly   the whole integration domain. Also notice 
            the upper bound for the number of spanning trees in 
           Lemma \ref{spanningnumber}, it suffices to 
            prove the upper bound estimate
            \begin{equation}
                \int_{D_{T_0}\cap [0,1]^{ |E_0|}} \prod_{e\notin T_0}\alpha_e^{-\frac{d}{2}} d \alpha\leq \delta^{-  |E_0| + |V_0|-1}.
            \end{equation}
                        
            Let $\mathcal{C}_e$ be the set of edges in the union   $e\cup T_0.$  Since the graph is regular,  we  see from the criterion \eqref{equ:criterion_singularity} that   $$\frac{4V_2+6V_3}{V_3+4}-d>0$$ for all 1VI subgraphs. 
           Take     \begin{equation}\delta:=\frac{1}{4}\cdot \min_{V_2, V_3} \Big\{\frac{4V_2+6V_3}{V_3+4}-d\Big\}. \end{equation} Here the minimal is taking over all $V_2,V_3\ge 0$, such that \begin{equation}
           \frac{4V_2+6V_3}{V_3+4}-d>0,
           \end{equation} saying $\delta=0.01$.

            And hence $(4V_2+6V_3)/(V_3+4)-(d+2\delta)>0$ by the choice of the $\delta$. Hence we have
                        \begin{equation}
               \# \big\{\cup_{e\in I} \mathcal{C}_e \big\}\ge \big(\frac{d}{2}+\delta \big) |I|,\quad  \forall I\subset E_0\setminus T_0.
            \end{equation}
This is    just the Hall's marriage condition; see e.g. \cite[Chapter 2]{tutte2001graph}. So by Hall's Marriage  Theorem in combinatorics,  
  we can  distribute at least $(\frac{d}{2}+\delta)-1$  
  copies of $\alpha_{e'}$ with $e'\in T_0$ to $\alpha_e$ (here -1 is used  to neglect  the case of $e\in \mathcal{C}_e$ itself). With such a matching,  combining  the assumption  $\alpha_{e'}\le \alpha_e$, first integrating  out variables from the tree and then other variables,    we thus get 
            \begin{equation}
                \int_{D_{T_0}\cap [0,1]^{ |E_0| }} \prod_{e\notin T_0}\alpha_e^{-\frac{d}{2}} d \alpha \le
                \int_{  [0,1]^{ |E_0|- |V_0|+1}} \prod_{e\notin T_0}  \alpha_e^{-1+\delta} d \alpha
                \le \delta^{-|E_0|+|V_0|-1}.
            \end{equation}

For any given  $d\le 5$,  there are only finite $(V_2,V_3)$ such that $(4V_2+6V_3)/(V_3+4)<d+0.5$. Verify all possible values and we know that the minimal value of $(4V_2+6V_3)/(V_3+4)-d=0.04$ is taken at $(V_2,V_3,d)=(0,21,5).$

           Combine  the number of spanning trees in Lemma \ref{spanningnumber} and we thus complete the   proof.
        \end{proof}
        
\begin{corollary}\label{proposition:poly_upper_bound}
Assume that $d<6$ and $\mathfrak{D}$ with graph $(\Tilde{E},\Tilde{V})$ is a regular and connected   diagram. If $n(\frac{W}{L})^2\ll 1$, then there is a constant $C>0$ only depending on $d$ such that
    \begin{equation} \label{regularupb}
         \sum_{w(e): \Tilde{\mathfrak{C}}(\mathfrak{D}) }
 \big(\mathcal{U}_{\mathfrak{D}}(\{w(e)\})\big)^{-\frac{d}{2}} \le \frac{(Cn)^{|\Tilde{E}|-\frac{d}{2}(|\Tilde{E}|-|\Tilde{V}|+1)}}{(|\Tilde{E}|-\frac{d}{2}(|\Tilde{E}|-|\Tilde{V}|+1))!}.
    \end{equation}
\end{corollary}
\begin{proof}
Apply  Lemma  \ref{proposition4.21},   Lemma  \ref{proposition4.22} and    Lemma \ref{lemma:upper_bound_integral}  successively, we  get 
\begin{align*}
   \sum_{w(e): \Tilde{\mathfrak{C}}(\mathfrak{D}) }\mathcal{U}^{-\frac{d}{2}}(\{w(e)\})&\le \sum_{\sum_e w(e)\le n}\mathcal{U}^{-\frac{d}{2}}(\{w(e)\}) 
    \le C^{ |\Tilde{E}|} \sum_{\forall e, w(e)\le \frac{n}{ |\Tilde{E}|}} \mathcal{U}^{-\frac{d}{2}}(\{w(e)\})
    \\&\le C^{ |\Tilde{E}|}\Big(\frac{n}{ |\Tilde{E}|}\Big)^{|\Tilde{E}|-\frac{d}{2}(  |\Tilde{E}| -  |\Tilde{V}|+1)}\int_{\alpha_e\le 1}\mathcal{U}(\{\alpha_e\})^{-\frac{d}{2}}d\alpha 
    \\&\le C^{ |\Tilde{E}|}\Big(\frac{n}{ |\Tilde{E}|}\Big)^{|\Tilde{E}|-\frac{d}{2}(  |\Tilde{E}| -  |\Tilde{V}|+1)} \cdot 18^{ |\Tilde{E}|}\delta^{-(|\Tilde{E}| -  |\Tilde{V}|+1)}\\
    &\le \frac{(Cn)^{|\Tilde{E}|-\frac{d}{2}(|\Tilde{E}|-|\Tilde{V}|+1)}}{(|\Tilde{E}|-\frac{d}{2}(|\Tilde{E}|-|\Tilde{V}|+1))!},
\end{align*} where in the last inequality 
the fact of $|\Tilde{V}|=2s-k$ and  $|\Tilde{E}|=3s-2k$ and the Stirling formula have been used.  
\end{proof}

Finally,    we are ready to   prove Theorem  \ref{thm:upper_bound_diagram_function} and Theorem \ref{thm:homogeneous_theorem}.

\begin{proof}[Proof of Theorem \ref{thm:upper_bound_diagram_function}]
Noting  that   
    \begin{equation}\label{equation:trick_d_d'}
        \mathcal{U}_{E_1}(\{w(e)\})^{-\frac{d}{2}}\le \mathcal{U}_{E_1}(\{w(e)\})^{-\frac{d'}{2}},  \quad \mathrm{for}\  d'\le d,
    \end{equation}
    we make  use   of  Lemma  \ref{proposition4.21} and Lemma \ref{lemma:upper_bound_integral} successively to obtain   
\begin{equation} \label{equ:3.40}
 \begin{aligned}
         & F_{\mathfrak{D}}(\{n_i\}_{i=1}^k)\le C^{|\Tilde{ E}|}\sum_{\forall e, w(e)\le \frac{n}{ |\Tilde{E}|}}\sum_{E_1\subset \Tilde{E}}\mathcal{U}_{E_1}(\{w(e)\})^{-\frac{d'}{2}} N^{ -  |\Tilde{E}| +  |\Tilde{V}|-1}
                \Big(\frac{N}{W^d}\Big)^{ |E_1|  - |V_1|+c(E_1)}\\
                &\le C^{|\Tilde{ E}|} \sum_{E_1\subset \Tilde{E}} \sum_{\forall e\notin E_1, w(e)\le \frac{n}{ |\Tilde{E}|}} \big(\frac{n}{ |\Tilde{E}|}\big)^{ |E_1| -\frac{d'}{2}( |E_1| - |V_1|+c(E_1))} N^{ -  |\Tilde{E}| +  |\Tilde{V}|-1}
                \Big(\frac{N}{W^d}\Big)^{ |E_1|  - |V_1|+c(E_1)},
                 \end{aligned}
             \end{equation}
             where we take a summation   over     $w(e)$ and $e\in E_1$   in the   second inequality.

    By Lemma \ref{lemma:upper_bound_integral}, we further get 
    \begin{equation} 
 \begin{aligned}
          F_{\mathfrak{D}}(\{n_i\}_{i=1}^k)&\le C^{|\Tilde{ E}|} \sum_{E_1\subset \Tilde{E}}   
         \big(\frac{n}{ |\Tilde{E}|}\big)^{|\Tilde{E}| -\frac{d'}{2}( |E_1| - |V_1|+c(E_1))} N^{ -  |\Tilde{E}| +  |\Tilde{V}|-1}
                \Big(\frac{N}{W^d}\Big)^{ |E_1|  - |V_1|+c(E_1)}\\
                &=C^{|\Tilde{ E}|} \big(\frac{n}{ |\Tilde{E}|}\big)^{|\Tilde{E}|} N^{ -  |\Tilde{E}| +  |\Tilde{V}|-1}\sum_{E_1\subset \Tilde{E}}   
                \Big(\frac{N|\Tilde{E}|^{\frac{d'}{2}}}{W^d n^{\frac{d'}{2}}  }\Big)^{ |E_1|  - |V_1|+c(E_1)}
                 \end{aligned}
             \end{equation}
     In the above summation over $E_1$,  
    the maximum   can be   attained  at either  $E_1=\Tilde{E}$ or $E_1=\emptyset$,  so by taking it    at $E_1=\Tilde{E}$ and $E_1=\emptyset$ we arrive at 
    \begin{equation}
        F_{\mathfrak{D}}(\{n_i\}_{i=1}^k)\le 2^{|\Tilde{E}|}C^{ |\Tilde{E}|}\Big( \big(\frac{n}{ |\Tilde{E}|}\big)^{ |\Tilde{E}|} N^{-  |\Tilde{E}| +  |\Tilde{V}|-1}+\big(\frac{n}{ |\Tilde{E}|}\big)^{ |\Tilde{E}|-\frac{d'}{2}(  |\Tilde{E}| -  |\Tilde{V}|+1)} W^{-d(  |\Tilde{E}| -  |\Tilde{V}|+1)}\Big).
    \end{equation}

    Finally,  by the relation $ |\Tilde{E}|=3s-2k$, $|\Tilde{V}|=2s-k$ with $s\ge k$ we see from the Stirling formula that 
    \begin{equation}
          |\Tilde{E}| ^{ |\Tilde{E}|}\ge C_{2}^{| \Tilde{E}|} ( |\Tilde{E}| )!,
    \end{equation}
    and 
    \begin{equation}
          |\Tilde{E}| ^{ |\Tilde{E}|-\frac{d'}{2}(  |\Tilde{E}| -  |\Tilde{V}|+1)} \ge C_3\cdot C_{4}^{ |\Tilde{E}|-\frac{d'}{2}(  |\Tilde{E}| -  |\Tilde{V}|+1)}( |\Tilde{E}|-\frac{d'}{2}(  |\Tilde{E}| -  |\Tilde{V}|+1))!.
    \end{equation}
   This thus completes the   proof.
\end{proof}

\begin{proof}[{Proof of Theorem \ref{thm:homogeneous_theorem}}]

Start  from the equation \eqref{equ:3.20} in Lemma 
\ref{coro:diagram_function_upper_bound} and we know that it suffices to 
seek the leading contribution in the sum via local limit theorems  in Theorem \ref{prop:gaussian_llt}. By the inequality \eqref{equ:3.40} derived in the proof of Theorem \ref{thm:upper_bound_diagram_function} and the  upper bound in Theorem \ref{thm:upper_bound_diagram_function}, we only need to consider the case $E_1=\emptyset$ in the supercritical  regime, $E_1=\Tilde{E}$ in the subcritical regime and infinitely  many $E_1$ in the critical regime, since  all the  other terms can be negligible. 

In the supercritical case,    the supercritical case of Theorem \ref{prop:gaussian_llt} shows that the leading  contribution  is
\begin{equation}
    \sum_{w(e):\mathfrak{C}(\mathfrak{D})}N^{|V|-|E|-1}=(1+o(1))
    C_{\mathfrak{D}}\mathrm{Vol}_{ \mathfrak{C}(\tau)}\Big(\Big\{\frac{n_i}{N^{\frac{1}{3}}}\Big\}\Big) N^{\frac{|E|-k}{3}}N^{|V|-|E|-1}.
\end{equation}

In the subcritical case, the subcritical case of Theorem \ref{prop:gaussian_llt} shows  the leading  contribution    is
 \begin{equation}
\begin{aligned}
    \sum_{w(e):\mathfrak{C}(\mathfrak{D})}\mathcal{U}_{\mathfrak{D}}(\{w(e)\})^{-\frac{d}{2}}&=W^{-\frac{2d}{6-d}\cdot \frac{d}{2}(|E|-|V|+1)}\sum_{w(e):\Tilde{\mathfrak{C}}(\mathfrak{D})}\mathcal{U}_{\mathfrak{D}}\Big(\big\{\frac{w(e)}{W^{\frac{2d}{6-d}}}\big\}\Big)^{-\frac{d}{2}}\\
    &=W^{-\frac{2d}{6-d}\cdot \frac{d}{2}(|E|-|V|+1)}
    (1+o(1))C_{\mathfrak{D}}
    W^{\frac{2d}{6-d}|\Tilde{E}|}
    \int_{ \mathfrak{C}(\tau)}  \big(\mathcal{U}_{\mathfrak{D}}(\{\alpha_e\})\big)^{-\frac{d}{2}} d\alpha,
\end{aligned}
\end{equation}
where a Riemann integral over a $|\Tilde{E}|$-dimensional space has been used to  approximate the Riemann sum. 
Noting that  $|\Tilde{E}|=|E|-k$, after simple calculations  we can complete the subcritical case. 

In the critical case, the proof is very similar to that in the subcritical case based on  the critical case of Theorem \ref{prop:gaussian_llt}, except for the integrals  relevant to the theta function.   
\end{proof}

\newpage
\section{Diagram expansion and convergence}\label{section:asymptotic_expansion}
In this section we are devoted to the proofs of   our main theorems in the Hermitian  case ($\beta=2$). Since the renormalization technique must be introduced to attack with the divergence from the tadpole diagram in the symmetric case, we leave the proofs in the real symmetric case to Section \ref{sec:tadpole_diagram}.

 \subsection{Transform of limit correlation  functions} \label{sec:transform of correaltion}
    In order to establish  diagram expansions for  the limit correlation  functions in the subcritical and critical regimes, we need to introduce two families of multivariate functions as transforms of  limit correlation  functions.

 \begin{definition}\label{def:sub_critical_transform}
 Let  $\mathscr{G}^{free}_{\beta,k,3}$ be a    subset of $\mathscr{G}^c_{\beta,k,3}$ in  Definition \ref{def:2.15}, which does not contain any tadpole diagram, put \begin{equation}
   \mathscr{\widetilde{G}}_{\beta,k;d} = \begin{cases}
        \mathscr{G}^c_{\beta,k,3}, & 1\leq d<4, \beta=2, \\
        \mathscr{G}^c_{\beta,k,3}, & d=1, \beta=1,\\
        \mathscr{G}^{free}_{\beta,k,3}, & 2\leq d<4, \beta=1.    
    \end{cases} 
\end{equation}  
(i) In the subcritical case, define the transform of the limit $k$-point  correlation function as 
\begin{equation}\label{sub1-point-2}
\begin{aligned}
   & \phi^{(\mathrm{sub})}_{1}(\beta,d;\tau_{1}) =
   \sum_{\mathfrak{D} \in 
         \mathscr{\widetilde{G}}_{\beta,1;d}
        }C_{\mathfrak{D}} \tau_{1}^{(3-\frac{d}{2})s-3} \int \cdots \int_{\sum_{e\in E}\alpha_e= 1} \big(\mathcal{U}_{\mathfrak{D}}(\{\alpha_e\})\big)^{-\frac{d}{2}}   \prod_{e\in {E}} d\alpha_e,\quad \tau_1>0.
\end{aligned}
\end{equation}

(ii)  In the critical case, define the transform of the limit $k$-point  correlation function as  
\begin{equation}
\begin{aligned}
    \phi_{k}^{(\mathrm{crit})}(\beta,d,\gamma;\tau_1, \ldots,\tau_k)
&=\sum_{\Pi}\prod_{P\in \Pi}  T_{\# P}^{(\mathrm{crit})}(\beta,d,\gamma; \{\tau_j\}_{j\in P}), \quad \tau_1, \ldots, \tau_k>0
\end{aligned}
	\end{equation}	
    where      
    the sum  is taken  over all partitions $\Pi$ of  \{$1,2,\cdots,k$\} and 
     \begin{equation}
    \begin{aligned}
        &T_{\{t_i\}_{i=1}^k}^{(\mathrm{crit})}(\beta,d,\gamma; \tau_1, \ldots,\tau_k)= \\
        &(\prod_{i=1}^k\tau_i^{-t_i})  \sum_{\mathfrak{D} \in 
         \mathscr{\widetilde{G}}_{\beta,k;d}
        }C_{\mathfrak{D}} 
         \gamma^{-6(s-k)} \int_{\alpha_e:\tilde{\mathfrak{C}}(\tau)}\int_{x_v\in \mathbb{T}^d}\prod_{e\in E}\theta(x_{e1}-x_{e_2},\alpha_e \Sigma)  dx d\alpha
    \end{aligned}
    \end{equation}
with $\theta(x,\Sigma)$ given in \eqref{theta} 
and the restriction $\tilde{\mathfrak{C}}$ given in Definition \ref{D-system_t}. When   all $t_i=1$, we also denote
   \begin{equation}
       T_{k}^{(\mathrm{crit})}(\beta,d,\gamma; \tau_1, \ldots,\tau_k)=T_{\{t_i\}_{i=1}^k}^{(\mathrm{crit})}(\beta,d,\gamma; \tau_1, \ldots,\tau_k).
   \end{equation}
\end{definition}

\subsection{Asymptotics for correlation moments}\label{sec:asy_mixed_moment}

\begin{theorem}\label{prop:summery_prop}
Let $k>0$ and $\tau>0$ be any fixed   integer and   real number, respectively.

\begin{itemize}
    \item[(i)](\textbf{Super-critical regime}) With the same assumption on  the bandwidth $W$ as in  \eqref{supercond2}
    , if   
    \begin{equation}
        n_1\le n_2\le \cdots \leq n_k\le \tau N^{\frac{1}{3}}
    \end{equation}
    and   $n_1+\cdots+n_k$ is  even, then  we have 
    \begin{equation}
        \mathbb{E}\big[\prod_{i=1}^{k}\tr \mathcal{P}_{n_i}(H)\big]=(1+o(1))\mathbb{E}\big[\prod_{i=1}^{k}\tr U_{n_i}\big(H_{\mathrm{GUE}}/2\big)\big].
    \end{equation}
    \item[(ii)](\textbf{Sub-critical  regime}) Assume that  $d<4$  and $W\ll L^{1-\frac{d}{6}}$.
    If  
\begin{equation}
   1\le n_1\le n_2\le \ldots \le n_k\le \tau W^{\frac{2d}{6-d}}
\end{equation}
and all   $n_1,\ldots ,n_k$ are even, then  we have
    \begin{equation}
        \mathbb{E}[\prod_{i=1}^{k}\tr \mathcal{P}_{n_i}(H)]=\big(1+o(1)\big)\prod_{i=1}^{k}\mathbb{E}[\tr \mathcal{P}_{n_i}(H))].
    \end{equation}
    \item[(iii)](\textbf{Critical regime}) Assume that $d<4$ and  $n_i=[\tau_i (\frac{L}{W})^{2}], \tau_i \in (0,\infty)$, $i=1,\ldots, k$  such that $\sum_{i=1}^k n_i$ is even. 
   If
    \begin{equation} 
 W= \big(1+o(1)\big)\gamma  L^{1-\frac{d}{6}}
    \end{equation} 
   for   some    $\gamma\in (0,\infty)$,  
         then  as $W\to \infty$   
\begin{equation}
          \mathbb{E}\Big[\prod_{i=1}^k \mathrm{Tr}\Big( \frac{1}{n_i}\mathcal{P}_{n_i}(H)\Big)\Big]=\big(1+o(1)\big) 
         \phi_{k}^{(\mathrm{crit})}(\beta,d,\gamma;\tau_1, \ldots,\tau_k),
                               \end{equation}
where $\phi_{k}^{(\mathrm{crit})}(\beta,d;\tau_1, \ldots,\tau_k)$ is given in Definition \ref{def:sub_critical_transform}. 
\end{itemize}
\end{theorem}
\begin{definition}\label{def:2.20}
  Given any partition $\pi=\{\pi_j\}_j$   of $[k]=\{1,2,\ldots,k\}$, 
we define  $\mathscr{G}^\pi_{\beta,k,3}$ 
as the collection of all diagrams   satisfying the following connectivity constraints:\\
(i) The connected components of the diagram are in bijection with the blocks 
  of the partition;\\
  (ii) For each block  $\pi_j$, all circuits labeled by indices $i\in \pi_j$ 
  must lie within the same connected component of $\bar{G}$.
\end{definition}

Introduce the concept of    the  $k$-point cluster function that  is defined via  a sum of connected typical diagram
\begin{equation}\label{def:T(n)}
    T(n_1,\cdots,n_k)=\sum_{\mathfrak{D}\in \mathscr{G}^c_{\beta,k,3}}\sum_{w:(\mathfrak{D},w)}\Sigma_{\beta}^{(2)}(\mathfrak{D},w).
\end{equation}
With Definition \ref{def:2.20},  it is easy to obtain the following  formula for cluster functions. 
\begin{proposition}

For any given  partition  $\pi$ of $[k]$, we have 
    \begin{equation}
        \sum_{\mathfrak{D}\in \mathscr{G}^\pi_{\beta,k, 3}}\sum_{w:(\mathfrak{D},w)}\Sigma_{\beta}^{(2)}(\mathfrak{D},w)=\prod_{j}T(\{n_i\}_{i\in \pi_j}).
    \end{equation}
\end{proposition}
\begin{proof}
    Since $\Sigma_{\beta}^{(2)}(\mathfrak{D},w)$ puts no restriction on random walks, we can sum over one connected component and then sum over another one, which gives the product of $T(\{n_i\})$ on the right hand side.
\end{proof}

We are ready to obtain   upper bounds and   leading calculations for the cluster functions.

\begin{theorem}\label{thm:T_upper_bound} 
Given $k\ge 1$, assume that  $W\rightarrow \infty$ and $ne^{-0.5c_\Sigma W^2}\rightarrow0$,  and $d'\le (4-\epsilon)\land d$ for any fixed $\epsilon>0$. Then there are constants $C_1, C_2$ and $C_3$ that may depend on $\epsilon$ such that 
    \begin{equation}
    \begin{aligned}
          T(n_1,\cdots,n_k)&\le  (C_3n)^k\exp\big\{C_1 n^{\frac{3}{2}}N^{-\frac{1}{2}}\big\}+(C_3n)^k\frac{N}{n^{\frac{d'}{2}}W^d}\exp\Big\{C_2n^{\frac{6-d'}{4-d'}}W^{-\frac{2d}{4-d'}}\Big\}.
    \end{aligned}
    \end{equation}
    \end{theorem}

\begin{proof}
In $\beta=2$, all diagram are regular in dimension $d'$. By Theorem \ref{thm:upper_bound_diagram_function}, we have 
    \begin{equation}
    \begin{aligned}
        T(\{n_i\})&\le CN\sum_{ \text{regular} \  k\text{-diagram }\mathfrak{D}\in \mathscr{G}_{\beta,k,3}}F_{\mathfrak{D}}(\{n_i\})\\
        &\le CN\sum_{\mathfrak{D}\in \mathscr{G}_{\beta,k,3}} \bigg\{\frac{(Cn)^{3s-2k}}{(3s-2k)!}N^{k-s-1}+\frac{(Cn)^{(3-\frac{d'}{2})s+(\frac{d'}{2}-2)k-\frac{d'}{2}}}{\big((3-\frac{d'}{2})s+(\frac{d'}{2}-2)k-\frac{d'}{2}\big)!}W^{(k-s-1)d}\bigg\}.
    \end{aligned}
    \end{equation}
    
It is sufficient to give an upper bound for the second term since the first one can be treated as a special case of $d=0$. By the Stirling formula, we see from Proposition \ref{diagrambound} that    
    \begin{equation}\label{equ:exponent_upper_bound}
        \begin{aligned}
            &\sum_{\mathfrak{D}\in \mathscr{G}_{\beta,k,3}} \frac{(Cn)^{(3-\frac{d'}{2})s+(\frac{d'}{2}-2)k-\frac{d'}{2}}}{((3-\frac{d'}{2})s+(\frac{d'}{2}-2)k-\frac{d'}{2})!}W^{(k-s-1)d}\\
            &=\sum_{s\ge k}D_{\beta,k}(s)\frac{(Cn)^{(3-\frac{d'}{2})s+(\frac{d'}{2}-2)k-\frac{d'}{2}}}{((3-\frac{d'}{2})s+(\frac{d'}{2}-2)k-\frac{d'}{2})!}W^{(k-s-1)d}\\
            &\le \frac{n^k}{(\sqrt{C n}W)^{d}}   
            \sum_{s\ge k}\frac{(C_1 s)^{s+k-1}}{(k-1)!}
            \Big(\frac{(Cn)^{3-\frac{d'}{2}}}{W^d}\Big)^{s-k}\frac{1}{\big((3-\frac{d'}{2})s+(\frac{d'}{2}-2)k-\frac{d'}{2}\big)!}\\
            &\le \frac{n^k}{(\sqrt{C n}W)^{d}}   
            \sum_{s\ge k}\frac{(C_1 s)^{s+k-1}}{(k-1)!}
            \Big(\frac{(Cn)^{\frac{6-d'}{4-d'}}}{W^\frac{2d}{4-d'}}\Big)^{(2-\frac{d'}{2})(s-k)}\frac{1}{\big((3-\frac{d'}{2})s+(\frac{d'}{2}-2)k-\frac{d'}{2}\big)!}\\
&\le \frac{n^k}{(\sqrt{C n}W)^{d}}   
            \sum_{s\ge k} 
            \Big(\frac{C_2 (Cn)^{\frac{6-d'}{4-d'}}}{W^\frac{2d}{4-d'}}\Big)^{(2-\frac{d'}{2})(s-k)}\frac{C_3}{\big((2-\frac{d'}{2})(s-k)\big)!}\\
            &\le C_4 \frac{n^k}{(\sqrt{ n}W)^{d}} 
            \exp\!\big\{C_5 n^{\frac{6-d'}{4-d'}}W^{-\frac{2d}{4-d'}}\big\}.
        \end{aligned}
    \end{equation}
\end{proof}

\begin{proposition}\label{prop:5.5}
   Given    $k\ge 1$ and  $d\ge 1$, with the same assumption on  the bandwidth $W$ as in  \eqref{supercond2}, if    
    \begin{equation}
       \tau _1N^{\frac{1}{3}}\le n_1\le n_2\le \cdots \leq n_k\le \tau_2N^{\frac{1}{3}}
    \end{equation} for  any fixed positive constant $\tau_1,\tau_2$
    and   $n_1+\cdots+n_k$ is even, then we have 
    \begin{equation}
   T(n_1, \ldots, n_k)=(1+o(1)) \big(\prod_{i=1}^k{ n_i}\big)\, \phi_k\Big(\Big\{\frac{n_i}{W^{\frac{2d}{6-d}}}\Big\}\Big).
\end{equation}
    Here $\phi_k(z_1, \ldots,z_k)=\sum_{s\ge k}g_{k,s}(z_1, \ldots,z_k)$, and   $g_{k,s}$  is a homogeneous function of degree $3(s-k)$.
\end{proposition}

\begin{proposition}\label{prop:T_critical}
    Given  $k\ge 1$ and $d<4$,  if $W= [\gamma L^{1-\frac{d}{6}}]$, for any    fixed constants  $\tau_i>0$ and $n_i=[\tau_i(\frac{L}{W})^2]$,
such that $n_1+\cdots+ n_k$ is even, then we have
\begin{equation}
    T(n_1, \ldots, n_k)=\big(1+o(1)\big)\big(\prod_{i=1}^k{ n_i}\big)\,   \varphi_k\Big(\gamma,\{\tau_i\}_{i=1}^{k}\Big).
\end{equation}
    Here $\varphi_k(\gamma,\{z_1, \ldots,z_k\})=\sum_{s\ge k}\gamma^{6(k-s)}f_{k,s}(z_1, \ldots,z_k)$.
\end{proposition}

\begin{proposition}\label{prop:5.4} Given  $k\ge 1$ and $d<4$,  if $W\ll L^{1-\frac{d}{6}}$, for any    fixed constant  $\tau>0$ and  for any 
\begin{equation}
   1\le n_1\le n_2\le \cdots \le n_k\le  \tau W^{\frac{2d}{6-d}}
\end{equation}
such that $n_1+\cdots+ n_k$ is even, then we have
\begin{equation}
    T(n_1, \ldots, n_k)=\big(1+o(1)\big)\frac{N}{W^{\frac{6d}{6-d}}} \big(\prod_{i=1}^k{ n_i}\big)\,   \psi_k\Big(\Big\{\frac{n_i}{W^{\frac{2d}{6-d}}}\Big\}\Big) .
\end{equation}
    Here $\psi_k(z_1, \ldots,z_k)=\sum_{s\ge k}h_{k,s}(z_1, \ldots,z_k)$, and $h_{k,s}$ is a homogeneous function of degree $\frac{1}{2}\big((6-d)s-(4+d)k-d\big)$.
\end{proposition}

\begin{proof}[Proof of Proposition \ref{prop:5.5}, \ref{prop:T_critical} and \ref{prop:5.4}]
    Recall the diagram function in \eqref{equ:3.1}, 
    \begin{equation}
        F_{\mathfrak{D}}(\{n_i\}_{i=1}^k)=\frac{1}{N}\sum_{w(e):(\mathfrak{D},w)}\Sigma_{\beta}^{(2)}(\mathfrak{D},w).
    \end{equation} 
   By  Corollary \ref{coro:2.19},  it suffices for us to just consider   
  the set of maximal $k$-diagrams, denoted by   $\mathcal{M}_{\mathrm{diagram}}$.
By  Lemma \ref{lem:2.21}, we know that for any   $\mathfrak{D}\in \mathscr{G}^{c}_{\beta,k,\ge 3}\backslash \mathscr{G}_{\beta,k,3}$ and $\mathfrak{D}\in   \mathcal{M}_{\text{diagram}}$, the diagram function can be  bounded by some $\mathfrak{D}'\in \mathscr{G}_{\beta,k,3}$, with new edges of   weight $1$ or 2, which implies   $F_{\mathfrak{D}}(\{n_i\})=o(F_{\mathfrak{D}'}(\{n_i\}))$. Hence we obtain    
  \begin{equation}\label{equ:5.14}
        \sum_{\mathfrak{D}\in \mathscr{G}^{c}_{\beta,k,\ge 3}}F_{\mathfrak{D}}(\{n_i\}_{i=1}^k)\le \sum_{\mathfrak{D'}\in \mathscr{G}^{c}_{\beta,k, 3}}C^{|E'|}F_{\mathfrak{D'}}(\{n_i\}_{i=1}^k).
        \end{equation}
    Here $C^{|E|}$ comes from the constant in Lemma \ref{lem:2.21} and the fact that  the number of $\mathfrak{D}'$ is bounded by $2^{|E|}$ for any  given $\mathfrak{D}$.
    
    On the other hand,  the upper bound of Theorem \ref{thm:T_upper_bound} shows that  we only need to consider all finite diagrams since the sum of  large-$s$ diagram can be negligible. Use the similar argument as in the derivation of   \eqref{equ:5.14}  to obtain 
    \begin{equation}\label{equ:T=Fsum}
        T(n_1, \ldots,n_k)=(1+o(1))N\sum_{\mathfrak{D}\in \mathscr{G}^c_{\beta,k, 3}}F_{\mathfrak{D}}(\{n_i\}_{i=1}^k).
    \end{equation}
    
    By Theorem \ref{thm:homogeneous_theorem},   in the subcritical case  as shown in  Proposition \ref{prop:5.4}  each diagram function of typical diagram tends to  a homogeneous function with homogeneous degree $(3-\frac{1}{2}d)s- (2+\frac{1}{2}d)k-\frac{1}{2}d$. So after simple algebraic  manipulations      we complete the proof of  Proposition \ref{prop:5.4}. 
    
    Similarly, we can complete  the proof of Proposition \ref{prop:5.5} amd \ref{prop:T_critical}.
\end{proof}

\begin{proposition}\label{prop:5.6}
With  the same assumptions as in   Proposition \ref{prop:5.5}, \ref{prop:T_critical} or Proposition \ref{prop:5.4}, we have
    \begin{equation}
    \begin{aligned}
        \mathbb{E}\big[\prod_{i=1}^k\tr\mathcal{P}_{n_i}(H)\big]
        &=\big(1+o(1)\big)\sum_{\pi}\prod_{j}T(\{n_i\}_{i\in \pi_j}).
        \end{aligned}
    \end{equation}
\end{proposition}
\begin{proof}
    By Theorem \ref{thm:thm2.5}, we have 
    \begin{equation}
         \mathbb{E}\big[\prod_{i=1}^k\tr\mathcal{P}_{n_i}(H)\big]=(1+o(1)) \mathbb{E}\big[\prod_{i=1}^k\tr V_{n_i}(H)\big].
    \end{equation}
    
  Recalling the diagram function in \eqref{equ:3.1},  by  Corollary \ref{coro:2.19} and Proposition \ref{coro:3.7},   we can use the same argument as in the proof of Proof of Proposition \ref{prop:5.4} and \ref{prop:5.5}  to derive 
    \begin{equation}
      \begin{aligned}
        \mathbb{E}\big[\prod_{i=1}^k\tr V_{n_i}(H)\big]&=(1+O(\mathcal{E}(n)))\sum_{\mathfrak{D}\in \mathscr{G}_{\beta,k,\ge 3}}F_{\mathfrak{D}}(\{n_i\}_{i=1}^k)\\
        &=(1+o(1))\sum_{\pi}\prod_{j}T(\{n_i\}_{i\in \pi_j}).
          \end{aligned}
    \end{equation} 
       \end{proof}

Now we are ready to complete the proof of Theorem \ref{prop:summery_prop}.
\begin{proof}[Proof of Theorem \ref{prop:summery_prop}]
    We only  prove the super-critical and sub-critical cases here. The proof in the critical case is similar but with lots of redundant discussions.
We start from the asymptotic estimate in   Proposition \ref{prop:5.6}.  

In the  super-critical case,  by Proposition \ref{prop:5.5} we know that  the function $T(\{n_i\})$ is the same as in the special   case of $W=L/2$. 
However, 
in the sub-critical case, by Proposition \ref{prop:5.4} we know that 
    \begin{equation}
        \begin{aligned}
        \mathbb{E}\big[\prod_{i=1}^k\tr\mathcal{P}_{n_i}(H)\big]
        &=\Big(1+o(1)+O\big(N^{-1}W^{\frac{6d}{6-d}}\big)\Big)\prod_{i=1}^kT(n_i)\\
        &=(1+o(1))\prod_{i=1}^k\mathbb{E}\big[\tr\mathcal{P}_{n_i}(H)\big].
        \end{aligned}
    \end{equation}

    This thus   completes the proof.
\end{proof}

\subsection{Asymptotics for moments of polynomial powers  }\label{sec:P_powers}
The relevant results for the expectation $\mathbb{E}[\prod_{i=1}^k\tr\mathcal{P}_{n_i}(H)]$ should be  enough and clear to detect the three-phase transition of edge statistics, however, in order to apply the convergence theorem via the limit  transform resulting from  the polynomials we need to further study   the mixed moment $\mathbb{E}[\prod_{i=1}^k\tr(\mathcal{P}_{n_i}(H))^{t_i}]$.
In this subsection, we highlight two central results—Theorem  \ref{prop:super_4} and \ref{prop:sub_4} below—which establish fundamental properties of the mixed moments.
 
\begin{theorem}\label{prop:super_4} Given any fixed integers  $k\ge 1$ and  $t_i\ge 1$ for $1\le i\le k$,  with the same assumption on  $W$ as in  Theorem \ref{metatheoremsup}, if  
    \begin{equation}
        n_1\le n_2\le \cdots \leq n_k\le \tau N^{\frac{1}{3}}
    \end{equation}
    for any fixed constant $\tau$  and  $t_1n_1+\cdots+t_kn_k$ is even, then   
    \begin{equation}
        \mathbb{E}\big[\prod_{i=1}^{k}\tr (\mathcal{P}_{n_i}(H))^{t_i}\big]=\big(1+o(1)\big)\mathbb{E}\big[\prod_{i=1}^{k}\tr (U_{n_i}(H_{\mathrm{GUE}}/2))^{t_i}\big].
    \end{equation}
\end{theorem}
\begin{theorem}\label{prop:sub_4} 
Given any fixed integers  $k\ge 1$ and  $t_i\ge 1$ for $1\le i\le k$, assume that     $d<4$ and   $W\ll L^{1-\frac{d}{6}}$. If for any fixed constant $\tau$  and 
\begin{equation}
   1\le n_1\le n_2\le \cdots \le n_k\le \tau W^{\frac{2d}{6-d}}
\end{equation}
and all  $t_1n_1,\ldots ,t_kn_k$ are  even, then  
    \begin{equation}
        \mathbb{E}\big[\prod_{i=1}^{k}\tr (\mathcal{P}_{n_i}(H))^{t_i}\big]=\big(1+o(1)\big)\prod_{i=1}^{k}\mathbb{E}\big[\tr (\mathcal{P}_{n_i}(H))^{t_i}\big].
    \end{equation}
\end{theorem}

\begin{theorem}\label{prop:critical_4}
    Given  $k\ge 1$ and $d<4$,  if $W= [\gamma L^{1-\frac{d}{6}}]$, for any    fixed constants  $\tau_i>0$ and $n_i=[\tau_i(\frac{L}{W})^2]$,
such that $n_1+\cdots+ n_k$ is even, then we have
\begin{equation}
    \mathbb{E}\big[\prod_{i=1}^{k}\frac{1}{n_i}\tr (\mathcal{P}_{n_i}(H))^{t_i}\big]=(1+o(1))\phi_{\{t_i\}}^{(\mathrm{crit})}(\beta,d,\gamma;\tau_1, \ldots,\tau_k).
\end{equation}
Here $\phi_{\{t_i\}}^{(\mathrm{crit})}(\beta,d,\gamma;\tau_1, \ldots,\tau_k)$ is defined in Definition \ref{def:sub_critical_transform}.
\end{theorem}

Since most of the proof ingredients have already been established in the previous sections, we first outline the key steps of the proof for the case when $k=1$.

\begin{itemize}
    \item\textbf{Reduction to $\mathbb{E}[\tr (V_n)^{t}]$.} Following the steps in Section \ref{sec:reduce_to_Vn}, we can prove $\mathbb{E}[\tr (\mathcal{P}_{n}(H))^{t}]=(1+o(1))\mathbb{E}[\tr (V_n)^{t}]$. 
    \item\textbf{Reduction to multi-tail diagrams.} There are possible backtracking paths between $V_n$ and itself, so when we adopt the diagram reduction, there might be many tail edges.  For instance,  in the case of $\mathbb{E}[\tr (\mathcal{P}_{n}(H))^{t}]$ there are $t$ tail edges; see Figure \ref{fig:leg_diagram} and Figure \ref{fig:tree_diagram} for   illustration. 
    \item\textbf{Diagram function limit and integral analysis.} By the singularity analysis and the criterion in Proposition \ref{propositionconvergence}, we see after the tadpole renormalization, that all other  diagrams are regular and hence the corresponding integral  is finte and well-defined. Following  the proof of Theorem \ref{thm:homogeneous_theorem}, we can obtain the diagram-wise limit for the diagram functions. Specifically, in the supercritical case, similar to Proposition \ref{prop:5.5} we know that the resulting diagram function has  the same asymptotics as in  the GUE case  of $W={L}/{2}$. 
    \item\textbf{Diagram functions are summable.} This result is analogous to Theorem  \ref{thm:T_upper_bound}.
\end{itemize}
\begin{proposition}\label{prop:Pt=Vt}
   Given integers $k\geq 1$ and   $t_1, \ldots, t_k\geq 1$,   let $n=t_1n_1+\cdots+ t_kn_k$. If $n W^{-d} \to 0
    $ as $W\to \infty$, then
    \begin{equation}
        \mathbb{E}\Big[\prod_{i=1}^k \tr(\mathcal{P}_{n_i}(H))^{t_i}\Big]=\Big(1+O\big(\frac{n}{W^d}\big)\Big)\mathbb{E}\Big[\prod_{i=1}^k \tr (V_{n_i})^{t_i}\Big].
    \end{equation}
\end{proposition}
\begin{proof}
    We proceed in much the    same way  as in   Theorem \ref{thm:thm2.5} to complete it. 
\end{proof}

Now we describe the key procedure to reduce $\mathbb{E}[\tr (V_n)^{t}]$ into multi-tail diagrams. We proceed as in almost the same way in  Section \ref{sec:diagram_reduction} to  deal with  the diagram representation for the path in $\mathbb{E}[\tr V_{l_0}V_{l_1}\cdots V_{l_t}]$.  However, there are possible back-trackings in the product $V_{l_i}V_{l_{i+1}}$.

We first perform a zip-up operation on the backtracking edges, as illustrated in Figure \ref{fig:zipup},
based on   the identity  
\begin{equation}
    1=\sum_{i=1}^{s}\mathbb{1}(x_{t-i}\ne x_{t+i})\prod_{j=0}^{i-1}\mathbb{1}(x_{t-j}= x_{t+j})+\prod_{j=0}^{s}\mathbb{1}(x_{t-j}= x_{t+j}), \quad s\leq t.
\end{equation}
We begin by detailing the zip-up operation applied to the product term  $V_{l_0}V_{l_1}$. In this context, the role of the non-backtracking event is expressed as 
\begin{equation}
    \prod_{i=0}^{l_0-2}\mathbb{1}(x_i\ne x_{i+2})\prod_{i=l_0}^{l_0+l_1-2}\mathbb{1}(x_i\ne x_{i+2}).
\end{equation}

Inserting the identity 
\begin{equation}
        1=\sum_{t=1}^{l_0\land l_1}\mathbb{1}(x_{l_0-t}\ne x_{l_0+t})\prod_{j=0}^{t-1}\mathbb{1}(x_{l_0-j}= x_{l_0+j})+\prod_{j=0}^{l_0\land l_1}\mathbb{1}(x_{l_0-j}= x_{l_0+j}),
    \end{equation}
we derive 
\begin{equation}
    \begin{aligned}
        &\prod_{i=0}^{l_0-2}\mathbb{1}(x_i\ne x_{i+2})\prod_{t=l_0}^{l_0+l_1-2}\mathbb{1}(x_t\ne x_{t+2})=\prod_{i=0}^{l_0-2}\mathbb{1}(x_i\ne x_{i+2})\\
        &\times \prod_{i=l_0}^{l_0+l_1-2}\mathbb{1}(x_i\ne x_{i+2}) \bigg(\sum_{t=1}^{l_0\land l_1}\mathbb{1}(x_{l_0-t}\ne x_{l_0+t})\prod_{j=0}^{t-1}\mathbb{1}(x_{l_0-j}= x_{l_0+j})+\prod_{j=0}^{l_0\land l_1}\mathbb{1}(x_{l_0-j}= x_{l_0+j}) \bigg)\\
        &=\sum_{t=1}^{l_0\land l_1}\mathbb{1}(x_{l_0-t}\ne x_{l_0+t})\prod_{i=0}^{l_0-t-1}\mathbb{1}(x_i\ne x_{i+2})\prod_{i=l_0+t-1}^{l_0+l_1-2}\mathbb{1}(x_i\ne x_{i+2})\prod_{j=0}^{t-1}\mathbb{1}(x_{l_0-j}=x_{l_0+j})\\
        &~~~~~~~~~~~~~~~~\cdot \prod_{i=l_0-t}^{l_0-2}\mathbb{1}(x_i\ne x_{i+2}) \prod_{i=l_0}^{l_0+t-2}\mathbb{1}(x_i\ne x_{i+2})\\
        &~~~+\prod_{i=0}^{l_0-2}\mathbb{1}(x_i\ne x_{i+2})\prod_{i=l_0}^{l_0+l_1-2}\mathbb{1}(x_i\ne x_{i+2})\prod_{j=0}^{l_0\land l_1}\mathbb{1}(x_{l_0-j}= x_{l_0+j}).
    \end{aligned}
\end{equation}
The factor 
\begin{equation}
    \sum_{t=1}^{l_0\land l_1}\mathbb{1}(x_{l_0-t}\ne x_{l_0+t})\prod_{i=0}^{l_0-t-1}\mathbb{1}(x_i\ne x_{i+2})\prod_{i=l_0+t-1}^{l_0+l_1-2}\mathbb{1}(x_i\ne x_{i+2})\prod_{j=0}^{t-1}\mathbb{1}(x_{l_0-j}=x_{l_0+j})
\end{equation}
enforces a non-backtracking path while retaining a single backtracking tail edge, as depicted in Figure \ref{fig:2zip_up}. 
For the product  of the form   $V_{l_0}V_{l_1}\cdots V_{l_t}$, the resulting structure evolves into a non-backtracking path with $t$ backtracking tail edges, illustrated in Figure \ref{fig:zipup}.

\begin{figure}
    \centering
    \includegraphics[width=0.6\textwidth]{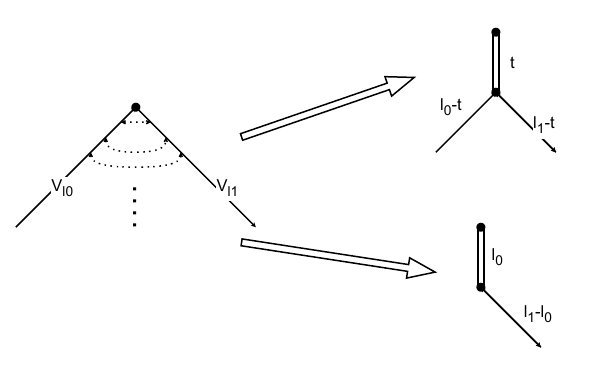}
    \caption{Zip-up operation for $V_{l_0}V_{l_1}$}
    \label{fig:2zip_up}
\end{figure}
\begin{figure}
    \centering
    \includegraphics[width=0.6\textwidth]{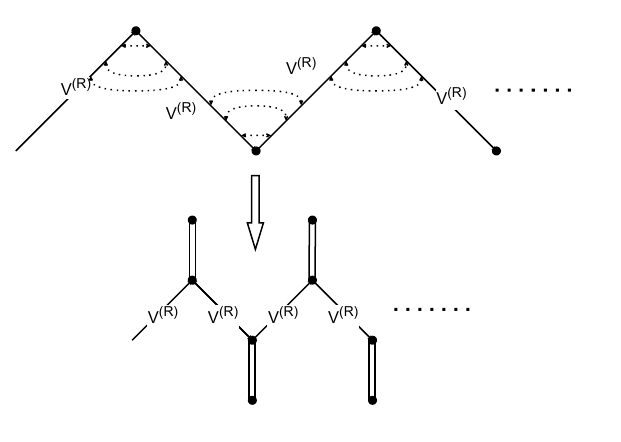}
    \caption{Zip-up operation}
    \label{fig:zipup}
\end{figure}

We now analyze the expectation   $\mathbb{E}[\tr V_{l_0}V_{l_1}\cdots V_{l_t}]$. The diagrammatic reduction technique introduced in Section \ref{sec:diagram_reduction} can be applied to non-backtracking paths. The resulting diagrams correspond to those given in Definition \ref{def:diagram}, but with $t$ additional tail edges; see Figure \ref{fig:leg_diagram} for an illustrative example. 

\begin{definition}\label{def:(k,t)diagram}Let $\beta\in \{1,2\}$.
\begin{itemize}
\item A \textbf{\textit{{$(k,\{t_i\}_{i=1}^k)$-diagram}}} of type $\beta$ is an (undirected) multi-graph $\bar{G} = (\bar{V}, \bar{E})$,
together with a $k$-tuple of circuits 
\begin{equation} \bar{\mathfrak{p}} = \bar{u}_0^1 \bar{u}_1^1 \cdots  \bar{u}_0^1 ,\,\,\,
             \bar{u}_0^2 \bar{u}_1^2 \cdots   \bar{u}_0^2 ,\,\,\,
             \cdots ,\,\,\, \bar{u}_0^k \bar{u}_1^k \cdots   \bar{u}_0^k \end{equation}
on $\bar{G}$, such that
\begin{itemize}
\item The $i$-th circuit of $\bar{\mathfrak{p}}$  is  {\em $t_i$-backtracking}, that is, in this circuit exactly $t_i$ edges 
are followed by their reverses unless in $\beta=1$ and the edge is $\bar{u}\bar{u}$; 
\item For every loop $e_0= (\bar{u}, \bar{u}) \in \bar{E}$,
\begin{equation} 
\# \left\{ (i,j) \, | \, (\bar{u}_j^i, \bar{u}_{j+1}^i) =e_0 \right\}  = 2,
     \end{equation}
and for every $
e= (\bar{u}, \bar{v}) \in \bar{E}$ with $\bar{u}\ne \bar{v}$,
\begin{equation}\begin{split}
&\# \left\{ (i,j) \, | \, (\bar{u}_j^i, \bar{u}_{j+1}^i) =e \right\} + \# \left\{ (i,j) \, | \, (\bar{u}_{j+1}^i, \bar{u}_{j}^i) = e \right\} = 2
    \quad (\beta = 1)~, \\
&\# \left\{ (i,j) \, | \, (\bar{u}_j^i,\bar{u}_{j+1}^i) = e \right\} = \# \left\{(i,j) \, | \, (\bar{u}_{j+1}^i, \bar{u}_{j}^i) = e \right\} = 1
    \quad (\beta = 2)~;
\end{split}\end{equation}
\item For $i=1,\ldots, k$, the $i$-th circuit has exact $t_i$ vertices of degree 1 and the degree of all other vertices is not less than 3. 
These   are called as   \textit{\textbf{tail vertices}} while the edges with endpoint of tail vertices are \textbf{\textit{tail edges}}.
\end{itemize}
\item A \textbf{\textit{{weighted $(k,\{t_i\}_{i=1}^k)$-diagram}}} is a $(k,\{t_i\}_{i=1}^k)$-diagram $\bar{G}$    assigned with a weight function
$\bar{w}: \bar{E} \to \{0,1,2,\cdots\}$,  in which  $\bar{w}(e)=0$ is  possible  only when $e \in \bar{E}$ is a tail edge.
\item   For simplicity,
we still  denote  by  $\mathfrak{D}=(\bar{G},\bar{\mathfrak{p}})$   as a $(k,\{t_i\}_{i=1}^k)$-diagram and  by the pair $(\mathfrak{D},w)$  as a weighted $(k,\{t_i\}_{i=1}^k)$-diagram.
\end{itemize}
\end{definition}

\begin{definition}\label{def:4.7}
    \begin{itemize}
        \item  A \textbf{\textit{typical (weighted) $(k,\{t_i\}_{i=1}^k)$-diagram}} $\bar{G} = (\bar{V}, \bar{E})$ is a $k$-diagram with $t$-tail edges, in which   the degree of tail vertex in $\bar{G}$ is 1 and the degree of all the other vertices is equal to $3$. Let  $\mathscr{G}_{\beta,k,\{t_i\},3}$ be the set of typical (weighted) $(k,\{t_i\})$-diagrams  of type $\beta$.
        \item A diagram $\mathfrak{D}=(\bar{G},\bar{\mathfrak{p}})$ is \textbf{\textit{connected}} if the  graph $\bar{G}$ is connected.  Let   $\mathscr{G}^c_{\beta,k,\{t_i\},3}$ be the set of connected typical (weighted) $(k,\{t_i\}_{i=1}^k)$-diagrams of type $\beta$ and   $\mathscr{G}^c_{\beta,k,\{t_i\},\ge 3}$ be  the set of connected general (weighted) $(k,\{t_i\}_{i=1}^k)$-diagrams.
    \end{itemize}
\end{definition}
We also introduce the linear system $\mathfrak{C}$ for typical $(k,\{t_i\}_{i=1}^k)$-diagram.
\begin{definition}[Linear $\mathfrak{D}$-system] \label{D-system_t}
  A family of non-negative integers $ \{ w({{e}}) \}_{{e}\in E}$
  is said to  satisfy a system of linear equations  associated with the typical $k$-diagram $\mathfrak{D}$  if   
		\begin{equation}  
		\mathfrak{C}(\mathfrak{D}): \quad \sum_{{e}\in {E}}c_i({e})w({{e}})=n_i,\quad i=1,...,k,  
		\end{equation}
		where  $c_i(e)\in \{0,1,2\}$ is the number of times that ${e}$ occurs in the $i$-th piece of non-backtracking path in ${\mathfrak{p}}$, and
\begin{equation}
    w(e)\ge\begin{cases}
       3,~~~~\text{if $e$ is a self-loop};\\
        1,~~~~\text{if $e$ is not a tail edge}.
   \end{cases} 
\end{equation} 
Similarly, a system of   inequalities, by removing the set $T$ of  all tail edges and denoting $\{\pi_j\}$ to be the partition of non-backtracking piece of $\bar{\mathfrak{p}}$, 
\begin{equation} 
   \widetilde{\mathfrak{C}}(\mathfrak{D}): \quad \begin{cases}
        &\sum_{{e}\in {E}\backslash T}c_i({e})w({{e}})\le n_i,\qquad i=1,...,k,\\
        &\sum_{i\in \pi_j}\sum_{{e}\in {E}\backslash T}c_i({e})w({{e}})\equiv n_i\pmod 2,\qquad i=1,...,k,\\
        &w(e)\ge 3,~~~~\text{if $e\in E\backslash T$ is a self-loop},\\
        &w(e)\ge 1,~~~~\text{if $e\in E\backslash T$}.
    \end{cases}
\end{equation}
we also introduce the simplex formed by systems of linear inequalities  in $\mathbb{R}_{+}^{|E|-k}$      
  \begin{equation}
        \widetilde{\mathfrak{C}}(\tau): \quad  \sum_{e\in {E}\backslash T_i}c_i(e)\alpha_e\le \tau_i,  \qquad i=1,...,k,
    \end{equation}
where      $c_i(e)\in \{0,1,2\}$ and $T_i$ is  the tail edge in the $i$-th piece of $\bar{\mathfrak{p}}$.

\end{definition}
\begin{figure}
    \centering
    \includegraphics[width=0.3\linewidth]{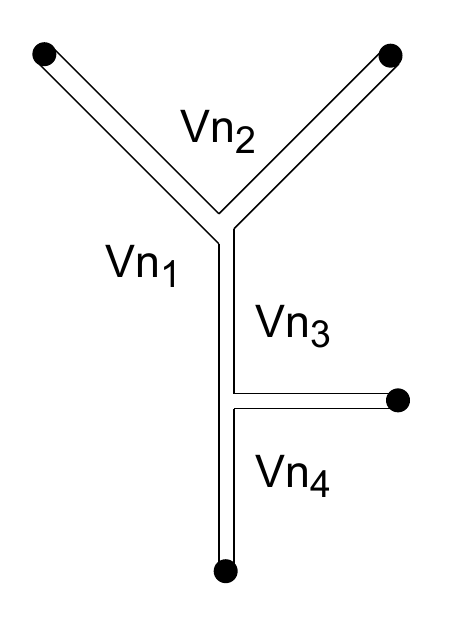}
    \caption{An example for $(1,4)$-diagram. The black vertices are tail vertices.}
    \label{fig:tree_diagram}
\end{figure}

There are possible tree diagrams since there are $\{t_i\}$ back-trackings are allowed, see Figure \ref{fig:tree_diagram} for an example. We extend the definition of diagram function $F_{\mathfrak{D}}(\{(n_i,t_i)\}_{i=1}^k)$ in Definition \ref{diagramf} to $(k,\{t_i\}_{i=1}^k)$-diagrams as
\begin{equation}
        F_{\mathfrak{D}}(\{(n_i,t_i)\}_{i=1}^k)=\frac{1}{N}\sum_{w(e):(\mathfrak{D},w)} 
        \sum_{\iota:V\rightarrow \Lambda_L}\prod_{e\in E}p_{w(e)}(\iota(u^{e}),\iota(v^{e})).
\end{equation}
We denote the sum of connected typical diagram again as
\begin{equation}\label{def:T(n,t)}
    T\big(\{(n_i,t_i)\}_{i=1}^k\big)=\sum_{\mathfrak{D}\in \mathscr{G}^c_{\beta,k,3}}\sum_{w:(\mathfrak{D},w)}\Sigma_{\beta}^{(2)}(\mathfrak{D},w).
\end{equation}
\begin{proposition}\label{prop:T_t=Fsum}Under the same  assumptions as in   Theorem \ref{prop:super_4} or Theorem \ref{prop:sub_4}, we have
    \begin{equation}
        T(\{(n_i,t_i)\}_{i=1}^k)=(1+o(1))N\sum_{\mathfrak{D}\in \mathscr{G}^c_{\beta,k,\{t_i\}, 3}}F_{\mathfrak{D}}(\{(n_i,t_i)\}_{i=1}^k).
    \end{equation}
\end{proposition}
\begin{proof} 
    The proof is exactly the same as  in the derivation of \eqref{equ:T=Fsum} in the proof of  Proposition \ref{prop:5.4} and \ref{prop:5.5}.
\end{proof}

\begin{proposition}\label{thm:T_t_upper_bound}
 Given integers $k\geq 1$ and   $t_1, \ldots, t_k\geq 1$,  let $n=\sum_{i=1}^k t_in_i$ and $t=\sum_{i=1}^k t_i$.
If 
  $W\rightarrow \infty$ and $ne^{-0.5c_\Sigma W^2}\rightarrow 0$,  and $d'\le (4-\epsilon)\land d$ for any fixed $\epsilon>0$,   then 
    \begin{equation}
    \begin{aligned}
        T\big(\{(n_i,t_i)\}_{i=1}^k\big)&\le  (Cn)^t\exp\big\{C_1 n^{\frac{3}{2}}N^{-\frac{1}{2}}\big\}+(Cn)^t\frac{N}{n^{\frac{d'}{2}}W^d}\exp\Big\{C_2n^{\frac{6-d'}{4-d'}}W^{-\frac{2d}{4-d'}}\Big\}+\delta_{k,1}(Cn)^{t-3}N.
    \end{aligned}
    \end{equation}
\end{proposition}
\begin{proof}[Proof of Proposition \ref{thm:T_t_upper_bound}]
    The proof is very similar to  that of   Proposition \ref{thm:T_upper_bound}. We compute the upper bound of diagram functions in the super-critical and sub-critical cases. Note that there are no singular diagram if $d'< 4$ in $\beta=2$ case. First note that we can remove the extra $t-k$ tail edges on the circuit so that the $(k,\{t_i\}_{i=1}^k)$-diagram $\mathfrak{D}$ (unless $\mathfrak{D}$ is a tree diagram) can be mapped to a $k$-diagram $\mathfrak{D}_0$ with $|E_0|=3s-k,|V_0|=2s$. Now put the $t-k$ tail edges back, we have $|E|=3(s-k)+2t, |V|=2(s-k)+2t$. Here it is worth stressing  that each time we add a tail edge we get two new edges and two new vertices.

    In the super-critical case, we have $t$ restrictions on the edge weights, so   the resulting upper bound is
    \begin{equation}
        \frac{n^{|E|-t}}{(|E|-t)!}N^{|V|-|E|}=\frac{n^t}{(3(s-k))!}\Big(\frac{n^3}{N}\Big)^{s-k}.
    \end{equation}
    In  the sub-critical case, by Corollary \ref{proposition:poly_upper_bound},  change $|\Tilde{E}|=3(s-k)+t,~|\Tilde{V}|=2(s-k)+t$, and  we have the upper bound  
    \begin{equation}
        W^{-d(|E|-|V|+1)}\frac{(Cn)^{|\Tilde{E}|-\frac{d'}{2}(|\Tilde{E}|-|\Tilde{V}|+1)}}{(|\Tilde{E}|-\frac{d'}{2}(|\Tilde{E}|-|\Tilde{V}|+1))!}=\frac{(Cn)^t}{n^{\frac{d'}{2}}W^{d}}\frac{(Cn^{3-\frac{d'}{2}}W^{-d})^{s-k}}{((3-\frac{d'}{2})(s-k)+t-\frac{d'}{2})!}.
    \end{equation}
    Note that for tree diagrams we have $s=0$, and for $k>1$ and connected diagrams we have $s\ge k$ (since there must be loop in the underlying graph hence $3s-k=|E|\ge |V|=2s$).  Also note that the number of diagrams with parameter $s$ is bounded by 
    \begin{equation}
        \frac{1}{(t-k)!}|E|(|E|+2)(|E|+4)\cdots (|E|+2(t-k-1))\, D_{k,\beta}(s)\le \frac{(C s)^{s+t-1}}{(t-1)!},
    \end{equation}
    where $D_{\beta,k}(s)$ is the number of typical $k$-diagrams with $|\bar{E}|=2s$ and $|\bar{V}|=3s-k$ defined in Proposition \ref{diagrambound}.

Thus, similar to the proof of Proposition \ref{thm:T_upper_bound}, we derive the first two terms in Proposition \ref{thm:T_t_upper_bound}. Turning to  the case of $k=1$ and $s=0$, the tree diagram, we obtain the last term $\delta_{k,1}(Cn)^{t-3}N$. 

This completes  the proof.
\end{proof}

\begin{proof}[Proofs of Theorem \ref{prop:super_4}-\ref{prop:critical_4}]
    The proof is similar to that  of Theorem \ref{prop:summery_prop}. In the super-critical case, by Proposition \ref{prop:T_t=Fsum} we  have  
        \begin{equation}
        T\big(\{(n_i,t_i)\}_{i=1}^k\big)=(1+o(1))N\sum_{\mathfrak{D}\in \mathscr{G}^c_{\beta,k,\{t_i\}, 3}}F_{\mathfrak{D}}(\{(n_i,t_i)\}_{i=1}^k).
    \end{equation}
    While in the super-critical case, the asymptotics of the diagram function is irrelevant to the bandwidth $W$, which  can be taken as  ${W}= {L}/{2}$. In this case,    the model reduces to a mean-field   Wigner matrix    with unimodular entries. By \cite[{Theorem I.5.3.}]{feldheim2010universality}, via   Chebyshev polynomials,  both the unimodular matrix and  GUE matrix have the same asymptotics. Moreover,   all $T(\{(n_i,t_i)\}_{i=1}^k)$ have the same asymptotics. Using the $T(\{(n_i,t_i)\}_{i=1}^k)$ version of Proposition \ref{prop:5.6}
        \begin{equation}\label{equ:P=T}
    \begin{aligned}
        \mathbb{E}[\prod_{i=1}^k\tr(\mathcal{P}_{n_i}(H))^{t_i}]
        &=(1+o(1))\sum_{\pi}\prod_{j}T(\{(n_i,t_i)\}_{i\in \pi_j}),
        \end{aligned}
    \end{equation}
   we know that  each term has the same asymptotics  as in the  GUE case. This  thus  proves Theorem \ref{prop:super_4}.

    Now we consider the sub-critical case, where we take $d'=d$. Put  $\pi=\{\pi_j\}_{j=1}^l$,  we see by Theorem \ref{thm:T_t_upper_bound}
    \begin{equation}
        \prod_{j}T(\{(n_i,t_i)\}_{i\in \pi_j})=O\big( n^t (N n^{-\frac{d}{2}} W^{-d} )^l \big).
    \end{equation} 
    Note that  $N n^{-\frac{d}{2}} W^{-d} \rightarrow \infty$ in the sub-critical case,  so only the case $l=k$     contributes the main term in \eqref{equ:P=T}. This means
    \begin{equation}
    \begin{aligned}
        \mathbb{E}[\prod_{i=1}^k\tr(\mathcal{P}_{n_i}(H))^{t_i}]
        &=(1+o(1))\prod_{i=1}^kT(n_i,t_i) 
         =(1+o(1))\prod_{i=1}^k\mathbb{E}[\tr(\mathcal{P}_{n_i}(H))^{t_i}].
        \end{aligned}
    \end{equation}
    This gives a proof of   Theorem \ref{prop:sub_4}. 
    
    The proof of Theorem \ref{prop:critical_4} is similar by using Proposition \ref{prop:T_t=Fsum} and Theorem \ref{thm:homogeneous_theorem}.
\end{proof}

Proposition \ref{thm:T_t_upper_bound} provides   a   sharp upper bound  that is of independent interest.  As an immediate corollary, we obtain right-tail decay estimates   for the top eigenvalue in $\beta=2$ case.
\begin{theorem}[\textbf{Tail probability decay}]\label{thm:upper_bound_largest_eigenvalue}
For the unimodular  \rm{RBM} with $\beta=2$, as $W\to \infty$  the following right  tail probability estimates hold for any  $x\ge 0$.
 
    (1)  If $d<4$ and $W\ll L^{1-\frac{d}{6}}$,
    then  
    \begin{equation}
        \mathbb{P}(\lambda_{\mathrm{max}}(H)\ge 2-a_4+x{W^{-\frac{4d}{6-d}}})\le C_1 \frac{N}{W^{\frac{6d}{6-d}}} e^{-C_2|x|^{\frac{6-d}{4}}}.
    \end{equation}

(2)
    Under  the   {\bf Supercritical Assumption},    
    we have   right tail estimate  of the  Tracy-Widom distribution
    \begin{equation}
        \mathbb{P}(\lambda_{\mathrm{max}}(H)\ge 2-a_4+xN^{-\frac{2}{3}})\le C_1 e^{-C_2|x|^{\frac{3}{2}}}.
    \end{equation}

    (3) For      $d\ge 4$ and for  any given   $\delta\in (0,0.1)$,  if    $W\ll L^{\frac{1}{3}}$, 
then     
\begin{equation}
        \mathbb{P}(\lambda_{\mathrm{max}}(H)\ge 2-a_4+xW^{-\frac{4d}{2+\delta}})\le C_{\delta, 1} N e^{-C_{\delta,2}|x|^{\frac{2+\delta}{4}}}.
\end{equation}
\end{theorem}
\begin{proof}[Proof of Theorem \ref{thm:upper_bound_largest_eigenvalue}]
By Lemma \ref{lem:B.7_outside_upperbound}, we have
\begin{equation}\label{equ:upper_4.51}
    \begin{aligned}
        &\mathbb{E}\big[e^{C_1 n \sqrt{(\lambda_{\mathrm{max}}-2-a_4)_+}}\big]\le \mathbb{E}\big[\mathrm{Tr} (\frac{1}{n}\mathcal{P}_{n}(H))^4\big]\\
        &\le C\left\{\exp\{Cn^{\frac{3}{2}}N^{-\frac{1}{2}}\}+\frac{N}{n^{\frac{d'}{2}}W^d}\exp\{Cn^{\frac{6-d'}{4-d'}}W^{-\frac{2d}{4-d'}}\}+\frac{N}{n^{3}}\right\},
    \end{aligned}
\end{equation}
where the last inequality comes from Proposition \ref{thm:T_t_upper_bound} by taking $k=1,t_1=4$ and from the upper bound of $T(\{(n,4)\})$. Hence by Markov's  inequality, 
\begin{equation}
    \mathbb{P}(\lambda_{\mathrm{max}}(H)\ge 2-a_4+y)\le e^{-n\sqrt{y}}\mathbb{E}\big[e^{C_1 n \sqrt{(\lambda_{\mathrm{max}}-2-a_4)_+}}\big].
\end{equation}

In the sub-critical case, we take $d'=d$, $y=xW^{-\frac{4d}{6-d}}$ and $n\sim C_3W^{\frac{2d}{6-d}}x^{\frac{4-d}{4}}$, we have
\begin{equation}
        \mathbb{P}(\lambda_{\mathrm{max}}(H)\ge 2-a_4+x{W^{-\frac{4d}{6-d}}})\le C_1' \frac{N}{W^{\frac{6d}{6-d}}} e^{-C_2'|x|^{\frac{6-d}{4}}}.
\end{equation}
Similarly, in the super-critical case, we take $d'=(4-\epsilon)\land d$, $y=xN^{-\frac{2}{3}}$ and $n= C_4 N^{\frac{1}{3}}x$, we have
\begin{equation}
        \mathbb{P}(\lambda_{\mathrm{max}}(H)\ge 2-a_4+xN^{-\frac{2}{3}})\le C_1'' e^{-C_2''|x|^{\frac{3}{2}}}.
\end{equation}
In the case $d\ge 4$ and $W\ll L^{\frac{1}{3}}$, we take $d'=4-\delta$, $y=xW^{-\frac{4d}{2+\delta}}$, $n=C_5W^{\frac{2d}{2+\delta}}x^{\frac{\delta}{4}}$, and we obtain
\begin{equation}
        \mathbb{P}(\lambda_{\mathrm{max}}(H)\ge 2-a_4+xW^{-\frac{4d}{2+\delta}})\le C_{\delta, 1} N e^{-C_{\delta,2}|x|^{\frac{2+\delta}{4}}}.
\end{equation}

    This thus completes the proof.  

\end{proof}

\subsection{Proofs of Theorem  \ref{metatheoremsup}--\ref{metatheoremcri}: \texorpdfstring{$\beta=2$}{β=2} case}\label{sec:beta=2_main_thm}

\begin{proof}[Proof of Theorem \ref{metatheoremsup}: $\beta=2$]
   Take all $t_i=4, 8$  and $n_i=[\tau_iN^{\frac{1}{3}}]$ in  Theorem \ref{prop:super_4},   we have  
    \begin{equation}\label{equ:4.64}
        \mathbb{E}\big[\prod_{i=1}^{k}\tr (\mathcal{P}_{n_i}(H))^4\big]=(1+o(1))\mathbb{E}\big[\prod_{i=1}^{k}\tr (U_{n_i}(H_{\mathrm{GUE}}/2))^4\big]
    \end{equation}
    and 
    \begin{equation}\label{equ:P8=U8}
        \mathbb{E}\big[\prod_{i=1}^{k}\tr (\mathcal{P}_{n_i}(H))^8\big]=(1+o(1))\mathbb{E}\big[\prod_{i=1}^{k}\tr (U_{n_i}(H_{\mathrm{GUE}}/2))^8\big].
    \end{equation}
    
    Now we consider the case of $k=1$ first and put  $n=[\tau N^{\frac{1}{3}}]$.  Let $\lambda_1\le \lambda_2\le \cdots \le \lambda_N$ be eigenvalues of $H$ and let  $\delta>0$ be a fixed small number.   Divide eigenvalues into three regimes and we rewrite  
      \begin{equation}
        \begin{aligned}
             &\tr \big(\frac{1}{n}\mathcal{P}_{n}(H)\big)^8= \sum_{i=1}^{N} \big(\frac{1}{n}\mathcal{P}_{n}(\lambda_i)\big)^8=\Sigma_{1}+\Sigma_{2}+\Sigma_{3},
        \end{aligned}
    \end{equation}
    where 
      \begin{equation}
        \begin{aligned}
            \Sigma_{1}=\sum_{i: |\lambda_i|\le 2-a_4-N^{-\frac{5}{12}+\delta}} &\big(\frac{1}{n}\mathcal{P}_{n}(\lambda_i)\big)^8, \quad 
             \Sigma_{2}= \sum_{i:|\lambda_i|\ge 2-a_4+N^{-\frac{2}{3}+\delta}}\big(\frac{1}{n}\mathcal{P}_{n}(\lambda_i)\big)^8,\\
             & \Sigma_{3}= \sum_{i:2-a_4-N^{-\frac{5}{12}+\delta}<|\lambda_i|< 2-a_4+N^{-\frac{2}{3}+\delta}}\big(\frac{1}{n}\mathcal{P}_{n}(\lambda_i)\big)^8.
        \end{aligned}
    \end{equation}
    By Lemma \ref{lem:B.6_inside_upperbound},
    \begin{equation}
        \Sigma_1 \le N \cdot N^{-\frac{8}{3}} N^{4(-\frac{5}{12}+\delta)}=o(1).
    \end{equation}
    Here $N^{-\frac{5}{12}+\delta}$ can also be replaced by $N^{-\frac{2}{3}+\delta}$ for any fixed $\delta>0$ if we change the exponent $8$ to $2k$ for some sufficiently  large constant $k$. 
    
    We also throw away the eigenvalues  in the second sum    $\Sigma_2$, as Theorem \ref{thm:upper_bound_largest_eigenvalue} shows that the probability of finding eigenvalues in $[2-a_4+N^{\frac{2}{3}+\delta},\infty)$ is exponentially small.
    
    For the third sum $\Sigma_3$,
    by Lemma \ref{lem:B.5} and \ref{lem:B.9_Chebyshev}, we need to  introduce rescaling variables  $\lambda_i=2-a_4+N^{-2/3}y_i$.   Combining equation \eqref{equ:4.64} with \eqref{equ:P8=U8} and applying the continuity theorem (Theorem \ref{prop:continuity_thm}), we conclude that in the super-critical case, all $k$-point correlation functions of the RBM (Definition \ref{defmodel}) converge weakly to their GUE  counterparts. This approach follows the pioneering work of Soshnikov  \cite{soshnikov1999universality}. 
\end{proof}

\begin{proof}[Proof of Theorem \ref{metatheoremsub}:  $\beta=2$]
    Use  the sub-critical part of Theorem \ref{prop:summery_prop} and we   know that
    \begin{equation}
        \mathbb{E}\big[\prod_{i=1}^{k}\tr \mathcal{P}_{n_i}(H)\big]=(1+o(1))\prod_{i=1}^{k}\mathbb{E}\big[\tr \mathcal{P}_{n_i}(H))\big].
    \end{equation}
    On the other hand, by Proposition \ref{prop:5.4},  
    \begin{equation}
    \mathbb{E}\big[\tr \mathcal{P}_{n_i}(H))\big]=(1+o(1))T(n_i)=(1+o(1))\frac{N}{W^{\frac{6d}{6-d}}} n_i \psi_1(\frac{n_i}{W^{\frac{2d}{6-d}}}),
    \end{equation}
    from which the desired result  immediately follows.  
\end{proof}

\begin{proof}[Proof of Theorem \ref{metatheoremcri}: $\beta=2$]
    The proof is immediate by taking all $t_i=1$ in Theorem \ref{prop:critical_4}.
\end{proof}

\section{Tadpole diagram renormalization}\label{sec:tadpole_diagram}

\subsection{Renormalization polynomials}\label{sec:renorm_poly}

To attack the tadpole  divergence problem, we introduce  a loop-free    event starting at the  $i$-th step  with cutoff parameter $R$  
\begin{equation} \label{Lnotation1}
    \mathbb{1}(\mathcal{L}\{i;R\})=\prod_{l=3}^{R}(1-\mathbb{1}_{i,2l}), \quad \mathbb{1}_{i,2l}=\prod_{t=i}^{i+l}\mathbb{1}(x_t=x_{t+l}),
\end{equation}
where   $\mathbb{1}_{i,2l}=0$ denotes a loop in which every vertex is traveled exactly two times and   $\mathbb{1}_{i,2l}=0$ by convention when   $i+2l$ exceeds  the index   of $x_i$. We also introduce   the loop-free event in an interval 
\begin{equation} \label{Lnotation2}
     \mathbb{1}\left(\mathcal{L}\left\{[t,s];R\right\}\right)=\prod_{i=t}^s\mathbb{1}(\mathcal{L}\{i;R\}).
\end{equation}

Recalling the four matrices $(\Phi_1)_{xy}, (\Phi_3)_{xy}, (\Phi_5)_{xy}$ and $(\Phi_7)_{xy}$ defined in \eqref{equ:2.12}, and  the notation    $\mathbf{b}^{(s)}=(b_1, \ldots, b_s)$ with nonnegative integers and    $|\mathbf{b}^{(s)}|=\sum_{i= 1}^{s}b_i$. Similar to Definition \ref{non-backtracking power}, we have 

\begin{definition} \label{non-backtracking R-power}

The   almost   nonbacktracking power of matrix $H$  is said to be a  loop-free   matrix with truncated parameter $R$,  denoted  by $\mathcal{V}^{(R)}_{n}$, if   the $(x,y)$ entry   
\begin{equation}\label{equ:V_n_def}
    (\mathcal{V}^{(R)}_{n})_{xy}=\sum_{s\ge 1}\sum_{|\mathbf{b}^{(s)}|=n}\sum_{\{x_i\}_{i=0}^s}\delta_{x_0,x}\delta_{x_s,y}\mathbb{1}(\mathcal{L}\{[0,s];R\})\prod_{i=0}^{s-2}\mathbb{1}({x_i\ne x_{i+2}})\prod_{i=1}^{s}(\Phi_{b_i})_{x_{i-1}x_{i}},
\end{equation}
where the sum $\sum_{|\mathbf{b}^{(s)}|=n}$ is taken  over all   $b_{i}\in\{1,3\}$, and   
  $\mathcal{V}^{(R)}_n=0$ for any integer  $n<0$ and $\mathcal{V}_{1}^{(R)}=\mathbb{I}$.  In the special case of $n=0$, the sum should be  taken over $s$ from $s=0$ and the associated term $\delta_{x,y}$ should be added.
   We also introduce two matrices by 
\begin{equation}
    (\underline{\Phi_5 \mathcal{V}_n^{(R)}})_{xy}:=\sum_{s\ge 1}\sum_{|\mathbf{b}^{(s)}|=n,b_1=5}\sum_{\{x_i\}_{i=0}^s}\delta_{x_0,x}\delta_{x_s,y}\mathbb{1}(\mathcal{L}\{[0,s];R\})\prod_{i=0}^{s-2}\mathbb{1}({x_i\ne x_{i+2}})\prod_{i=1}^{s}(\Phi_{b_i})_{x_{i-1}x_{i}},
\end{equation} 
and
\begin{equation}
    (\underline{\Phi_7 \mathcal{V}_n^{(R)}})_{xy}:=\sum_{s\ge 1}\sum_{|\mathbf{b}^{(s)}|=n,b_1=7}\sum_{\{x_i\}_{i=0}^s}\delta_{x_0,x}\delta_{x_s,y}\mathbb{1}(\mathcal{L}\{[0,s];R\})\prod_{i=0}^{s-2}\mathbb{1}({x_i\ne x_{i+2}})\prod_{i=1}^{s}(\Phi_{b_i})_{x_{i-1}x_{i}}.
\end{equation} where the sum is over all other $b_{i}\in \{1,3\}$ for $i>1$. 
\end{definition}

\begin{lemma}\label{lem:lemma4.2}
Let    $a_4=\sum_{y\in \Lambda_L}|H_{xy}|^4$, we have the recurrence relation 
    \begin{equation}\label{equ:V_n-relation}
        H\mathcal{V}^{(R)}_{n-1}=\mathcal{V}^{(R)}_n+\mathcal{V}^{(R)}_{n-2}-a_4\mathcal{V}^{(R)}_{n-4}+\sum_{l=3}^{3R}a_{2l}\mathcal{V}^{(R)}_{n-2l}+\underline{\mathcal{E}^{(1)}_n}+\underline{\mathcal{E}^{(2)}_n}+\underline{\mathcal{E}^{(3)}_n}.
    \end{equation}
   Here     for $t\ge 3$ 
\begin{equation}\label{equ:4.13}
a_{2t}=\sum_{l=3}^{R}\sum_{|\mathbf{b}^{(2s)}|=2t,b_1=1}\sum_{\{x_i\}_{i=0}^{2l}}\delta_{x_0,x}\mathbb{1}_{0,2l}\prod_{i=3}^{l-1}(1-\mathbb{1}_{0,2i})\prod_{i=0}^{2l-2}\mathbb{1}(x_i\ne x_{i+2})\prod_{i=1}^{2l}(\Phi_{b_i})_{x_ix_{i+1}},
\end{equation}
 and    the three error terms  
\begin{equation}
\underline{\mathcal{E}^{(1)}_n}:=\underline{(\Phi_{5}\mathcal{V}_{n-5}^{(R)})}+\underline{(\Phi_{7}\mathcal{V}_{n-7}^{(R)})},
\end{equation}
\begin{equation}
\begin{aligned}
    (\underline{\mathcal{E}^{(2)}_{n}})_{xy}&:=-\sum_{s\ge 1}\sum_{|\mathbf{b}^{(s)}|=n-1}\sum_{\{x_i\}_{i=1}^{s+1}}\delta_{x,x_2}\delta_{x_{s+1},y}(\Phi_1)_{xx_1}\big(1-\mathbb{1}(\mathcal{L}\{1;R\})\big)\\
    &\qquad \qquad \qquad \cdot \mathbb{1}(\mathcal{L}\{[2,s+1];R\})\prod_{i=1}^{s-1}\mathbb{1}({x_i\ne x_{i+2}})\prod_{i=1}^{s}(\Phi_{b_i})_{x_{i}x_{i+1}},
\end{aligned}
\end{equation}
\begin{equation}
    (\underline{\mathcal{E}^{(3)}_n})_{xy}=-\sum_{s\ge 1}\sum_{|\mathbf{b}^{(s)}|=n-1}\sum_{\{x_i\}_{i=0}^{s+1}}\delta_{x_0,x}\delta_{x_{s+1},y}\sum_{l=3}^{R}\mathbb{1}_{0,2l}\prod_{i=3}^{l-1}(1-\mathbb{1}_{0,2i})B_{2l}\,(\Phi_1)_{x_0x_1}\prod_{i=1}^{s}(\Phi_{b_i})_{x_{i}x_{i+1}},
\end{equation}
with  
\begin{equation}
    B_{2l}=\prod_{i=0}^{2l-2}\mathbb{1}(x_i\ne x_{i+2})\prod_{i=2l}^{s-1}\mathbb{1}(x_i\ne x_{i+2})\mathbb{1}(\mathcal{L}\{[2l,s+1];R\})\Big(1-\mathbb{1}(x_{2l-1}\ne x_{2l+1})\prod_{i=1}^{2l-1}\mathbb{1}(\mathcal{L}\{i;R\})\Big).
\end{equation}
\end{lemma}

\begin{proof}
For simplicity,   we omit the superscript $R$ of $\mathcal{V}^{(R)}$. By the definition in \eqref{equ:V_n_def}, when $n\geq 2$  we change the range of index and obtain  
\begin{equation}
\begin{aligned}
    (\mathcal{V}_{n-1})_{xy}&=\sum_{s\ge 1}\sum_{|\mathbf{b}^{(s)}|=n-1}\sum_{\{x_i\}_{i=0}^s}\delta_{x_0,x}\delta_{x_s,y}\mathbb{1}(\mathcal{L}\{[0,s];R\})\prod_{i=0}^{s-2}\mathbb{1}({x_i\ne x_{i+2}})\prod_{i=1}^{s}(\Phi_{b_i})_{x_{i-1}x_{i}}\\
    &=\sum_{s\ge 1}\sum_{|\mathbf{b}^{(s)}|=n-1}\sum_{\{x_i\}_{i=1}^{s+1}}\delta_{x_1,x}\delta_{x_s,y}\mathbb{1}(\mathcal{L}\{[1,s+1];R\})\prod_{i=1}^{s-1}\mathbb{1}({x_i\ne x_{i+2}})\prod_{i=1}^{s}(\Phi_{b_i})_{x_{i}x_{i+1}}.
    \end{aligned}
\end{equation} 
By the identity 
\begin{equation}
    1=\mathbb{1}(x_0\ne x_2)\mathbb{1}(\mathcal{L}\{0;R\})+\mathbb{1}(x_0=x_2)+\mathbb{1}(x_0\ne x_2)(1-\mathbb{1}(\mathcal{L}\{0;R\})),
\end{equation}
 split the summation into three cases  and we can rewrite the sum into three parts   
\begin{equation}\label{equ:4.11}
   \begin{aligned} (H\mathcal{V}_{n-1})_{xy}&=\sum_{s\ge 1}\sum_{|\mathbf{b}^{(s)}|=n-1}\sum_{\{x_i\}_{i=0}^{s+1}}\delta_{x_0,x}\delta_{x_{s+1},y}\\
   &\times \mathbb{1}(\mathcal{L}\{[1,s+1];R\})\prod_{i=1}^{s-1}\mathbb{1}({x_i\ne x_{i+2}})(\Phi_1)_{x_0x_1}\prod_{i=1}^{s}(\Phi_{b_i})_{x_{i}x_{i+1}}\\
   &:=\Sigma^{(i)}+\Sigma^{(ii)}+\Sigma^{(iii)}.
   \end{aligned}
\end{equation}

\begin{itemize}
    \item[\textbf{Case (i)}:] the first sum associated with the  factor  $\mathbb{1}(x_0\ne x_2)\mathbb{1}(\mathcal{L}\{0;R\})$. 
        Note that 
\begin{equation}
    \mathbb{1}(\mathcal{L}\{[0,s+1];R\})=\mathbb{1}(\mathcal{L}\{0;R\})\mathbb{1}(\mathcal{L}\{[1,s+1];R\})
\end{equation}
and 
\begin{equation}
    \mathbb{1}(x_0\ne x_2)\prod_{i=1}^{s-1}\mathbb{1}({x_i\ne x_{i+2}})=\prod_{i=0}^{s-1}\mathbb{1}({x_i\ne x_{i+2}}),
\end{equation}
we see from 
\eqref{equ:4.11} that the first sum  can be rewritten as 
\begin{equation}\label{equ:4.15}
   \Sigma^{(i)}= \sum_{s\ge 2}\sum_{|\mathbf{b}^{(s)}|=n,b_1=1}\sum_{\{x_i\}_{i=0}^{s}}\delta_{x_0,x}\delta_{x_{s},y}\mathbb{1}(\mathcal{L}\{[0,s];R\})\prod_{i=0}^{s-2}\mathbb{1}({x_i\ne x_{i+2}})\prod_{i=1}^{s}(\Phi_{b_i})_{x_{i-1}x_{i}}.
\end{equation}
\item[\textbf{Case (ii):}] the second sum associated with the factor  $\mathbb{1}(x_0=x_2)=\delta_{x_0,x_2}$. 
Rewrite the sum  
\begin{equation}\label{equ:4.16}
\begin{aligned}
    \Sigma^{(ii)}
    &=\sum_{s\ge 1}\sum_{|\mathbf{b}^{(s)}|=n-1}\sum_{\{x_i\}_{i=2}^{s+1}}\delta_{x,x_2} \delta_{x_{s+1},y} \mathbb{1}(\mathcal{L}\{[2,s+1];R\})\prod_{i=2}^{s-1}\mathbb{1}({x_i\ne x_{i+2}})\prod_{i=2}^{s}(\Phi_{b_i})_{x_{i}x_{i+1}}\\
    &~~~~~~ ~\times   \sum_{x_1}\mathbb{1}(x_1\ne x_3)(\Phi_1)_{xx_1}(\Phi_{b_1})_{x_1x}\mathbb{1}(\mathcal{L}\{1;R\}),
\end{aligned}
\end{equation}
and note that   $\mathbb{1}(\mathcal{L}\{[2,s+1];R\})\prod_{i=2}^{s-1}\mathbb{1}({x_i\ne x_{i+2}})$ is not relevant to   $x_1$,  summing  over $x_1$ gives rise to 
\begin{equation}
\begin{aligned}
    &\sum_{x_1}\mathbb{1}(x_1\ne x_3)(\Phi_1)_{xx_1}(\Phi_{b_1})_{x_1x}\mathbb{1}(\mathcal{L}\{1;R\})\\
    &=\sum_{x_1}\mathbb{1}(x_1\ne x_3)(\Phi_1)_{xx_1}(\Phi_{b_1})_{x_1x}-\sum_{x_1}\mathbb{1}(x_1\ne x_3)(\Phi_1)_{xx_1}(\Phi_{b_1})_{x_1x}(1-\mathbb{1}(\mathcal{L}\{1;R\}))\\
    &=\delta_{b_1,1}-a_4\delta_{b_1,3}-\mathbb{1}(s\ge 2) (\Phi_1)_{xx_3}(\Phi_{b_1})_{x_3x}-\sum_{x_1}\mathbb{1}(x_1\ne x_3)(\Phi_1)_{xx_1}(\Phi_{b_1})_{x_1x}(1-\mathbb{1}(\mathcal{L}\{1;R\})).
\end{aligned}
\end{equation}
This suggests that $ \Sigma^{(ii)}$ in \eqref{equ:4.16} should be   divided into   $4$ cases. Correspondingly, put \begin{equation}
    \Sigma^{(ii)}= \Sigma^{(ii)}_1+\Sigma^{(ii)}_2+\Sigma^{(ii)}_3+\Sigma^{(ii)}_4.
\end{equation} 
    \item[\textbf{Case (ii.a):}]  the term  $\delta_{b_1,1}$.  By Definition \ref{non-backtracking R-power}, we see from  the sum \eqref{equ:4.16} that 
\begin{equation}
\begin{aligned}
   \Sigma^{(ii)}_1 &=\sum_{s\ge 1}\sum_{|\mathbf{b}^{(s)}|=n-1,b_1=1}\sum_{\{x_i\}_{i=2}^{s+1}}\delta_{x,x_2}\delta_{x_{s+1},y}\mathbb{1}(\mathcal{L}\{[2,s+1];R\})\prod_{i=2}^{s-1}\mathbb{1}({x_i\ne x_{i+2}})\prod_{i=2}^{s}(\Phi_{b_i})_{x_{i}x_{i+1}}\\
    &=\sum_{s\ge 0}\sum_{|\mathbf{b}^{(s)}|=n-2}\sum_{\{x_i\}_{i=0}^{s}}\delta_{x,x_0} \delta_{x_{s},y}\mathbb{1}(\mathcal{L}\{[0,s];R\})\prod_{i=0}^{s-2}\mathbb{1}({x_i\ne x_{i+2}})\prod_{i=1}^{s}(\Phi_{b_i})_{x_{i-1}x_{i}}=\mathcal{V}_{n-2}.
\end{aligned}
\end{equation}
Here the sum with $s=0$ has contribution $0$ whenever $n>2$.
\item[\textbf{Case (ii.b):}] the term  $-a_4\delta_{b_1,3}$. At this moment, by Definition \ref{non-backtracking R-power} the corresponding sum from \eqref{equ:4.16}  
\begin{equation}
\begin{aligned}
     &\Sigma^{(ii)}_2=-a_4\sum_{s\ge 1}\sum_{|\mathbf{b}^{(s)}|=n-1,b_1=3}\sum_{\{x_i\}_{i=2}^{s+1}}\delta_{x,x_2}\delta_{x_{s+1},y}\mathbb{1}(\mathcal{L}\{[2,s+1];R\})\prod_{i=2}^{s-1}\mathbb{1}({x_i\ne x_{i+2}})\prod_{i=2}^{s}(\Phi_{b_i})_{x_{i}x_{i+1}}\\
    &=-a_4\sum_{t\ge 0}\sum_{|\mathbf{b}^{(t)}|=n-4}\sum_{\{x_i\}_{i=0}^{t}}\delta_{x_0,x}\delta_{x_{t},y}\mathbb{1}(\mathcal{L}\{[0,t];R\})\prod_{i=0}^{t-2}\mathbb{1}({x_i\ne x_{i+2}})\prod_{i=1}^{t}(\Phi_{b_i})_{x_{i-1}x_{i}}=-a_4\mathcal{V}_{n-4},
\end{aligned}
\end{equation}
where we have substituted $t=s-1$ and shifted the subscript $b_{i}$ to $b_{i-1}$ and $x_{i}$ to $x_{i-2}$.
\item[\textbf{Case (ii.c):}] the term  $-(\Phi_1)_{xx_3}(\Phi_{b_1})_{x_3x}$. Let   $l=1+b_1+b_2$,  we see from $b_i\in\{1,3\}$ that 
\begin{equation}
    -\sum_{1+b_1+b_2=l} \mathbb{1}(s\ge 2)  (\Phi_1)_{x_2x_3}(\Phi_{b_1})_{x_3x_2}(\Phi_{b_2})_{x_2x_3}= \begin{cases}
        \mathbb{1}(s\ge 2) (\Phi_{3})_{x_2x_3},~&l=3,\\
        \mathbb{1}(s\ge 2) (\Phi_{5})_{x_2x_3},~&l=5,\\
       \mathbb{1}(s\ge 2)  (\Phi_{7})_{x_2x_3},~&l=7.
    \end{cases}
\end{equation}
Thus    the $\mathbb{1}(s\ge 2) \Phi_3$ term from $\Sigma^{(ii)}_2$ gives  rise to 
\begin{equation}\label{equ:4.21}
\begin{aligned}
    &\sum_{s\ge 2}\sum_{|\mathbf{b}^{(s)}|=n-1,b_1=b_2=1}\sum_{\{x_i\}_{i=2}^{s+1}}\delta_{x,x_2}\delta_{x_{s+1},y}\mathbb{1}(\mathcal{L}\{[2,s+1];R\})\prod_{i=2}^{s-1}\mathbb{1}({x_i\ne x_{i+2}})(\Phi_{3})_{x_2x_3}\prod_{i=3}^{s}(\Phi_{b_i})_{x_{i}x_{i+1}}\\
    &=\sum_{s\ge 1}\sum_{|\mathbf{b}^{(s)}|=n,b_1=3}\sum_{\{x_i\}_{i=0}^{s}}\delta_{x,x_0}\delta_{x_s,y}\mathbb{1}(\mathcal{L}\{[0,s];R\})\prod_{i=0}^{s-2}\mathbb{1}({x_i\ne x_{i+2}})\prod_{i=1}^{s}(\Phi_{b_i})_{x_{i-1}x_{i}}.
\end{aligned}
\end{equation}
Noting that  $n\geq 2$, we can start  the summation index $s$ from 1 and obtain $\mathcal{V}_n$ by summing  up the right-hand side of \eqref{equ:4.15} and \eqref{equ:4.21}, which gives  .  

And similarly, using of  Definition \ref{non-backtracking R-power} shows that  the other two  terms in total gives the first error term 
\begin{equation}
\underline{\mathcal{E}^{(1)}_n}:=\underline{(\Phi_{5}\mathcal{V}_{n-5})}+\underline{(\Phi_{7}\mathcal{V}_{n-7})}.
\end{equation}

\item[\textbf{Case (ii.d):}] the term  $-\sum_{x_1}\mathbb{1}(x_1\ne x_3)(\Phi_1)_{xx_1}(\Phi_{b_1})_{x_1x}(1-\mathbb{1}(\mathcal{L}\{1;R\}))$.  Obviously, we see from  the sum \eqref{equ:4.16} that the relevant sum is exactly the error $ (\underline{\mathcal{E}^{(2)}_{n}})_{xy}$.

In total, combination of  case (i) and case (ii) gives rise to 
\begin{equation}
    \mathcal{V}_n+\mathcal{V}_{n-2}-a_4\mathcal{V}_{n-4}+\underline{\mathcal{E}^{(1)}_n}+\underline{\mathcal{E}^{(2)}_n}.
\end{equation}

\item[\textbf{Case (iii):}] the term  $\mathbb{1}(x_0\ne x_2)(1-\mathbb{1}(\mathcal{L}\{0;R\}))$. Using the expansion   
\begin{equation}
\begin{aligned}
    1-\mathbb{1}(\mathcal{L}\{0;R\})&=1-\prod_{l=3}^R(1-\mathbb{1}_{0,2l})=\sum_{l=3}^{R}\mathbb{1}_{0,2l}\prod_{i=3}^{l-1}(1-\mathbb{1}_{0,2i}),
\end{aligned}
\end{equation}
we have 
\begin{equation}
\begin{aligned}
    &\mathbb{1}(x_0\ne x_2)(1-\mathbb{1}(\mathcal{L}\{0;R\}))\mathbb{1}(\mathcal{L}\{[1,s+1];R\})\prod_{i=1}^{s-1}\mathbb{1}(x_i\ne x_{i+2})\\
    &=\sum_{l=3}^{R}\mathbb{1}_{0,2l}\prod_{i=3}^{l-1}(1-\mathbb{1}_{0,2i})\prod_{i=1}^{s+1}\mathbb{1}(\mathcal{L}\{i;R\})\prod_{i=0}^{s-1}\mathbb{1}(x_i\ne x_{i+2}).
\end{aligned}
\end{equation}
So fix $l$ and  we get
\begin{equation}
\begin{aligned}
    &\prod_{i=1}^{s+1}\mathbb{1}(\mathcal{L}\{i;R\})\prod_{i=0}^{s-1}\mathbb{1}(x_i\ne x_{i+2})=\prod_{i=0}^{2l-2}\mathbb{1}(x_i\ne x_{i+2})\prod_{i=2l}^{s-1}\mathbb{1}(x_i\ne x_{i+2})\\
    &\mathbb{1}(\mathcal{L}\{[2l,s+1];R\})\mathbb{1}(x_{2l-1}\ne x_{2l+1})\prod_{i=1}^{2l-1}\mathbb{1}(\mathcal{L}\{i;R\})\\
    &=\prod_{i=0}^{2l-2}\mathbb{1}(x_i\ne x_{i+2})\prod_{i=2l}^{s-1}\mathbb{1}(x_i\ne x_{i+2})\mathbb{1}(\mathcal{L}\{[2l,s+1];R\})-\prod_{i=0}^{2l-2}\mathbb{1}(x_i\ne x_{i+2})\\
    &\prod_{i=2l}^{s-1}\mathbb{1}(x_i\ne x_{i+2})\mathbb{1}(\mathcal{L}\{[2l,s+1];R\})\Big(1-\mathbb{1}(x_{2l-1}\ne x_{2l+1})\prod_{i=1}^{2l-1}\mathbb{1}(\mathcal{L}\{i;R\})\Big)\\
    &:=A_{2l}-B_{2l}.
\end{aligned}
\end{equation}

  By  the definition,   the sum associated with $B_{2l}$ leads to $(\underline{\mathcal{E}^{(3)}_n})_{xy}$. This further implies that    the third part from  
\eqref{equ:4.11}   reads  
\begin{equation}
    \begin{aligned}  &\Sigma^{(iii)}=   (\underline{\mathcal{E}^{(3)}_n})_{xy}+\\
    &\sum_{s\ge 1}\sum_{|\mathbf{b}^{(s)}|=n-1}\sum_{\{x_i\}_{i=0}^{s+1}}\delta_{x_0,x}\delta_{x_{s+1},y} \sum_{l=3}^{R}\mathbb{1}_{0,2l}\prod_{i=3}^{l-1}(1-\mathbb{1}_{0,2i})A_{2l} (\Phi_1)_{x_0x_1}\prod_{i=1}^{s}(\Phi_{b_i})_{x_{i}x_{i+1}}.
    \end{aligned}
\end{equation}

At this time, it suffices to deal with the sum associated with $A_{2l}$. To end this, 
note that for event $A_{2l}$,  $x_{2l}=x_{0}=x$ because of the indictor function $\mathbb{1}_{0,2l}$, so the sum over $x_{i},~0\le i\le 2l-1$ are irrelevant to $x_{i},~2l+1\le i\le s+1$. Hence  
\begin{equation}
    \begin{aligned}
        &\sum_{s\ge 1}\sum_{|\mathbf{b}^{(s)}|=n-1}\sum_{\{x_i\}_{i=0}^{s+1}}\delta_{x_0,x}\delta_{x_{s+1},y}\sum_{l=3}^{R}\mathbb{1}_{0,2l}\prod_{i=3}^{l-1}(1-\mathbb{1}_{0,2i})A_{2l}(\Phi_1)_{x_0x_1}\prod_{i=1}^{s}(\Phi_{b_i})_{x_{i}x_{i+1}}\\
        &=\sum_{l=3}^{R}\sum_{\{b_i\}_{i=0}^{2l-1},b_0=1}\bigg(\sum_{\{x_i\}_{i=0}^{2l}}\delta_{x_0,x}\mathbb{1}_{0,2l}\prod_{i=3}^{l-1}(1-\mathbb{1}_{0,2i})\prod_{i=0}^{2l-2}\mathbb{1}(x_i\ne x_{i+2})\prod_{i=0}^{2l-1}(\Phi_{b_i})_{x_{i}x_{i+1}}\bigg) \bigg(\sum_{s-2l\ge 0}\\
        &\sum_{\{b_i\}_{i=2l}^{s+1}}\mathbb{1}({\sum_{i=0}^{s}b_i=n})\sum_{\{x_{i}\}_{i=2l}^{s+1}}\delta_{x_{2l},x}\delta_{x_{s+1},y}\prod_{i=2l}^{s-1}\mathbb{1}(x_i\ne x_{i+2})\mathbb{1}(\mathcal{L}\{[2l,s+1];R\})\prod_{i=2l}^{s}(\Phi_{b_i})_{x_ix_{i+1}}\bigg)\\
        &= \sum_{l=3}^{R}\sum_{\{b_i\}_{i=1}^{2l},b_1=1}\left(\sum_{\{x_i\}_{i=0}^{2l}}\mathbb{1}_{0,2l}\prod_{i=3}^{l-1}(1-\mathbb{1}_{0,2i})\prod_{i=0}^{2l-2}\mathbb{1}(x_i\ne x_{i+2})\prod_{i=1}^{2l}(\Phi_{b_i})_{x_{i-1}x_i}\right)(\mathcal{V}_{n-|\mathbf{b}^{(2l)}|})_{xy}\\
        &=\sum_{t=3}^{3R}a_{2t}(\mathcal{V}_{n-2t})_{xy},
    \end{aligned}
\end{equation}
where 
\begin{equation}
a_{2t}= \sum_{|\mathbf{b}^{(2s)}|=2t,b_1=1}\sum_{\{x_i\}_{i=0}^{2l}}\delta_{x_0,x}\mathbb{1}_{0,2l}\prod_{i=3}^{l-1}(1-\mathbb{1}_{0,2i})\prod_{i=0}^{2l-2}\mathbb{1}(x_i\ne x_{i+2})\prod_{i=1}^{2l}(\Phi_{b_i})_{x_ix_{i+1}}.
\end{equation}

\end{itemize}

Finally, combine Case  (i), Case  (ii) and Case  (iii), and we complete the proof.
\end{proof}

\begin{proposition}\label{Prop:4.2}
Introduce a family of renormalization polynomials 
\begin{equation} \label{generalpolynomial}
    \widetilde{\mathcal{P}}^{(R)}_{n}(H)=H\widetilde{\mathcal{P}}_{n-1}^{(R)}(H)-\widetilde{\mathcal{P}}^{(R)}_{n-2}(H)+a_4\widetilde{\mathcal{P}}^{(R)}_{n-4}(H)-\sum_{l=3}^{3R} a_{2l}\widetilde{\mathcal{P}}^{(R)}_{n-2l}(H),\quad n\ge 3,
\end{equation}
and 
\begin{equation}
    \widetilde{\mathcal{P}}^{(R)}_{0}(H)=I,~ \widetilde{\mathcal{P}}^{(R)}_{1}(H)=H, ~\widetilde{\mathcal{P}}^{(R)}_{2}(H)=H^2-I, ~ \widetilde{\mathcal{P}}_{m}(H)=0 \ (\forall m<0).
\end{equation}
Then we have
\begin{equation}\label{equ:P_n_expansion}
    \widetilde{\mathcal{P}}^{(R)}_{n}(H)=\sum_{t\ge 0}\sum_{\{l_i\}_{i=0}^{t},|\mathbf{l}|=n}\sum_{\{c_i\}_{i=1}^{t}}\mathcal{V}^{(R)}_{l_0}\underline{\mathcal{E}_{l_1}^{(c_1)}}\ldots \underline{\mathcal{E}_{l_t}^{(c_t)}},
\end{equation}
where the sum is over $|\mathbf{l}|:=\sum_{i=0}^{t}l_i=n$ with $l_0\ge 0$, $l_i\ge 1$ for $i\ge 1$   and $c_i\in \{1,2,3\}.$
\end{proposition}
\begin{proof}
    The proof is almost the same as in that  of Proposition \ref{prop:2.7}.  Introduce   $\widetilde{\mathcal{D}}_n$ through
    \begin{equation}
        \widetilde{\mathcal{P}}^{(R)}_n=\mathcal{V}_n^{(R)}+\widetilde{\mathcal{D}}_n,
    \end{equation}
    then  we have $\widetilde{\mathcal{D}}_m=0$ for any $m\leq 3$, and 
    \begin{equation}\label{equ:4.45}
        \widetilde{\mathcal{D}}_n(H)=H\widetilde{\mathcal{D}}_{n-1}(H)-\widetilde{\mathcal{D}}_{n-2}(H)+a_4\widetilde{\mathcal{D}}_{n-4}(H)-\sum_{l=3}^R a_{2l}\widetilde{\mathcal{D}}_{n-2l}(H)+\sum_{i=1}^{3}\underline{\mathcal{E}^{(i)}_n}.
    \end{equation}
    We next prove the following relation by induction
    \begin{equation}\label{equ:4.46}
    \widetilde{\mathcal{D}}_{n}(H)=\sum_{t\ge 1}\sum_{\{l_i\}_{i=0}^{t},|\mathbf{l}|=n}\sum_{\{c_i\}_{i=1}^{t}}\mathcal{V}^{(R)}_{l_0}\underline{\mathcal{E}_{l_1}^{(c_1)}}\ldots \underline{\mathcal{E}_{l_t}^{(c_t)}}.
\end{equation}

  Assuming \eqref{equ:4.46} holds true  up to   $n-1$, now consider 
\begin{equation}
    H\widetilde{\mathcal{D}}_{n-1}(H)=\sum_{t\ge 1}\sum_{\{l_i\}_{i=0}^{t},|\mathbf{l}|=n-1}\sum_{\{c_i\}_{i=1}^{t}}H\mathcal{V}^{(R)}_{l_0}\underline{\mathcal{E}_{l_1}^{(c_1)}}\ldots \underline{\mathcal{E}_{l_t}^{(c_t)}}.
\end{equation}
Using \eqref{equ:V_n-relation}, we have 
    \begin{equation}\label{equ:4.49}
    \begin{aligned}
        &H\widetilde{\mathcal{D}}_{n-1}(H)=\sum_{t\ge 1}\sum_{\{l_i\}_{i=0}^{t},|\mathbf{l}|=n-1}\sum_{\{c_i\}_{i=1}^{t}}\\
        &\qquad \bigg(\mathcal{V}^{(R)}_{l_0+1}+\mathcal{V}^{(R)}_{l_0-1}-a_4\mathcal{V}^{(R)}_{l_0-3}+\sum_{l=3}^{R}a_{2l}\mathcal{V}^{(R)}_{l_0+1-2l}+\underline{\mathcal{E}^{(1)}_{l_0+1}}+\underline{\mathcal{E}^{(2)}_{l_0+1}}+\underline{\mathcal{E}^{(3)}_{l_0+1}}\bigg)\underline{\mathcal{E}_{l_1}^{(c_1)}}\ldots \underline{\mathcal{E}_{l_t}^{(c_t)}}\\
        &=\widetilde{\mathcal{D}}_{n-2}-a_4 \widetilde{\mathcal{D}}_{n-4}+\sum_{l=3}^{R}a_{2l}\widetilde{\mathcal{D}}_{n-2l}\\&\quad+\sum_{t\ge 1}\sum_{\{l_i\}_{i=0}^{t},|\mathbf{l}|=n-1}\sum_{\{c_i\}_{i=1}^{t}}\left(\mathcal{V}^{(R)}_{l_0+1}+\underline{\mathcal{E}^{(1)}_{l_0+1}}+\underline{\mathcal{E}^{(2)}_{l_0+1}}+\underline{\mathcal{E}^{(3)}_{l_0+1}}\right)\underline{\mathcal{E}_{l_1}^{(c_1)}}\ldots \underline{\mathcal{E}_{l_t}^{(c_t)}}.
    \end{aligned}
    \end{equation}
    On the other hand, 
    \begin{equation}\label{equ:4.50}
        \begin{aligned}
            &\sum_{t\ge 1}\sum_{\{l_i\}_{i=0}^{t},|\mathbf{l}|=n-1}\sum_{\{c_i\}_{i=1}^{t}}\left(\mathcal{V}^{(R)}_{l_0+1}+\underline{\mathcal{E}^{(1)}_{l_0+1}}+\underline{\mathcal{E}^{(2)}_{l_0+1}}+\underline{\mathcal{E}^{(3)}_{l_0+1}}\right)\underline{\mathcal{E}_{l_1}^{(c_1)}}\ldots \underline{\mathcal{E}_{l_t}^{(c_t)}}\\
            &=\bigg(\sum_{t\ge 1}\sum_{\{l_i\}_{i=1}^{t},|\mathbf{l}|=n,l_0\ge 1}\sum_{\{c_i\}_{i=1}^{t}}+\sum_{t\ge 2}\sum_{\{l_i\}_{i=1}^{t},|\mathbf{l}|=n,l_0=0}\sum_{\{c_i\}_{i=1}^{t}}\bigg)\mathcal{V}^{(R)}_{l_0}\underline{\mathcal{E}_{l_1}^{(c_1)}}\ldots \underline{\mathcal{E}_{l_t}^{(c_t)}}\\
            &=\widetilde{\mathcal{D}}_{n}- \sum_{l_0+l_1=n,l_0=0}\sum_{\{c_i\}_{i=1}^{t}}\mathcal{V}^{(R)}_{l_0}\underline{\mathcal{E}_{l_1}^{(c_1)}}\ldots \underline{\mathcal{E}_{l_t}^{(c_t)}}\\
            &=\widetilde{\mathcal{D}}_{n}-\sum_{i=1}^3\underline{\mathcal{E}_{n}^{(i)}}.
        \end{aligned}
    \end{equation}
    Combining \eqref{equ:4.49} and \eqref{equ:4.50}, we see  that \eqref{equ:4.46} satisfies   the recurrence relation \eqref{equ:4.45}. By induction, we  complete the   proposition.
\end{proof}

\subsection{Upper bound estimates}\label{sec:real_upper_bound}

  Our goal is to obtain  upper bounds for the trace terms of  the expansion  in \eqref{equ:P_n_expansion} of  Proposition \ref{Prop:4.2}.
\begin{proposition}\label{Prop:upper_error}

If $n\ll W^d$ as $W\to \infty$, then  there is some constant $C>0$ such that for any nonnegative integer $t$,
    \begin{equation}
        \sum_{\{l_i\}_{i=0}^{t},|\mathbf{l}|=n}\left|\mathbb{E}\left[\tr \mathcal{V}^{(R)}_{l_0}\underline{\mathcal{E}_{l_1}^{(c_1)}}\ldots \underline{\mathcal{E}_{l_t}^{(c_t)}} \right]\right|\le \Big(C\frac{1+\delta_{d,2}(\log W)^2}{W^{2d}}\Big)^t \sum_{\{l_i\}_{i=0}^{t},|\mathbf{l}|=n} \mathbb{E}\left[\tr V^{(R)}_{l_0}V_{l_1}^{(R)}\ldots V_{l_t}^{(R)}\right],
    \end{equation}
    where 
\begin{equation} \label{VR}
    {V}^{(R)}_{l}:=\sum_{\{x_i\}_{i=0}^l}\delta_{x_0,x}\delta_{x_l,y}\mathbb{1}(\mathcal{L}\{[0,l];R\})\prod_{i=0}^{l-2}\mathbb{1}({x_i\ne x_{i+2}})\prod_{i=1}^{l}(\Phi_{1})_{x_{i-1}x_{i}}.
\end{equation}
\end{proposition}

To prove the above proposition, we  introduce three new matrices, instead of those in \eqref{equ:2.12}, 
\begin{equation}
    (\widetilde{\Phi}_3)_{xy}=|H_{xy}|^2H_{xy},~(\widetilde{\Phi}_5)_{xy}=2|H_{xy}|^4H_{xy},~(\widetilde{\Phi}_7)_{xy}=|H_{xy}|^6H_{xy}.
\end{equation}
We also introduce the $\Phi_3$-free sums  
\begin{equation}
    (\underline{\widetilde{\Phi}_5 V_{n-5}})_{xy}:=\sum_{\{x_j\}_{j=0}^{n-3}}\delta_{x,x_0}\delta_{x_{n-3},y}\mathbb{1}(\mathcal{L}\{[0,n-3];R\})\left[\prod_{i=0}^{n-5}\mathbb{1}(x_i\ne x_{i+2})\right](\widetilde{\Phi}_{5})_{x_0x_1}\prod_{i=1}^{n-4} (\Phi_{1})_{x_{i}x_{i+1}},
\end{equation}
\begin{equation}
    (\underline{\widetilde{\Phi}_7 V_{n-7}})_{xy}:=\sum_{\{x_j\}_{j=0}^{n-5}}\delta_{x,x_0}\delta_{x_{n-5},y}\mathbb{1}(\mathcal{L}\{[0,n-5];R\})\left[\prod_{i=0}^{n-7}\mathbb{1}(x_i\ne x_{i+2})\right](\widetilde{\Phi}_{7})_{x_0x_1}\prod_{i=1}^{n-6} (\Phi_{1})_{x_{i}x_{i+1}},
\end{equation}
and  the modified error terms 
\begin{equation}
\underline{{E}^{(1)}_n}:=\underline{(\widetilde{\Phi}_{5}{V}_{n-5})}+\underline{(\widetilde{\Phi}_{7}V_{n-7})},
\end{equation}
\begin{equation}
\begin{aligned}
    (\underline{E^{(2)}_{n}})_{xy}&:=\sum_{\{x_i\}_{i=1}^{n}}\delta_{x,x_2}\delta_{x_{n},y}(\Phi_1)_{xx_1}(1-\mathbb{1}(\mathcal{L}\{1;R\}))\mathbb{1}(\mathcal{L}\{[2,n];R\})\prod_{i=1}^{n-2}\mathbb{1}({x_i\ne x_{i+2}})\prod_{i=1}^{n-1}(\Phi_{1})_{x_{i}x_{i+1}},
\end{aligned}
\end{equation}
\begin{equation}\label{equ:E3_4.56}
    (\underline{E^{(3)}_n})_{xy}=\sum_{\{x_i\}_{i=0}^{n}}\delta_{x_0,x}\delta_{x_{n},y}\sum_{l=3}^{R}\mathbb{1}_{0,2l}\prod_{i=3}^{l-1}(1-\mathbb{1}_{0,2i})B_{2l}(\Phi_1)_{x_0x_1}\prod_{i=1}^{n-1}(\Phi_{1})_{x_{i}x_{i+1}},
\end{equation}
where
\begin{equation} \label{B2l}
    B_{2l}=\prod_{i=0}^{2l-2}\mathbb{1}(x_i\ne x_{i+2})\prod_{i=2l}^{n-2}\mathbb{1}(x_i\ne x_{i+2})\mathbb{1}(\mathcal{L}\{[2l,n];R\})\Big(1-\mathbb{1}(x_{2l-1}\ne x_{2l+1})\prod_{i=1}^{2l-1}\mathbb{1}(\mathcal{L}\{i;R\})\Big).
\end{equation}
\begin{lemma}If $n\ll W^d$ and $R\le n$, then we have
    \begin{equation}
        \sum_{\{l_i\}_{i=0}^{t},|\mathbf{l}|=n}\left|\mathbb{E}[\tr \mathcal{V}^{(R)}_{l_0}\underline{\mathcal{E}_{l_1}^{(c_1)}}\ldots \underline{\mathcal{E}_{l_t}^{(c_t)}}]\right|\le \big(1+O(\frac{n}{W^d})\big)\sum_{\{l_i\}_{i=0}^{t},|\mathbf{l}|=n}\mathbb{E}\big[\tr V^{(R)}_{l_0}{E_{l_1}^{(c_1)}}\ldots {E_{l_t}^{(c_t)}}\big].
    \end{equation}
\end{lemma}
\begin{proof}
    We just need to get rid of all $\Phi_3$ terms. The proof can be completed  in a very  similar way  as in  Lemma \ref{PcalV}.
\end{proof}
The following monotonicity lemma, similar to Lemma \ref{lem:2.7}, is also important.
\begin{lemma}\label{lem:lem4.5_monotonicity}
For   all   non-negative integers $ l_i,j_i$, $i=1,\ldots,k$, let $j=\sum_{i}^t j_i$, if   $j W^{-d} \to 0$ as $W\to \infty$,  then     we have 
    \begin{equation}\label{equ:lem4.5_monotonicity}
        \mathbb{E}\Big[\tr\prod_{i=1}^t  V_{l_i}^{(R)}\Big]\le \bigg(1+O\Big(\frac{Rt(1+\delta_{d,2}\log j)}{W^d}\Big)\bigg)\mathbb{E}\Big[ \tr\prod_{i=1}^t V_{l_i+j_i+j_{i+1}}^{(R)}\Big],
    \end{equation}
    where  $j_{t+1}:=j_1$.
\end{lemma}
\begin{figure}
    \centering
    \includegraphics[width=0.4\textwidth]{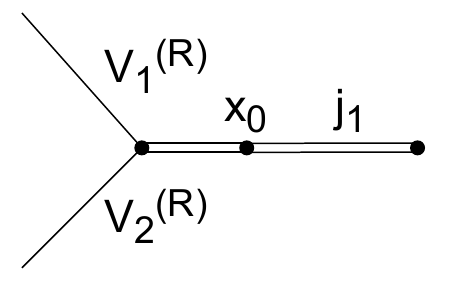}
    \caption{Add backtracking edge}
    \label{fig:add_backtracking}
\end{figure}
\begin{proof} The proof is very similar to Lemma \ref{lem:2.7}. 
We add a backtracking path of length $j_i$ between $V_{l_i}^{(R)}$ and $V_{l_{i+1}}^{(R)}$. Now we consider the probability that the new backtracking path breaks  the loop-erased structure of $V^{(R)}$; see Figure \ref{fig:add_backtracking}. By the definition of $\mathbb{1}(\mathcal{L}\{[0,n];R\})$, if the new backtracking path does not have any coincidence of the last $2R$ points of $V^{(R)}_{l_1},V_{l_2}^{(R)}$ (we denote by $X$ to be the set of $2R$ points), then the event $\mathbb{1}(\mathcal{L}\{[0,n];R\})$ does  not break. However, we denote the start point of the backtracking path as $x_0$, we can compute the exception for the number of coincidence as 
\begin{equation}
    \mathbb{E}[\#\{\text{Intersections}\}]=\sum_{x\in X}\sum_{l=1}^{j_1}p_{l}(x_0,x)\le C \sum_{l=1}^{j_1}Rl^{-\frac{d}{2}}W^{-d}=O\Big(\frac{R(1+\delta_{d,2}\log j_1)}{W^d}\Big).
\end{equation}

Thus we know that adding a backtracking path has at least $1-O\big(W^{-d}R(1+\delta_{d,2}\log j_1)\big)$ probability not breaking the loop-erased structure. Counting all $t$ path together, we arrive at\eqref{equ:lem4.5_monotonicity}.
\end{proof}

\begin{lemma}\label{lem:4.5} If $RW^{-d}(1+\delta_{d,2}\log R)\rightarrow0$ as $W\rightarrow \infty$, then there is some constant $C>0$  such that for any integer $t$,  
    \begin{equation}\label{equ:4.60}
        \sum_{\{l_i\}_{i=0}^{t},|\mathbf{l}|=n}\mathbb{E}[\tr V^{(R)}_{l_0}{E_{l_1}^{(c_1)}}\ldots {E_{l_t}^{(c_t)}}]\le \Big(C\frac{1+\delta_{d,2}\log W}{W^{2d}}\Big)^t \sum_{\{l_i\}_{i=0}^{t},|\mathbf{l}|=n} \mathbb{E}\left[\tr V^{(R)}_{l_0}V_{l_1}^{(R)}\ldots V_{l_t}^{(R)}\right].
    \end{equation}
\end{lemma}
With Lemma \ref{lem:4.5}, we can prove Proposition \ref{Prop:upper_error}.
\begin{proof}[Proof of Proposition \ref{Prop:upper_error}]
    By Lemma \ref{lem:4.5}, we have
    \begin{equation}
    \begin{aligned}
       \sum_{\{l_i\}_{i=0}^{t},|\mathbf{l}|=n}\left|\mathbb{E}\left[\tr \mathcal{V}^{(R)}_{l_0}\underline{\mathcal{E}_{l_1}^{(c_1)}}\ldots \underline{\mathcal{E}_{l_t}^{(c_t)}} \right]\right|&\le \sum_{\{l_i\}_{i=0}^{t},|\mathbf{l}|=n}\mathbb{E}[\tr V^{(R)}_{l_0}{E_{l_1}^{(c_1)}}\ldots {E_{l_t}^{(c_t)}}]\\
       \le &\Big(C\frac{1+\delta_{d,2}\log W}{W^{2d}}\Big)^t \sum_{\{l_i\}_{i=0}^{t},|\mathbf{l}|=n} \mathbb{E}\left[\tr V^{(R)}_{l_0}V_{l_1}^{(R)}\ldots V_{l_t}^{(R)}\right]. 
    \end{aligned}
    \end{equation}
    Thus we finish the proof.
\end{proof}
\begin{definition}\label{equ:prec_E}
    Let $U_{xy}$ be a path 
    from $x$ to $y$ with certain shape constraints, such as, the condition $\mathbb{1}_{0,2l}$ whose first $2l$ steps need to form a double loop, or $\underline{\widetilde{\Phi}_5V_{n-5}}$, $\underline{\widetilde{\Phi}_7V_{n-7}}$ whose first $5$ or $7$ steps traverse the same edge. We classify the index of the paths  into two classes:  
\begin{itemize}
    \item[(i)] \textbf{free index}, one index    not in the shape constraints; 
    \item[(ii)] \textbf{loop index}, one index    under the shape constraints.
\end{itemize}
    For example, in $\underline{\widetilde{\Phi}_5V_{n-5}}$,   the   constraints for $x_0,x_1,x_2,x_3$ are $\delta_{x_0,x_2}\delta_{x_1,x_3}\delta_{x_2,x_4}\delta_{x_3,x_5}$, while   in $\mathbb{1}_{0,2l}$   the   constraints for $x_0,x_1,\ldots x_{2l-1}$ are $\prod_{i=0}^{l}\delta_{x_i,x_{i+l}}$.

    We construct the forgetful functor  $\phi$ from each entry of a matrix $U$ by removing all  the loop indices and shape constraints, and keeping the remaining path, 
   denoted by $\phi(U)=V$. For example, $\phi(\underline{\widetilde{\Phi}_5V_{n-5}})=V_{n-4}^{(R)}$, $\phi(\mathbb{1}_{0,2l}V_{n})=V_{n-2l}^{(R)}$. At this time,  we also denote 
   $U\prec C_W\cdot V$, once   
   fixing    the free indices and  the total weight of the loop indices not greater than $C_W$.     For example, $$\underline{\widetilde{\Phi}_5V_{n-5}}\prec W^{-2d}V_{n-4}^{(R)}, \quad \mathbb{1}_{0,2l}V_{n}^{(R)}\prec l^{-\frac{d}{2}}W^{-d}V_{n-2l}^{(R)}.$$
\end{definition}
\begin{proof}[Proof of Lemma \ref{lem:4.5}]
By the   forgetful functor  $\phi$ in the above definition, we proceed to deal with the three errors respectively. 
\begin{itemize} 
    \item For the first error $\underline{E^{(1)}_{n}}$,   we have
\begin{equation}\label{equ:E1bound}
    \underline{{E}^{(1)}_n}=\underline{(\widetilde{\Phi}_{5}V_{n-5})}+\underline{(\widetilde{\Phi}_{7}V_{n-7})}\prec CW^{-2d} V_{n-4}^{(R)}+CW^{-3d}V_{n-6}^{(R)}
\end{equation} for some constant $C>0$. 
 Here the inequality should be understand as when we take the expectation and consider the counting of path. 
\item For the second  error $\underline{E^{(2)}_{n}}$,  by   \eqref{Lnotation1}  and \eqref{Lnotation2}  we have
\begin{equation}
\begin{aligned}
    &(1-\mathbb{1}(\mathcal{L}\{1;R\}))\mathbb{1}(\mathcal{L}\{[2,n];R\})\prod_{i=1}^{n-2}\mathbb{1}({x_i\ne x_{i+2}})\\
    &\le (\sum_{l=3}^R \mathbb{1}_{1,2l})\cdot \mathbb{1}(\mathcal{L}\{[2,n];R\})\prod_{i=1}^{n-2}\mathbb{1}({x_i\ne x_{i+2}})\\
    &\le \sum_{l=3}^R \mathbb{1}_{1,2l}\mathbb{1}(\mathcal{L}\{[2l+1,n];R\})\prod_{i=2l+1}^{n-2}\mathbb{1}({x_i\ne x_{i+2}}).
\end{aligned}
\end{equation}

On one hand, note that in the last inequality above there is no restriction on $x_i$ with all $i\le 2l$,   and   $x_1=x_{l+1}=x_{2l+1}$ under the condition $\mathbb{1}_{1,2l}$,  
sum over all $x_i$ for $i\le 2l$ and $i\ne 1,l+1,2l+1$ and we use the upper bound of the transition probability  to obtain (see Figure \ref{fig:E2case}) 
\begin{equation}
    \sum_{\{x_i\}_{i=2,\neq l+1,}^{2l}}\mathbb{1}_{1,2l}\delta_{x_{2},x}(\Phi_1)_{xx_{2l+1}}\prod_{i=1}^{2l}(\Phi_1)_{x_ix_{i+1}}\prec p_{1}(x_0,x_1)p_{l-1}(x_1,x_0)(\Phi_{1})_{xx_{2l+1}}\prec Cl^{-\frac{d}{2}}W^{-2d}(\Phi_{1})_{xx_{2l+1}}.
\end{equation}

On the other hand,  summing over  the subscript of $i\ge 2l+1$  gives us 
\begin{equation}
\begin{aligned}
    &\sum_{\{x_i\}_{i=2l+1}^{n}}\delta_{x_{n},y}(\Phi_1)_{xx_{2l+1}}\sum_{l=3}^{R}\mathbb{1}(\mathcal{L}\{[2l+1,n];R\})\prod_{i=2l+1}^{n-2}\mathbb{1}({x_i\ne x_{i+2}})\prod_{i=2l+1}^{n-1}(\Phi_{1})_{x_{i}x_{i+1}}\\
    &=\sum_{x_{2l+1}}(\Phi_1)_{xx_{2l+1}}(V^{(R)}_{n-2l-1})_{x_{2l+1}y}
    =(HV^{(R)}_{n-2l-1})_{xy}.
\end{aligned}
\end{equation}
Combine the two sums together and we have
\begin{equation}\label{equ:4.65}
\begin{aligned}
\underline{{E}^{(2)}_{n}}&\prec CW^{-2d}\sum_{l=3}^{R}l^{-\frac{d}{2}}(HV^{(R)}_{n-2l-1}).
\end{aligned}
\end{equation}

With \eqref{VR} in mind,  we can rewrite the product $HV^{(R)}_{n-2l-1}$ as a sum of three parts 
\begin{multline} 
\label{VR-2}
    (H{V}^{(R)}_{n-2l-1})_{xy}= \sum_{\{x_i\}_{i=0}^{n-2l-1}}
    \delta_{x_0,x}\delta_{x_{n-2l},y}  (\Phi_{1})_{x_{0}x_{1}} \\   
   \cdot 1 \cdot \mathbb{1}(\mathcal{L}\{[1,n-2l-1];R\})\prod_{i=1}^{n-2l-2}\mathbb{1}({x_i\ne x_{i+2}})\prod_{i=1}^{l}(\Phi_{1})_{x_{i}x_{i+1}},
\end{multline} 
where the constant $1$ can be 
  split as  
\begin{equation}
    1=\mathbb{1}(x_0\ne x_2)\mathbb{1}(\mathcal{L}\{0;R\})+\mathbb{1}(x_0=x_2)+\mathbb{1}(x_0\ne x_2)(1-\mathbb{1}(\mathcal{L}\{0;R\})).
\end{equation}
The  sum associated with the first  term gives $V^{(R)}_{n-2l}$,  and the sum  over $x_1$  relevant to  the second term can be controlled   by $V^{(R)}_{n-2l-2}$ up to the relation $\prec$. However, in order  to deal with   the third term, we use the simple inequality   
\begin{equation}
    \mathbb{1}(x_0\ne x_2)(1-\mathbb{1}(\mathcal{L}\{0;R\}))\le \sum_{l'=3}^R\mathbb{1}_{0,2l'},
\end{equation}
to obtain 
\begin{equation}
    \sum_{l=3}^{R}\frac{1}{l'^{\frac{d}{2}}W^d}V^{(R)}_{n-2l-2l'}.
\end{equation}
This case gives
\begin{equation}
    CW^{-2d}\sum_{l=3}^Rl^{-\frac{d}{2}}l'^{-\frac{d}{2}}W^{-d}V^{(R)}_{n-2l-2l'}\prec CW^{-2d}\sum_{l=3}^{2R}l^{-\frac{d}{2}}(V^{(R)}_{n-2l})
\end{equation}

In summary, with the discussion above and  after some simple calculations,  we see from   \eqref{equ:4.65} that 
\begin{equation}\label{equ:E2bound}
    \underline{{E}^{(2)}_{n}}\prec CW^{-2d}\sum_{l=3}^{2R}l^{-\frac{d}{2}}(V^{(R)}_{n-2l}).
\end{equation} 
\item For the third error  term
$\underline{E^{(3)}_{n}}$, we first give an upper bound for $B_{2l}$  given in \eqref{B2l} as follows: 
\begin{equation}
\begin{aligned}
    B_{2l}&=\prod_{i=0}^{2l-2}\mathbb{1}(x_i\ne x_{i+2})\prod_{i=2l}^{n-2}\mathbb{1}(x_i\ne x_{i+2})\mathbb{1}(\mathcal{L}\{[2l,n];R\})\Big(1-\mathbb{1}(x_{2l-1}\ne x_{2l+1})\prod_{i=1}^{2l-1}\mathbb{1}(\mathcal{L}\{i;R\})\Big)\\
    &\le \prod_{i=2l}^{n-2}\mathbb{1}(x_i\ne x_{i+2})\mathbb{1}(\mathcal{L}\{[2l,n];R\})\Big(1-\mathbb{1}(x_{2l-1}\ne x_{2l+1})\prod_{i=1}^{2l-1}\mathbb{1}(\mathcal{L}\{i;R\})\Big)\\
    &\le \prod_{i=2l}^{n-2}\mathbb{1}(x_i\ne x_{i+2})\mathbb{1}(\mathcal{L}\{[2l, n];R\})\Big(\delta_{x_{2l-1},x_{2l+1}}+\sum_{i=1}^{2l-1}\sum_{t=3}^{R}\mathbb{1}_{i,2t}\Big).
\end{aligned}
\end{equation}
Next, according to the  last equality  we  need to study  the sum \eqref{equ:E3_4.56} respectively in   
Case (i) $\delta_{x_{2l-1},x_{2l+1}}$, and  Case  (ii) $\mathbb{1}_{i,2t}$.  
\begin{itemize}
    \item[\textbf{Case(i):}] Sum over all possible $x_i$, $i=0,\ldots, 2l-1$, and  we have  
    \begin{equation}
        \sum_{x_0,\ldots x_{2l-2}}\mathbb{1}_{0,2l}\prod_{i=0}^{2l-1}(\Phi_1)_{x_ix_{i+1}}\delta_{x_{2l-1},x_{2l+1}}=p_{l-1}(x_{2l-1},x_{2l+1})p_1(x_{2l-1},x_{2l+1})\le Cl^{-\frac{d}{2}}W^{-2d},
    \end{equation}
    from which   the sum involving  $\delta_{x_{2l-1},x_{2l+1}}$ can be controlled by, up to the relation   $\prec$,  
    \begin{equation}
        \sum_{l=3}^{R}Cl^{-\frac{d}{2}}W^{-2d}V_{n-2l}^{(R)}.
    \end{equation}
    \item[\textbf{Case(ii):}] For any $l$, we claim the following estimate
     \begin{equation}\label{equ:B.48}
        \sum_{x_0,\ldots x_{2l-1}}\mathbb{1}_{0,2l}\sum_{i=1}^{2l-1}\sum_{t=3}^{R}\mathbb{1}_{i,2t}\prod_{i=0}^{2l-1}(\Phi_1)_{x_ix_{i+1}}\le Cl^{-\frac{d}{2}}W^{-2d}.
    \end{equation}
    To prove it we need to split the sum $\sum_{i=1}^{2l-1}\sum_{t=3}^R \mathbb{1}_{i,2t}$  into three  cases: (a) $i\ne 2l-1, t\ne l$, (b) $i\ne 2l-1,t=l$ and (c) $i=2l-1$, as   illustrated in Figure \ref{fig:case2a} and \ref{fig:case2c}.
    
    \begin{itemize}
        \item[\textbf{Case (ii.a):}] In this case,  we have $\mathbb{1}_{i,2t}\le \delta_{x_{i},x_{i+t}}\delta_{x_{i+1},x_{i+t+1}}.$ As shown in Figure \ref{fig:case2a}, there are further  different cases according to  $i+t<2l$ or $i+t\ge 2l$. In the case $i+t<2l$, without loss of generality we let $i<l$ and $x_i$ split the loop into two loops of length $l_1,l_2$ such that $l_1+l_2=l-1$
        \begin{equation}
        \begin{aligned}
         \sum_{x_0,\ldots x_{2l-1}}\mathbb{1}_{0,2l}\mathbb{1}_{i,2t}\prod_{i=0}^{2l-1}(\Phi_1)_{x_ix_{i+1}}&\le \sum_{x_0,\ldots x_{2l-1}}\mathbb{1}_{0,2l}\delta_{x_{i},x_{i+t}}\delta_{x_{i+1},x_{i+t+1}}\prod_{i=0}^{2l-1}(\Phi_1)_{x_ix_{i+1}}\\
         &\le \sum_{x_{i},x_{i+1}}p_{i}(x_0,x_i)p_{l_1-i}(x_0,x_i)p_{l_2}(x_{i+1},x_{i+1})p_1(x_i,x_{i+1})^2\\
         &\le W^{-d}\sum_{x_{i},x_{i+1}}p_{i}(x_0,x_i)p_{l_1-i}(x_0,x_i)p_{l_2}(x_{i+1},x_{i+1})p_1(x_i,x_{i+1})\\
         &\le \sum_{x_{i}}p_{i}(x_0,x_i)p_{l_1-i}(x_0,x_i)p_{l_2}(0,0)\\
         &\le C_1W^{-3d}l_1^{-\frac{d}{2}}l_{2}^{-\frac{d}{2}}.
        \end{aligned}
        \end{equation} 
        A similar   result holds  in the case  that $i+t>2l$. Together,   we know from  \eqref{equ:B.48} that 
        \begin{equation}
        \begin{aligned}
            &\sum_{x_0,\ldots x_{2l-1}}\mathbb{1}_{0,2l}\sum_{=1}^{2l-2}\sum_{t=3,t\ne l}^R\mathbb{1}_{i,2t}\prod_{i=0}^{2l-1}(\Phi_1)_{x_ix_{i+1}}\\
            &\le C\sum_{l_1=1}^{l-1}\sum_{t\le R}W^{-3d}l_1^{-\frac{d}{2}}(l-l_1)^{-\frac{d}{2}}\le \frac{CR(1+\delta_{d,2}\log R)}{W^d}l^{-\frac{d}{2}}W^{-2d}.
        \end{aligned}
        \end{equation}
        
        \item[\textbf{Case (ii.b):}] In   this case we take $i\ge l$,  as illustrated in  the lower part of Figure \ref{fig:case2a}.  We proceed in a similar way to      \textbf{Case (ii.a)} to obtain  
        \begin{equation}
        \begin{aligned}
            &\sum_{x_0,\ldots x_{2l-1}}\mathbb{1}_{0,2l}\sum_{i=1}^{2l-2}\mathbb{1}_{i,2l}\prod_{i=0}^{2l-1}(\Phi_1)_{x_ix_{i+1}}\le \frac{C(1+\delta_{d,2}\log R)}{W^d}l^{-\frac{d}{2}}W^{-2d}.
        \end{aligned}
        \end{equation}
        \item[\textbf{Case (ii.c):}] In this case, there will be another loop starting  from $x_{2l-1}$, as shown in Figure \ref{fig:case2c}. We need to sum up  $x_i$ not only for $i=0,\ldots 2l-1$ but also   for $i=2l,\ldots,2l+2t-1$.  This  gives
        \begin{equation}
        \begin{aligned}
            &\sum_{x_0,\ldots x_{2l-1},\ldots x_{2l+2t-2}}\mathbb{1}_{0,2l}\sum_{t=3}^{R}\mathbb{1}_{2l-1,2t}\prod_{i=0}^{2l-1}(\Phi_1)_{x_ix_{i+1}}\\
            &\le \sum_{t=3}^R\sum_{x_0}p_{l-1}(x_0,x_{2l-1})p_{t-1}(x_0,x_{2l-1})p_{1}(x_0,x_{2l-1})^2\\
            &\le CW^{-2d}\sum_{t=3}^R\sum_{x_0}p_{l-1}(x_0,x_{2l-1})p_{t-1}(x_0,x_{2l-1})\\
            &\le CW^{-2d}\sum_{t=3}^Rp_{l+t-2}(x_{2l-1},x_{2l-1})
            \le CW^{-3d}\sum_{t=3}^R(l+t)^{-\frac{d}{2}}.
        \end{aligned}
        \end{equation}
    \end{itemize}

  With  the inequality \eqref{equ:B.48}, similar to the argument used in the second error term, combine     
  Case(i) and Case(ii) and  we obtain
\begin{equation}\label{equ:E3bound}
\begin{aligned}
    \underline{E^{(3)}_{n}}&\prec \sum_{l=3}^{R}\frac{C(1+\delta_{d,2}\log R)}{W^d}l^{-\frac{d}{2}}W^{-2d}V_{n-2l}^{(R)}+\sum_{l=3}^R\sum_{t=3}^R C(l+t)^{-\frac{d}{2}}W^{-3d}V_{n-2l-2t}^{(R)}\\
    &\prec \frac{1+\delta_{d,2}\log R}{W^d} \sum_{l=3}^{2R} Cl^{-\frac{d}{2}}W^{-2d}V_{n-2l}^{(R)}.
\end{aligned}
\end{equation}

\end{itemize}
\end{itemize}

Finally, for simplicity  we  just prove \eqref{equ:4.60} in a special case of $t=3$ and $c_1=1,c_2=2,c_3=3$.
By \eqref{equ:E1bound},\eqref{equ:E2bound} and \eqref{equ:E3bound}, we have
\begin{equation}
    \mathbb{E}[\tr V^{(R)}_{l_0}{E_{l_1}^{(1)}}{E_{l_2}^{(2)}}{E_{l_3}^{(3)}}]\le (C_1W^{-2d})^3
    \sum_{ t_1,t_2,t_3=3}^{2R} 
    t_1^{-\frac{d}{2}}t_2^{-\frac{d}{2}}t_3^{-\frac{d}{2}}\mathbb{E}\left[\tr V^{(R)}_{l_0}V_{l_1-2t_1}^{(R)}V_{l_2-2t_2}^{(R)}V_{l_3-2t_3}^{(R)}\right].
\end{equation}
By Lemma \ref{lem:lem4.5_monotonicity}, we have
\begin{equation}
    \mathbb{E}\left[\tr V^{(R)}_{l_0}V_{l_1-2t_1}^{(R)}V_{l_2-2t_2}^{(R)}V_{l_3-2t_3}^{(R)}\right]\le \big(1+O(\frac{R(1+\delta_{d,2}\log R)}{W^d})\big)^3\mathbb{E}\left[\tr V^{(R)}_{l_0+t_1}V_{l_1-t_1+t_2}^{(R)}V_{l_2-t_2+t_3}^{(R)}V_{l_3-t_3}^{(R)}\right].
\end{equation}
We also have 
\begin{equation}
    \sum_{\{l_i\}_{i=0}^{3},|\mathbf{l}|=n} \mathbb{E}\left[\tr V^{(R)}_{l_0+t_1}V_{l_1-t_1+t_2}^{(R)}V_{l_2-t_2+t_3}^{(R)}V_{l_3-t_3}^{(R)}\right]\le \sum_{\{l_i\}_{i=0}^{3},|\mathbf{l}|=n} \mathbb{E}\left[\tr V^{(R)}_{l_0}V_{l_1}^{(R)}V_{l_2}^{(R)}V_{l_3}^{(R)}\right],
\end{equation}
so  for any given $d>1$
\begin{equation}
\begin{aligned}
     \sum_{\{l_i\}_{i=0}^{3},|\mathbf{l}|=n}&\mathbb{E}[\tr V^{(R)}_{l_0}{E_{l_1}^{(1)}}{E_{l_2}^{(2)}}{E_{l_3}^{(3)}}]\le \Big(1+O\big(\frac{R(1+\delta_{d,2}\log R)}{W^d}\big)\Big)^3(CW^{-2d})^3\\
    &\times \sum_{t_1,t_2,t_3=3}^{R}(t_1t_2t_3)^{-\frac{d}{2}}\sum_{\{l_i\}_{i=0}^{3},|\mathbf{l}|=n} \mathbb{E}\left[\tr V^{(R)}_{l_0}V_{l_1}^{(R)}V_{l_2}^{(R)}V_{l_3}^{(R)}\right]\\
    &\le \Big(C\frac{1+\delta_{d,2}\log W}{W^{2d}}\Big)^3 \sum_{\{l_i\}_{i=0}^{3},|\mathbf{l}|=n} \mathbb{E}\left[\tr V^{(R)}_{l_0}V_{l_1}^{(R)}V_{l_2}^{(R)} V_{l_3}^{(R)}\right].
\end{aligned}
\end{equation}

Thus, this completes the proof of  the desired result.
\end{proof}

\begin{figure}[htbp]
    \centering
    \begin{minipage}{0.48\textwidth}
        \centering
        \includegraphics[width=\textwidth]{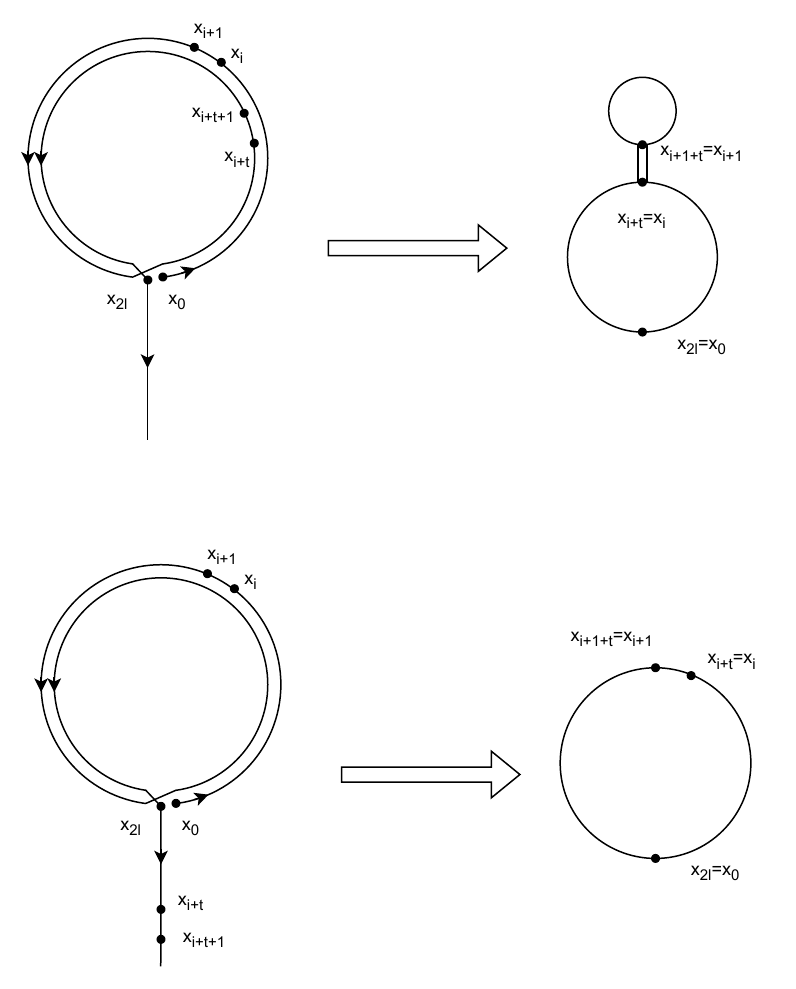}
        \caption{case(ii.a) example}
        \label{fig:case2a}
    \end{minipage}
    \hfill
    \begin{minipage}{0.48\textwidth}
        \centering
        \includegraphics[width=\textwidth]{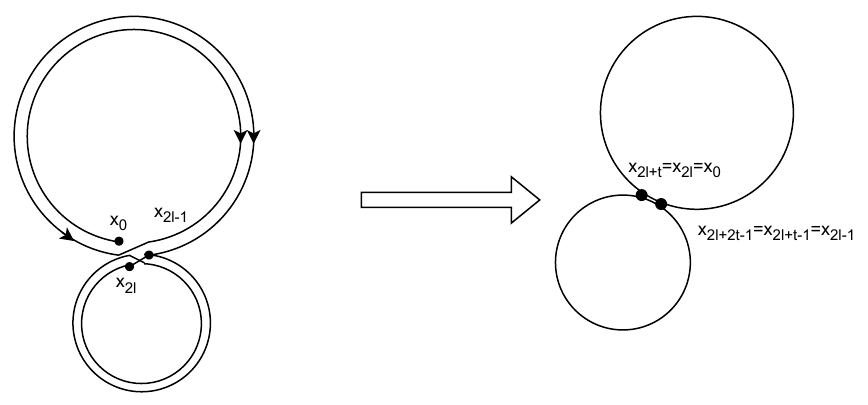}
        \caption{case(ii.c) example}
        \label{fig:case2c}
    \end{minipage}
\end{figure}
\begin{figure}[htbp]
    \centering
    \begin{minipage}{0.45\textwidth}
        \centering
        \includegraphics[width=0.9\textwidth]{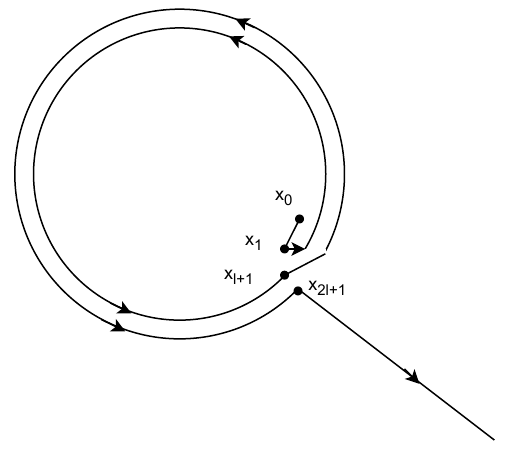}
        \caption{$E^{(2)}$ case}
        \label{fig:E2case}
    \end{minipage}
    \hfill
    \begin{minipage}{0.45\textwidth}
        \centering
        \includegraphics[width=0.9\textwidth]{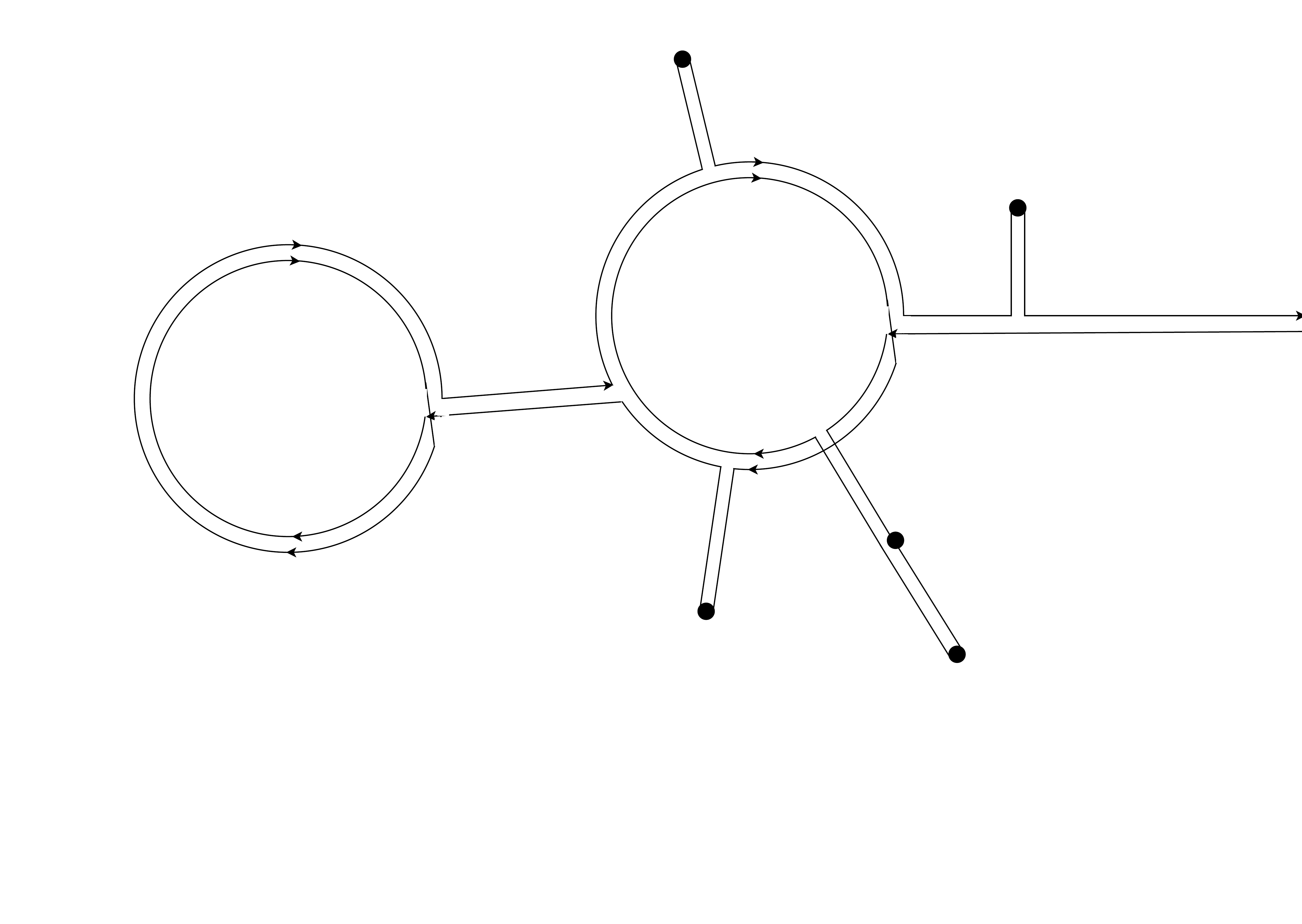}
        \caption{Diagram with extra tail edges}
        \label{fig:leg_diagram}
    \end{minipage}
\end{figure}

\subsection{Dominating polynomial moments }\label{sec:4.2}

We proceed as in almost the same way in  Section \ref{sec:P_powers} to  deal with  the diagram representation for the path in $\mathbb{E}[\tr V^{(R)}_{l_0}V^{(R)}_{l_1}\cdots V^{(R)}_{l_t}]$.

The diagrammatic reduction technique introduced in Section \ref{sec:diagram_reduction} and Section \ref{sec:P_powers} can be applied to the expectation value $\mathbb{E}[\tr V_{l_0}^{(R)}V_{l_1}^{(R)}\ldots V_{l_t}^{(R)}]$. The resulting diagrams correspond to those defined in Definition \ref{def:(k,t)diagram}, but with   stricter constraints on self-loops (see Figure \ref{fig:leg_diagram} for an illustrative example). Crucially, each piece $V_{l}^{(R)}$ inherently excludes self-loops shorter than length $R$. Consequently, the tadpole structure in Figure \ref{fig:leg_diagram} must have a minimum length of $R$.

For such a diagram, the diagram function may be  still defined as 
    \begin{equation}
    \begin{aligned}
        F^{(R)}_{\mathfrak{D}}(\{n_i\}_{i=1}^k)&:=\frac{1}{N}\sum_{w(e):\mathfrak{C}^{(R)}} \sum_{\iota:\bar{V}\rightarrow [N]}\prod_{\bar{e}\in \bar{E}}p_{w(\bar{e})}(\iota(u_{\bar{e}}),\iota(v_{\bar{e}})),
    \end{aligned}
    \end{equation}
    where $\mathfrak{C}^{(R)}$ is the linear system defined in Definition \ref{def:diagram} but with the  replacement of the  assumption \eqref{equ:2.56} by 
\begin{equation}\label{equ:4.89}
    w(e)\ge\begin{cases}
       R,~~~~\text{if $e$ is a self-loop};\\
        1,~~~~\text{if $e$ is not a tail edge}.
   \end{cases} 
\end{equation} 

After tadpole renormalization, all diagrammatic singularities are eliminated. The limiting behavior and uniform bounds of the renormalized diagram functions  
 are fully characterized by 
  theorem \ref{thm:homogeneous_theorem} and proposition \ref{thm:upper_bound_diagram_function}.

There is no singularity problem after the tadpole renormalization. The diagram function's limit behaviors and upper bounds are given by Theorem \ref{thm:homogeneous_theorem} and Proposition \ref{thm:upper_bound_diagram_function}. To be precise, 
\begin{itemize}
    \item In the subcritical and critical  cases, we choose $R=n$.
    This parameter choice explicitly prohibits self-loops in the diagrammatic expansion. In spatial dimensions  $d<4$, all remaining diagrams remain regular   by Proposition~\ref{propositionconvergence}.

    \item In the supercritical case, we take $R=\frac{L^2}{W^2}$. The renormalized diagram functions converge to identical limits as those governing GOE/GUE  ensembles. This universality persists despite the modified renormalization scale. 
\end{itemize}

For a diagram with $t$ extra edges and for $k=1$, we have $|E|=3s+2t-1$ and $|V|=2s+2t$.
 By Theorem \ref{thm:homogeneous_theorem}, we obtain for even $n$  
\begin{equation}
    F_{\mathfrak{D}}(n)=\begin{cases}
        CW^{d(-E+V-1)}n^{E-1-\frac{d}{2}(E-V+1)},~~~&n(\frac{W}{L})^2\ll 1;\\
        C n^{E-1}N^{V-E-1},~~~&n(\frac{W}{L})^2\gg 1.
    \end{cases}
\end{equation}
Furthermore, by Proposition \ref{Prop:upper_error} we have an additional factor $W^{-2dt}$ and 
\begin{equation}\label{equ:5.83}
    W^{-2dt}F_{\mathfrak{D}}(n)=\begin{cases}
        CW^{-d(s+2t)}n^{3s+2t-2-\frac{ds}{2}},~~~&n(\frac{W}{L})^2\ll 1;\\
        C W^{-2dt}n^{3s+2t-2}N^{-s},~~~&n(\frac{W}{L})^2\gg 1.
    \end{cases}
\end{equation}
By factoring out the scaling term  $(nW^{-d})^{2t}$
  in each  case, we observe that all diagram contributions vanish asymptotically when 
  $n\ll W^d$. 
  
  For general $k$, we can extend our definition of $k$-diagram to $k$-diagram with tail edges. Combine Proposition \ref{Prop:4.2}  and Proposition \ref{Prop:upper_error},  and we have indeed  proved the following corollary.

\begin{corollary}\label{coro:4.4}
For  $d<4$, if  $n\ll W^d$ as $W\to \infty$,  then 
    \begin{equation}
        \mathbb{E}\Big[\prod_{i=1}^k \tr \widetilde{\mathcal{P}}_{n_i}(H)\Big]=\big(1+o(1)\big)\mathbb{E}\Big[\prod_{i=1}^k\tr V_{n_i}^{(R)}\Big].
    \end{equation}
\end{corollary}
\begin{proof}
For simplicity, we only give the proof in  the case $k=1$ since the case  $k>1$ is similar. For $k=1$, we have 
\begin{equation}
    \mathbb{E}[\tr \widetilde{P}_{2n+1}(H)]=0,
\end{equation}
and 
\begin{equation}
\begin{aligned}
\mathbb{E}[\tr \widetilde{P}_{2n}(H)]&=\sum_{t\ge 0}\sum_{\{l_i\}_{i=0}^{t},|\mathbf{l}|=n}\sum_{\{c_i\}_{i=1}^{t}}\mathbb{E}[\tr \mathcal{V}^{(R)}_{l_0}\underline{\mathcal{E}_{l_1}^{(c_1)}}\ldots \underline{\mathcal{E}_{l_t}^{(c_t)}}]\\
&=\mathbb{E}[\tr \mathcal{V}^{(R)}_{2n}(H)]+\sum_{t\ge 1}\sum_{\{l_i\}_{i=0}^{t},|\mathbf{l}|=n}\sum_{\{c_i\}_{i=1}^{t}}\mathbb{E}[\tr \mathcal{V}^{(R)}_{l_0}\underline{\mathcal{E}_{l_1}^{(c_1)}}\ldots \underline{\mathcal{E}_{l_t}^{(c_t)}}].
\end{aligned}
\end{equation}
By Proposition \ref{Prop:upper_error}, we know that
\begin{equation}
\begin{aligned}
   &\sum_{t\ge 1}\sum_{\{l_i\}_{i=0}^{t},|\mathbf{l}|=n}\sum_{\{c_i\}_{i=1}^{t}}\mathbb{E}[\tr \mathcal{V}^{(R)}_{l_0}\underline{\mathcal{E}_{l_1}^{(c_1)}}\ldots \underline{\mathcal{E}_{l_t}^{(c_t)}}]\\
   &=\sum_{t\ge 1}(C\frac{1+\delta_{d,2}(\log W)^2}{W^d}\big)^t\sum_{\{l_i\}_{i=0}^{t},|\mathbf{l}|=n}\mathbb{E}\left[\tr V^{(R)}_{l_0}V_{l_1}^{(R)}\ldots V_{l_t}^{(R)}\right]. 
\end{aligned}
\end{equation}
By adopting the diagram expansion and using \eqref{equ:5.83}, we know that the above sum is negligible compared to $\mathbb{E}[\tr {V}^{(R)}_{2n}(H)]$.

Following the argument of Lemma \ref{lem:lemma2.7}, we know that 
\begin{equation}
    \mathbb{E}[\tr \mathcal{V}^{(R)}_{2n}(H)]=(1+O(\frac{n}{W^d}))\mathbb{E}[\tr {V}^{(R)}_{2n}(H)].
\end{equation}
Thus we finish the proof.
\end{proof}

\subsection{Proofs of Theorem  \ref{metatheoremsup}--\ref{metatheoremcri}: \texorpdfstring{$\beta=1$}{β=1} case}\label{sec:beta=1_main_thm}
Corollary \ref{coro:4.4} shows that the main term of mixed moment of $\widetilde{\mathcal{P}}_n(H)$ is the mixed moment of $V_n^{(R)}$, where the later term's asymptotics are given by the limit of diagram functions studied before, e.g. Theorem \ref{thm:homogeneous_theorem}. Hence similar results of Theorem \ref{prop:super_4} and Theorem \ref{prop:sub_4} apply for $\widetilde{\mathcal{P}}_n(x)$ case. We can  proceed as in  Section \ref{sec:P_powers} to give the proof.

\begin{theorem}\label{prop:super_4_P} Given any fixed integers  $k\ge 1$ and  $t_i\ge 1$ for $1\le i\le k$,  with the same assumption on  $W$ as in  Theorem \ref{metatheoremsup}, if  
    \begin{equation}
        n_1\le n_2\le \cdots \leq n_k\le \tau N^{\frac{1}{3}}
    \end{equation}
    for any fixed constant $\tau$  and  $t_1n_1+\cdots+t_kn_k$ is even, then   
    \begin{equation}
        \mathbb{E}\big[\prod_{i=1}^{k}\tr (\widetilde{\mathcal{P}}_{n_i}(H))^{t_i}\big]=\big(1+o(1)\big)\mathbb{E}\big[\prod_{i=1}^{k}\tr (U_{n_i}(H_{\mathrm{GOE}}/2))^{t_i}\big].
    \end{equation}
\end{theorem}
\begin{theorem}\label{prop:sub_4_P} 
Given any fixed integers  $k\ge 1$ and  $t_i\ge 1$ for $1\le i\le k$, assume that     $d<4$ and   $W\ll L^{1-\frac{d}{6}}$. If for any fixed constant $\tau$  and 
\begin{equation}
   1\le n_1\le n_2\le \cdots \le n_k\le \tau W^{\frac{2d}{6-d}}
\end{equation}
and all  $t_1n_1,\ldots ,t_kn_k$ are  even, then  
    \begin{equation}
        \mathbb{E}\big[\prod_{i=1}^{k}\tr (\widetilde{\mathcal{P}}_{n_i}(H))^{t_i}\big]=\big(1+o(1)\big)\prod_{i=1}^{k}\mathbb{E}\big[\tr (\widetilde{\mathcal{P}}_{n_i}(H))^{t_i}\big].
    \end{equation}
\end{theorem}
\begin{theorem}\label{prop:critical_4_P}
    Given  $k\ge 1$ and $d<4$,  if $W= [\gamma L^{1-\frac{d}{6}}]$, for any    fixed constants  $\tau_i>0$ and $n_i=[\tau_i(\frac{L}{W})^2]$,
such that $n_1+\cdots+ n_k$ is even, then we have
\begin{equation}
    \mathbb{E}\big[\prod_{i=1}^{k}\tr (\frac
    {1}{n_i}\widetilde{\mathcal{P}}_{n_i}(H))^{t_i}\big]=(1+o(1))\phi_{\{t_i\}}^{(\mathrm{crit})}(\beta,d,2, \gamma;\tau_1, \ldots,\tau_k).
\end{equation}
Here $\phi_{\{t_i\}}^{(\mathrm{crit})}(\beta,d,2,\gamma;\tau_1, \ldots,\tau_k)$ is defined in Definition \ref{def:sub_critical_transform}.
\end{theorem}
\begin{proposition}\label{thm:T_t_upper_bound_P}
In real symmetric case, given integers $k\geq 1$ and   $t_1, \ldots, t_k\geq 1$,  let $n=\sum_{i=1}^k t_in_i$ and $t=\sum_{i=1}^k t_i$.
If 
  $W\rightarrow \infty$ and $ne^{-0.5c_\Sigma W^2}\rightarrow 0$,  and $d'\le (4-\epsilon)\land d$ for any fixed $\epsilon>0$,   then 
    \begin{equation}
    \begin{aligned}
        T\big(\{(n_i,t_i)\}_{i=1}^k\big)&\le  (Cn)^t\exp\big\{C_1 n^{\frac{3}{2}}N^{-\frac{1}{2}}\big\}+(Cn)^t\frac{N}{n^{\frac{d'}{2}}W^d}\exp\Big\{C_2n^{\frac{6-d'}{4-d'}}W^{-\frac{2d}{4-d'}}\Big\}+\delta_{k,1}(Cn)^{t-3}N.
    \end{aligned}
    \end{equation}
\end{proposition}
However, due to the lack of tail estimate of polynomial $\widetilde{\mathcal{P}}(x)$ in the subcritical case, we can only prove the tail probability decay of the largest eigenvalue for supercritical case, by applying Proposition \ref{thm:T_t_upper_bound_P}, Lemma \ref{lem:5.8} and \ref{lem:B.8}.
\begin{theorem}[\textbf{Tail probability decay}]
\label{thm:upper_bound_largest_eigenvalue_P}
For the unimodular  \rm{RBM} with $\beta=1$, as $W\to \infty$ 
the following right  tail probability estimate holds for any  $x\ge 0$. Under  the   {\bf Supercritical Assumption}, we have right tail estimate of the Tracy-Widom distribution
    \begin{equation}
        \mathbb{P}(\lambda_{\mathrm{max}}(H)\ge 2+A_d+xN^{-\frac{2}{3}})\le C_1 e^{-C_2|x|^{\frac{3}{2}}}.
    \end{equation}
Here 
\begin{equation}\label{equ:def_Ad}
    A_d:=\begin{cases}
        -a_4+\sum_{l=3}^{W}a_{2l}=\Theta(\frac{\log W}{W^d}),&d=2\\
        -a_4+\sum_{l=3}^{\infty}a_{2l}=\Theta(\frac{1}{W^d}),&d>2.\\
    \end{cases}
\end{equation}
\end{theorem}

With all the preparation above, following the proofs in Section \ref{sec:beta=2_main_thm}, we arrive the proof of main theorem in $\beta=1$ case.
\begin{proof}[Proof of Theorem \ref{metatheoremsup}: $\beta=1$]
   Take all $t_i=4, 8$  and $n_i=[\tau_iN^{\frac{1}{3}}]$ in  Theorem \ref{prop:super_4_P},   we have  
    \begin{equation}\label{equ:4.64beta=1}
        \mathbb{E}\big[\prod_{i=1}^{k}\tr (\widetilde{\mathcal{P}}_{n_i}(H))^4\big]=(1+o(1))\mathbb{E}\big[\prod_{i=1}^{k}\tr (U_{n_i}(H_{\mathrm{GOE}}/2))^4]\big]
    \end{equation}
    and 
    \begin{equation}\label{equ:P8=U8beta=1}
        \mathbb{E}\big[\prod_{i=1}^{k}\tr (\widetilde{\mathcal{P}}_{n_i}(H))^8\big]=(1+o(1))\mathbb{E}\big[\prod_{i=1}^{k}\tr (U_{n_i}(H_{\mathrm{GOE}}/2))^8]\big].
    \end{equation}
The remaining steps  are exactly the same as in the $\beta=2$ case, by changing $a_4$ to $A_d$ defined in \eqref{equ:def_Ad}.
\end{proof}

\begin{proof}[Proof of Theorem \ref{metatheoremsub}:  $\beta=1$]
    Use Theorem \ref{prop:sub_4_P} and we   know that
    \begin{equation}
        \mathbb{E}\big[\prod_{i=1}^{k}\tr \widetilde{\mathcal{P}}_{n_i}(H)\big]=(1+o(1))\prod_{i=1}^{k}\mathbb{E}\big[\tr \widetilde{\mathcal{P}}_{n_i}(H))\big].
    \end{equation}
    On the other hand, consider the $\beta=1$ version of Proposition \ref{prop:5.4},  
    \begin{equation}
    \mathbb{E}\big[\tr \widetilde{\mathcal{P}}_{n_i}(H))\big]=(1+o(1))T(n_i)=(1+o(1))\frac{N}{W^{\frac{6d}{6-d}}} n_i \psi_1(\frac{n_i}{W^{\frac{2d}{6-d}}}),
    \end{equation}
    from which the desired result  immediately follows.  
\end{proof}

\begin{proof}[Proof of Theorem \ref{metatheoremcri}: $\beta=1$]
    The proof is immediate by taking all $t_i=1$ in Theorem \ref{prop:critical_4_P}.
\end{proof}

\newpage
\section{Concluding   remarks and  open questions} \label{openquestions}

We conclude our investigation with several important observations and open questions that naturally arise from this work.

\begin{enumerate}  
    \item[\textbf{I.}]\textbf{Universality and generalization.}
    While our current analysis focuses on unimodular random variables and Gaussian   profile functions for clarity of presentation, we conjecture that all main results (Theorems \ref{metatheoremsup}, \ref{metatheoremsub}, and \ref{metatheoremcri}) extend to i.i.d. variables with higher moments and general profile functions, potentially building on techniques from \cite{feldheim2010universality,erdHos2011quantum2,erdHos2015altshuler2}. 

\begin{question}[Universality Conjecture]\label{Universality}
    Prove that Theorems \ref{metatheoremsup}, \ref{metatheoremsub}, and \ref{metatheoremcri} remain valid for general i.i.d. entries and general  profile functions. We conjecture that $a_4$ is relevant to the fourth moment of the random variables, a more correct form of $a_4=\sum_{y\in \Lambda_L}(2\sigma_{xy}^4 - \mathbb{E}|H_{xy}|^4)$.
\end{question}

\begin{question}[Heavy-tailed Random Band Matrices]\label{htRBM}
    Investigate the spectral properties of heavy-tailed random band matrices, particularly regarding possible new phenomena. Recent progress on heavy-tailed Wigner matrices \cite{aggarwal2021goe,aggarwal2022mobility} may provide valuable insights.
\end{question}    
    \item[\textbf{II.}]\textbf{Edge eigenvector statistics.} 
    Building on the observation in \cite[Remark VI in Sect. 9]{sodin2010spectral}, we expect the Chebyshev (renormalization) polynomial  expansion method to provide new insights into eigenvector localization-delocalization phenomena at the spectral edge. 

    \begin{question}[{\textbf{Eigenvector statistics}}] \label{Vector}
        Can the method of  renormalization polynomial expansion  be used    to characterize the    eigenvector localization,   delocalization  and the  transition  for  RBM at the spectral edge?  
    \end{question}

    \item[\textbf{III.}]\textbf{Challenges in higher dimensions.} 
    
    The extension of our results to dimensions $d \geq 4$ presents several fundamental obstacles:

\begin{itemize}
    \item \textbf{Random walk approximation.} Propositions \ref{proposition:Gaussian_asymptotic_independence} and \ref{proposition:Gaussian_self_avoiding} require the condition $n \ll W^d$, which fails for $d \geq 4$. This breakdown prevents the approximation of paths by independent products of normal random walks. While the relevant diagrams are self-avoiding, their direct computation remains challenging. We note recent progress on self-avoiding walks (SAW for short)\cite{van2003lace,michta2023scaling} may provide potential approaches.
    
    \item \textbf{Singular diagram analysis.} Current mathematical techniques can only rigorously handle the tadpole diagram, despite extensive physical literature on diagram renormalization \cite{collins1984renormalization}. The treatment of more singular patterns remains an open problem.
    
    \item \textbf{Mismatched decay rates.} Theorem \ref{thm:upper_bound_largest_eigenvalue} suggests tail decay of order $e^{-C|x|^{(6-d)/4}}$, while the modified Chebyshev polynomials grow as $\mathcal{P}_n(2+x/n^2) \geq e^{C\sqrt{x}}$. This incompatibility makes Chebyshev polynomials unsuitable test functions for $d > 4$.
\end{itemize}

\begin{question}[Singular diagram correction]\label{q:singular}
     For dimensions $d \geq 4$, determine the quantitative impact of non-tadpole singular diagrams (Fig. \ref{singular_diagram_ks<3} and  Fig. \ref{singular_diagram_5d}) on the modified Chebyshev expansion. 
\end{question}

\begin{question}[SAW in higher dimensions]\label{q:saw}
    Characterize how self-avoiding walk  properties in $d \geq 4$ affect edge statistics when the condition $n \ll W^d$ fails. 
\end{question}

Resolution of these two questions would extend the Chebyshev polynomial method to $d=4$. However, since the critical dimension appears to be $d_c=6$, which is  consistent with $\phi^3$ field theory, we propose separating the remaining challenges into two cases

    \begin{question}[{\textbf{Phase transition for   $4\le d< 6$}}]\label{PT4-6}
    Characterize the phase 
    transition of  eigenvalues
    at the spectral edge after appropriate renormalization when  
   dimensions $d \in [4,6)$.
\end{question}

    \begin{question}[{\textbf{Universality in higher dimensions $d\ge 6$}}]\label{TW6}
Investigate whether GOE/GUE edge statistics  hold    in dimensions $d \geq 6$ under 
      the polynomial growth condition $W \gg L^\epsilon$ for any $\epsilon > 0$, or 
          the minimal width requirement $W > C_d$, where $C_d$ is a dimension-dependent critical threshold.
    \end{question}

    \item[\textbf{IV}]\textbf{Subcritical regime.} 
    The subcritical point process can be understood as a superposition of $NW^{-\frac{6d}{6-d}}$ non-Poisson processes, as evidenced by Proposition \ref{prop:5.4} which shows that all connected cumulants $T(\{n(i)\})$ scale as $NW^{-\frac{6d}{6-d}}$. This raises the question that identifying  the limiting universal point process whose connected cumulants are given by the rescaled form $\frac{T(\{n(i)\})}{NW^{-\frac{6d}{6-d}}}$, independent of specific profile function
            and  entry distribution.

\begin{question}[Characterization of Subcritical Process]\label{q:subcrit_process}
    Establish a complete description of the non-Poisson subcritical point process, analogous to the Airy point  process. 
\end{question}

The superposition structure of independent processes naturally suggests the following conjecture regarding extremal statistics:
\begin{question}[Top eigenvalue distribution]\label{topsub}
      Does the top eigenvalue in the subcritical regime follow an extreme value distribution? How to verify the transition for the limit distribution of the  top eigenvalue  from the critical regime to the subcritical and the supercritical regimes? 
\end{question}

    \item[\textbf{V.}]\textbf{Non-periodic RBM.}
  While the semicircle law remains valid for random band matrices with non-periodic variance profiles, including cases with i.i.d. entries and general profile functions when $1 \ll W \ll L$, we conjecture that Theorems \ref{metatheoremsup} and \ref{metatheoremsub} continue to hold in both supercritical and subcritical regimes. The critical regime, however, presents fundamentally different behavior that needs careful investigation.

\begin{question}[Geometric dependence of criticality]\label{RBMmanifold}
    Characterize how the critical behavior of random band matrices depends on underlying geometric structures, particularly for discrete manifolds arising from Lie groups and homogeneous spaces. This investigation should reveal what geometric information is encoded in the critical statistics and develop appropriate analytical methods for such settings.
\end{question}

    \item[\textbf{VI.}]\textbf{Bulk statistics.} Erd{\H{o}}s and Knowles 
    pioneered the application of Chebyshev polynomial expansions to study two fundamental problems in random matrix theory: bulk eigenvector delocalization via quantum diffusion methods \cite{erdHos2011quantum}, and mesoscopic spectral statistics through Chebyshev-Fourier expansions \cite{erdHos2015altshuler1}. Their analysis was constrained to the regimes $t \ll W^{\frac{d}{3}}$ and $\eta \gg W^{-\frac{d}{3}}$ due to inherent limitations in Feynman diagram techniques. 
While our singular analysis approach may overcome these technical restrictions, a fundamental question as highlighted in \cite{sodin2014several} remains open: Can the moment method capture the full local statistics in the bulk?

    \item[\textbf{VII.}]\textbf{Power-law RBM.} 
    All the preceding theoretical considerations and open questions naturally extend to power-law random band matrices. For instance,   the universality questions  and  the dimensional dependence of critical phenomena.

    \item[\textbf{VIII.}]\textbf{Connection to $\phi^3_d$ model.} 
    
    Our analysis reveals an interesting connection between the edge statistics of random band matrices and $\phi^3_d$ quantum field theory. Remarkably, the Tracy-Widom distribution appears to correspond to a special case of $\phi^3_d$ theory in dimension $d=0$. This observation naturally leads to:
    
    \begin{question}
       Is there a deeper   connection between RBM and $\phi^3$ model? 
      
    \end{question}    
   
\end{enumerate}

Based on the results and methods  established in this paper,   it seems that Questions~\ref{htRBM},  ~\ref{PT4-6}  and~\ref{TW6} are particularly challenging and may require significant new insights.

\ \

\noindent{\bf Acknowledgements}
We would like to thank  L. Erd\H{o}s for sharing his vision for   random band matrices, and also to thank R. Geng, Y. Gu and F. Yang  for valuable   discussions. G. Zou  is deeply grateful
  to  X. Li and F. Yang for their kind hospitality during his academic stay at Peking  University and Tsinghua University. This work was supported by the National Natural Science Foundation of China \#12371157  and \#12090012.

\appendix
\section{Random walks on the torus} 
\label{section:heat_kernel_profile}
\begin{definition}[Jacobi $\theta$ function] The Jacobi $\theta$ function with covariance matrix $\Sigma>0$ is defined as 
    \begin{equation}  \label{theta}
        \theta(x,\Sigma)=\frac{1}{(\sqrt{2\pi})^{d}\sqrt{\text{det}(\Sigma)}}\sum_{n\in \mathbb{Z}^d}e^{-\frac{1}{2}{(n+x)^T\Sigma^{-1}}(n+x)},\quad x\in \mathbb{R}^d.
    \end{equation}
   \end{definition}
The  Gaussian profile for the  unimodular RBM, equivalently, the transition probability $p_{xy}$,   can thus  be    written as 
    \begin{equation} \label{trp}
        \sigma_{xy}^2 =p_{xy} =\frac{W^d}{ML^d}\theta(\frac{x-y}{L},\frac{W^2}{L^2}\Sigma).
    \end{equation}
 Moreover,    when the bandwidth $W\ll L$, it's easy to see that   
    \begin{equation}
        \sigma^2_{xy}        =\frac{1}{M(\sqrt{2\pi})^{d}} \big(1+O(e^{-c(\frac{L}{W})^2})\big) e^{-\frac{1}{2}({x-y})^T (W^2 \Sigma)^{-1}(x-y)}
    \end{equation}
    for some $c>0$.  This indeed indicates  the Gaussian profile.

Let us   introduce    random walks  on the torus.
\begin{definition}[$f$-RW on the torus]  \label{f-RW}
Given a symmetric density function
$f(x)$ on $\mathbb{R}^d$ i.e. $f(x)=f(-x)$, $S_n$ is said to be a  random walk 
with periodic profile $f$ and bandwidth $W > 0$ on   $\Lambda_L$ , if it is a Markov
chain with transition matrix
\begin{equation}p_{xy}  =\frac{1}{M}  \sum _{n\in \mathbb{Z^d}}f\big(\frac{y-x+nL}{W}\big), \quad  M= \sum _{x\in  \Lambda_L}f\big(\frac{x}{W}\big).\end{equation}
  In particular,  when $f$ is the  Gaussian density as in Definition \ref{defmodel}, we call it a Gaussian  random walk on torus and denote by $p_n(x,y)$  the $n$-step transition probability from $x$ to $y$ in $\Lambda_L$.
\end{definition}

  Asymptotic properties  for the $n$-step transition probability  $p_n(x,y)$  play a  central role.  {{In this section we always}} consider  a  random walk on torus with the  Gaussian profile as in Definition \ref{f-RW}.

\subsection{Asymptotic properties for random walks}
Consider a  random walk on the lattice $\Lambda_L$ with  the transition probability  $p_{xy}$ given in \eqref{trp}. The local limit theorems and uniform  upper bound estimates hold. 

\begin{theorem}[\textbf{Local limit theorem}]\label{prop:gaussian_llt}
As  $W\rightarrow\infty$, if   $ne^{-0.5c_{\Sigma} W^2}\ll 1$, then     \begin{equation} \label{n-tp}
    p_{n}(x,y)= \frac{1}{N}\theta\big(\frac{x-y}{L},\frac{nW^2}{L^2}\Sigma\big) \big(1+O(n^{1+\epsilon}e^{-cW^2})\big),
    \end{equation}
     for $x,y \in \Lambda_{L}$, where 
   $N=L^d$ and      $c_\Sigma>0$ denotes the smallest eigenvalue of   $\Sigma>0$.  Moreover,  under the same assumption   the following   three distinct asymptotics  hold.  
 
    \begin{itemize}
        \item[(i)]({\bf Subcritcal regime}) When    $n\frac{W^2}{L^2}\ll 1$,   \begin{equation}p_{n}(x,y)
                =
        \frac{1}{ (\sqrt{2\pi n W^2}) ^{d}\det\Sigma}
        e^{-\frac{1}{2nW^2}(x-y)^T\Sigma^{-1}(x-y)} \big(1+o(1)\big);\end{equation}
        \item[(ii)] ({\bf Supcritcal regime})  When   $n\frac{W^2}{L^2}\gg 1$,  
        \begin{equation}p_{n}(x,y)=
       \frac{1}{N} \big(1+o(1)\big);\end{equation}
               \item[(iii)] ({\bf Critcal regime})  When  $n\frac{W^2}{L^2}\rightarrow \tau\in (0,\infty)$,  
        \begin{equation}p_{n}(x,y)=  \frac{1}{N} \theta(\frac{x-y}{L},\tau \Sigma) \big(1+o(1)\big).\end{equation} 
    \end{itemize}
\end{theorem}
\begin{proof}
    By the Poisson summation formula
    \begin{equation} \label{thetaduality}
    \begin{aligned}
        \theta(x,\Sigma) 
         =\sum_{n\in \mathbb{Z}^d} e^{-n^T \Sigma n+in\cdot x},
    \end{aligned}
    \end{equation} 
      as the positive number  $q \to \infty$    we have
    \begin{equation} \label{thetaasym}
    \begin{aligned}
        \theta(x,q\Sigma)        &=\sum_{n\in \mathbb{Z}^d} e^{-qn^T \Sigma n+in\cdot x}=1+O(e^{-qc_\Sigma}).
    \end{aligned}
    \end{equation} 
        Therefore,  by  definition of $M$ in \eqref{VP} and  the $\theta$ function  in \eqref{theta} we obtain 
    \begin{equation}\label{equ:A.11}
        \begin{aligned}
            M =W^d \theta(x,W^2\Sigma)=W^d(1+O(e^{-c_{\Sigma}W^2})).
        \end{aligned}
    \end{equation}

Now we consider the discrete convolution of the $\theta$ function:
    \begin{equation}
        \begin{aligned}
        &\frac{1}{N}\sum_{y\in \Lambda_L}\theta(\frac{x-y}{L}, \tau_1  \Sigma)\theta(\frac{y}{L}, \tau_2  \Sigma)=\frac{1}{N(2\pi \sqrt{\tau_1   \tau_2} )^{d} \det(\Sigma)}\\
        &\times \sum_{n\in \mathbb{Z}^d}\sum_{{m}\in \mathbb{Z}^d}\sum_{y\in \Lambda_L}e^{-\frac{1}{2} ({x-y+n}L)^T( \tau_1 L^2 \Sigma)^{-1}({x-y+n}L)-\frac{1}{2}({y+m}L)^T( \tau_2  L^2\Sigma)^{-1}({y+m}L) }\\
        &=\frac{1}{N(2\pi \sqrt{\tau_1   \tau_2} )^{d} \det(\Sigma)}\sum_{n\in \mathbb{Z}^d}\sum_{y\in \mathbb{Z}^d}e^{-\frac{1}{2} ({x-y+n}L)^T( \tau_1 L^2 \Sigma)^{-1}({x-y+n}L)-\frac{1}{2}y^T( \tau_2  L^2\Sigma)^{-1}y}.
        \end{aligned}
    \end{equation}
By  \eqref{thetaasym},   we have
\begin{equation}
\begin{aligned}
    &\sum_{y\in \mathbb{Z}^d}e^{-\frac{1}{2} {(x-y)}^T( \tau_1 L^2 \Sigma)^{-1}{(x-y)}-\frac{1}{2}y^T( \tau_2  L^2\Sigma)^{-1}y}\\
    &=\sum_{y\in \mathbb{Z}^d}e^{-\frac{1}{2}\frac{1}{ \tau_1 + \tau_2 }x^T (L^2\Sigma)^{-1}x-\frac{1}{2}\frac{ \tau_1 + \tau_2 }{ \tau_1  \tau_2 }{(y-}\frac{ \tau_2 }{ \tau_1 + \tau_2 }x)(L^2\Sigma)^{-1}{(y}-\frac{ \tau_2 }{ \tau_1 + \tau_2 }x)}\\
    &=N\sqrt{\det(\Sigma)}\sqrt{\frac{ \tau_1  \tau_2 }{ \tau_1 + \tau_2 }} e^{-\frac{1}{ \tau_1 + \tau_2 }x^T (L^2\Sigma)^{-1}x}\theta(\frac{ \tau_2 }{ \tau_1 + \tau_2 }x, L^2\frac{ \tau_1  \tau_2 }{ \tau_1 + \tau_2 }\Sigma)\\
    &=\big(1+O(e^{-c_{\Sigma}L^2 \frac{ \tau_1  \tau_2 }{ \tau_1 + \tau_2 }})\big)N\sqrt{\det(\Sigma)}\Big(\sqrt{\frac{ 2\pi \tau_1  \tau_2 }{ \tau_1 + \tau_2 }}\Big)^d e^{-\frac{1}{ \tau_1 + \tau_2 }x^T (L^2\Sigma)^{-1}x},
\end{aligned}
\end{equation}
Hence  
\begin{equation}
    \begin{aligned}
        &\frac{1}{N}\sum_{y\in \mathbb{Z}^d}\theta(\frac{x-y}{L}, \tau_1  \Sigma)\theta(\frac{y}{L}, \tau_2  \Sigma)\\
        &=(1+O(e^{-c_{\Sigma}L^2 \frac{ \tau_1  \tau_2 }{ \tau_1 + \tau_2 }}))\sum_{n\in \mathbb{Z}^d}\frac{e^{-\frac{1}{ \tau_1 + \tau_2 }{(x+n)}^T (L^2\Sigma)^{-1}{(x+n)}}}{(\sqrt{2\pi ( \tau_1 + \tau_2 )})^{d}\sqrt{\det(\Sigma})}\\
        &=(1+O(e^{-cL^2 \frac{ \tau_1  \tau_2 }{ \tau_1 + \tau_2 }}))\theta(\frac{x}{L},( \tau_1 + \tau_2 )\Sigma).
    \end{aligned}
\end{equation}

Take $\tau_i=n_i (W/L)^2$ with $n=n_1+n_2$ and   $i=1,2$, noting  that  $$e^{-c_{\Sigma}L^2\frac{ \tau_1  \tau_2 }{ \tau_1 + \tau_2 }}\le e^{-0.5c_{\Sigma}W^2}$$ whenever  $n_i\ge 1$ and  
\begin{equation}
    (1+O(e^{-0.5c_{\Sigma}W^2}))^n=1+O(ne^{-0.5c_{\Sigma}W^2}), 
\end{equation}
which comes   from the simple inequality  with $a\in (0,1]$
\begin{equation} \label{simpleineq}
    |(a+b)^n-a^n| \leq  n|b| e^{n|b|},\quad   n|b|\to 0, 
\end{equation}
 we can obtain the asymptotics \eqref{n-tp} after performing     the convolution $n$ times.

 Finally, based on \eqref{n-tp}, it is easy to see the subcritical regime estimate follows from the definition \eqref{theta}  and  the supercritical and critical   regime  estimates  follow from the   Poisson summation formula
   \eqref{thetaduality}.
\end{proof}

These  three different limits for  the $n$-step transition probability $p_n(x,y)$ indicate
 different behaviors of diagram function $F_{\mathfrak{D}}$ in the  subcritical, critical and  supercritical regimes, as we shall explain in Section \ref{section_diagram}.

\begin{lemma}[\textbf{Heat kernel upper bound}]\label{prop:gaussian_hkb}
 If $ne^{-0.5c_\Sigma W^2}\ll 1$ as $W\to \infty$, then there exist constants $C_1$ and  $C_2$,  independent of $n$ and $W$,  such that   
    \begin{equation}
        p_{n}({x,y})\le C_1\prod_{i=1}^d\Big( \frac{1}{\sqrt{nW^2}}  e^{-C_2\frac{(x_i-y_i)^2}{nW^2}}+\frac{1}{L}\Big),
    \end{equation}
   uniformly  for $x,y \in \Lambda_{L}$.
\end{lemma}
\begin{proof}
    By  the local limit theorem  in Theorem  \ref{prop:gaussian_llt}, since $ne^{-0.5c_{\Sigma}W^2}\ll 1$ as $W\to \infty$,   we see that  
    \begin{equation}
        p_{n}({x,y})\le   \frac{C}{N}\theta(\frac{{x-y}}{L},\frac{nW^2}{L^2}\Sigma) 
    \end{equation} 
    for some constant $C>0$. 
    In order  to  give an upper bound  of the theta function,  we just need to do it  for   the one-dimensional  theta function  $\theta_1(x,t)$.  Indeed, choose   $C_0>c_0>0$ such that $c_0\mathbb{I}\le \Sigma\le C_0\mathbb{I}$, then for $t>0$  we see from \eqref{theta} that 
        \begin{equation}
        \begin{aligned}
        \theta(x,t \Sigma)
        &\le \frac{1}{(\sqrt{ 2\pi c_0 t} )^{d} }\sum_{n\in \mathbb{Z}^d} e^{-\frac{1}{2 C_{0}t}({n+x})^T ({n+x})}\\
        &=\Big(\frac{C_0}{c_0} \Big)^{\frac{d}{2}} \prod_{i=1}^d \theta_1(x_i, C_0 t).
        \end{aligned}
    \end{equation}
     
    It is sufficient to prove 
    \begin{equation} \label{thetaupper}
        \theta_1(x,t)\le C_3\Big( \frac{1}{\sqrt{t}}e^{-C_4\frac{x^2}{t}}+1\Big),
    \end{equation}
     for any  $x\in (-\frac{1}{2},\frac{1}{2}]$ and $t>0$.  For this, we discuss  two cases of $t$.  
     When    $0<t\le 1$, we have 
    \begin{equation}
    \begin{aligned}
        \sqrt{2\pi t}\theta_1(x,t)&=\sum_{n\in \mathbb{Z}}e^{-\frac{(x+n)^2}{2t}}\le e^{-\frac{x^2}{2t}}+2\sum_{n=1 }^{\infty}e^{-\frac{1}{2t}(n-\frac{1}{2})^2}
                \\
        &\le e^{-\frac{x^2}{2t}}+2\sum_{n=1 }^{\infty}e^{-\frac{1}{8t}(2n-1)} \le  e^{-\frac{x^2}{2t}}+ \frac{1}{1-e^{-\frac{1}{4t}}}e^{-\frac{1}{8t}}\\
        &\le Ce^{-\frac{x^2}{2t}}
    \end{aligned}
    \end{equation}
    for some  constant $C>0$.
    When $t\ge 1$,   the Poisson summation formula  gives us  
    \begin{equation}
    \begin{aligned}
        \theta_1(x,t)&=\sum_{n\in \mathbb{Z}}e^{-tn^2+inx} \le \sum_{n\in \mathbb{Z}}e^{-tn^2}\le C.
        \end{aligned}
    \end{equation}
    Combine the two    cases above and    we get the desired  estimate \eqref{thetaupper}.

       Replacing  $x, t$ by $(x-y)/L, n W^2/L^2$ respectively,  we immediately arrive at   the desired proposition.
\end{proof}

We will establish asymptotic independence  and  self-avoiding property 
for  random walks on torus with the  Gaussian profile.

\begin{lemma}[\textbf{Asymptotic independence}]\label{proposition:Gaussian_asymptotic_independence}For  two independent   Gaussian random walks $S_{n}^{1}$ and $S_{n}^{2}$ starting at    the origin as  in Definition  \ref{f-RW},    set 
\begin{equation}A=\{S_{1}^{1}=x_1, S_{n_1}^{1}=x_2, S_{1}^{2}=x_3, S_{n_2}^{2}=x_4
 \} \end{equation} and let    $I(x_1,x_2,n_1;x_3,x_4,n_2)$   be the intersection event of $S_{n}^1$ and $S_{n}^2$ respectively from $x_1$ to $x_2$ with $n_1$ steps and $x_3$ to $x_4$ with $n_2$ steps. Denote by  $R_{n_1,n_2}$   the number of intersection times  
$$R_{n_1,n_2}=\sum_{i=1}^{n_1}\sum_{j=1}^{n_2}1_{\{S_{i}^{(1)}=S_{j}^{(2)}\}}.$$  
If $ne^{-0.5c_{\Sigma}W^2}\ll 1$ as $W\to \infty$, then    
\begin{equation}
    \mathbb{P}\big(I(x_1,x_2,n_1;x_3,x_4,n_2)|A\big)\le  \mathbb{E}[ R_{n_1,n_2}|A]= O\big(\frac{n^2}{N}\big)+\begin{cases}
		O\big(\frac{n^{\frac{3}{2}}}{W}\big), & d=1;\\
		O\big(\frac{n\log n}{W^2}\big), & d=2;\\
		O\big(\frac{n}{W^d}\big), & d>2,
	\end{cases}
\end{equation}
    where $n=n_1+n_2$.
\end{lemma}
\begin{proof}
Without loss of generality, we just consider  the case $\Sigma=\mathbf{I}$ for simplicity, while  the  general case is a straightforward repetition. 
Think of  $x$ as  an intersection point of   two random walks,  we can calculate  the expected number of intersection times       \begin{equation}
        \mathbb{E}[R_{n_1,n_2}|A]=\sum_{1\le i\le n_1}\sum_{1\le j\le n_2}\sum_{x\in \Lambda_L}{\frac{p_{n_1-i}(x_1,x)p_{i}(x,x_2) }{p_{n_1}(x_1,x_2) }}
        {\frac{ p_{n_2-j}(x_3,x)p_{j}(x,x_4)}{ p_{n_2}(x_3,x_4)}}.
    \end{equation}
    By the Chapman-Kolmogorov equation
    \begin{equation}\label{CK}
        \frac{1}{p_{n+m}(a,b)}\sum_{x\in \Lambda_L} p_{n}(a,x)p_{m}(x,b)=1,
    \end{equation}
    take the partial maximum over $x$ and we  obtain   
    \begin{equation}\label{expinter}
    \begin{aligned}
        \mathbb{E}[R_{n_1,n_2}|A]&\le \frac{1}{p_{n_1}(x_1,x_2)}\sum_{1\le i\le n_1}\sum_{1\le j\le n_2} \max_{x\in \Lambda_L} \{p_{n_1-i}(x_1,x)p_{i}(x,x_2)\}\\
        &=\frac{n_2}{p_{n_1}(x_1,x_2)}\sum_{1\le i\le n_1}  \max_{x\in \Lambda_L} \{p_{n_1-i}(x_1,x)p_{i}(x,x_2)\}.
    \end{aligned}
    \end{equation}
   
    Using the identity  
    \begin{equation}\label{equ:A.28}
        \frac{\|x\|^2}{\tau_1} + \frac{\|y\|^2}{\tau_2}-\frac{\|x+y\|^2}{\tau_1+\tau_2} =\frac{\|\tau_2 x-\tau_1 y\|^2}{
        \tau_1\tau_2 (\tau_1+\tau_2)},
    \end{equation}
    and the inequality  $\theta(x,\tau)\le \theta(0,\tau)$ from \eqref{thetaduality}, 
   for fixed $k\in \mathbb{Z}^d$ we derive 
    \begin{align*}
        \sum_{m_1+m_2=k} e^{-\frac{\|m_1+x\|^2}{2\tau_1} - \frac{\|m_2+y\|^2}{2\tau_2}}&= e^{-\frac{\|k+x+y\|^2}{2(\tau_1+\tau_2)}}\sum_{m_1+m_2=k}e^{-\frac{\|\tau_2 k+\tau_1x-\tau_2y-(\tau_1+\tau_2)m_2\|^2}{2\tau_1\tau_2(\tau_1+\tau_2)}}\tag{use \eqref{equ:A.28}} \\
        &\le e^{-\frac{\|k+x+y\|^2}{2(\tau_1+\tau_2)}}\sum_{m_2\in \mathbb{Z}^d}e^{-\frac{\|(\tau_1+\tau_2)m_2\|^2}{2\tau_1\tau_2(\tau_1+\tau_2)}}        \\
        &\le C\Big(1+\big(\frac{\tau_1\tau_2}{\tau_1+\tau_2}\big)^{\frac{d}{2}}\Big) e^{-\frac{\|k+x+y\|^2}{2(\tau_1+\tau_2)}}.
    \end{align*}

  Therefore    we get  
    \begin{equation}
        \begin{aligned}
\theta(x,\tau_1)\theta(y,\tau_2)&=\frac{1}{(2\pi \sqrt{\tau_1  \tau_2})^{d}}\sum_{m_1,m_2\in \mathbb{Z}^d} e^{-\frac{\|m_1+x\|^2}{2\tau_1} - \frac{\|m_2+y\|^2}{2\tau_2}}\\
& \le C\Big(1+\big(\frac{\tau_1\tau_2}{\tau_1+\tau_2}\big)^{\frac{d}{2}}\Big) \theta(x+y,\tau_1+\tau_2),
        \end{aligned}
    \end{equation}
  for some constant $C>0$.   
   
    \begin{equation}
        \frac{\theta(x,\tau_1)\theta(y,\tau_2)}{\theta(x+y,\tau_1+\tau_2)}\le C [1+(\frac{\tau_1+\tau_2}{\tau_1\tau_2})^{\frac{d}{2}}].
    \end{equation}
    
    Combination of  Theorem   \ref{prop:gaussian_llt} and Proposition \ref{prop:gaussian_hkb} yields   for $1\leq i<n_1$
    \begin{equation}\label{equation:division_p1_p2_p3}
    \begin{aligned}
        \frac{1}{p_{n_1}(x_1,x_2)}   \max_{x\in \Lambda_L} \{p_{n_1-i}(x_1,x)p_{i}(x,x_2)\}&\le \frac{C}{N} \max_{x,y \in \Lambda_L}\frac{\theta(x,(n_1-i)\frac{W^2}{L^2})\theta(y,i\frac{W^2}{L^2})}{\theta(x+y,n_1\frac{W^2}{L^2})}\\
        &\le \frac{C}{N}\Big(1+\frac{L^d}{W^d} \big(\frac{n_1}{i(n_1-i)}\big)^{\frac{d}{2}}\Big) \\
        &\le 
       C \Big(\frac{1}{N}+\frac{1}{M} \big(\frac{n_1}{i(n_1-i)}\big)^{\frac{d}{2}}\Big).
    \end{aligned}
    \end{equation}
    Hence we  see from \eqref{expinter} that   
    \begin{equation}
    \begin{aligned}
        \mathbb{E}[ R_{n_1,n_2}|A]
        \le \frac{Cn_2}{M}
         \sum_{i=1}^{n_{1}-1}
        \Big(\frac{n_1}{i(n_1-i)}\Big)^{\frac{d}{2}}+\frac{Cn_1n_2}{N}.
    \end{aligned}
    \end{equation}
    By the simple  fact 
    \begin{equation}
\sum_{i=1}^{n-1}\Big(\frac{n}{i(n -i)}\Big)^{\frac{d}{2}}=\begin{cases}
		O\big(n^{\frac{1}{2}}\big), & d=1;\\
		O\big(\log n\big), & d=2;\\
		O\big(1\big), & d>2,
	\end{cases}
\end{equation}
   application of Markov inequality  gives us  
    \begin{equation}
        \mathbb{P}\big(I(x_1,x_2,n_1;x_3,x_4,n_2)|A\big) \le  \mathbb{E}[ R_{n_1,n_2}|A] \le O\big(\frac{n^2}{N}\big)+\begin{cases}
		O\big(\frac{n^{\frac{3}{2}}}{M}\big), & d=1;\\
		O\big(\frac{n\log n}{M}\big), & d=2;\\
		O\big(\frac{n}{M}\big), & d>2,
	\end{cases}
    \end{equation}
  where    $n=n_1+n_2.$ Notice that by \eqref{equ:A.11} we have $M=(1+o(1))W^d$. This  completes the proof.
\end{proof}
\begin{lemma}[\textbf{Self-avoiding property}]\label{proposition:Gaussian_self_avoiding}
For   the    Gaussian random walk  $S_{n} $   starting at    the origin as  in Definition  \ref{f-RW}, 
 set 
\begin{equation}
B=\{S_{1}=x_1,  S_{n} =x_2\} 
\end{equation}   
and let     
 $J(x_1,x_2,n)$   be the event of the self-intersection of $S_n$  from $x_1$ to $x_2$ in $n$ steps. Denote by  $R_{n_1,n_2}$   the number of intersection times  
$$R_{n}=\sum_{1\le i<j\leq n} 1_{\{S_{i} =S_{j} \}}.$$  
 If $ne^{-0.5c_{\Sigma}W^2}\ll 1$ as $W\to \infty$, then  
\begin{equation}
    \mathbb{P}\big(J(x_1,x_2,n)|B\big)\le  \mathbb{E}[ R_n |B]=O\big(\frac{n^2}{N}\big)+ \begin{cases}
		O\big(\frac{n^{\frac{3}{2}}}{W}\big), & d=1;\\
		O\big(\frac{n\log n}{W^2}\big), & d=2;\\
		O\big(\frac{n}{W^d}\big), & d>2.
	\end{cases}
\end{equation}
\end{lemma}

\begin{proof}
Without loss of generality, we just consider  the case $\Sigma=\mathbf{I}$ for simplicity, while  the  general case is a straightforward repetition. 

Think of  $x$ as  a self-intersection point,  we can calculate  the expected number of intersection times  
        \begin{equation}
    \begin{aligned}
        \mathbb{E}[ R_n |B]&=\sum_{1\le i<j\le n} 
        \frac{1}{p_{n}(x_1,x_2)}  \sum_{x\in \Lambda_L}   p_{i}(x_1,x)p_{j-i}(x,x)p_{n-j}(x,x_2)\\
        &=\frac{1}{p_{n}(x_1,x_2)}\sum_{1\le i<j\le n}\sum_{x\in \Lambda_L}p_{i}(x_1,x)p_{n-j}(x,x_2)p_{j-i}(0,0)\\
        &=\frac{1}{p_{n}(x_1,x_2)}\sum_{1\le i<j\le n}p_{n-j+i}(x_1,x_2)p_{j-i}(0,0)\\
        &=\frac{1}{p_{n}(x_1,x_2)}\sum_{1\le k\le n-1}(n-k)p_{n-k}(x_1,x_2)p_{k}(0,0)\\
        &\le n\sum_{1\le k\le n-1}\frac{p_{n-k}(x_1,x_2)p_{k}(x_2,x_2)}{p_{n}(x_1,x_2)},
    \end{aligned}
    \end{equation}
    where in the third equality    \eqref{CK} has been used. Again 
    by the upper bound \eqref{equation:division_p1_p2_p3} one arrives at 
    \begin{equation}
        \begin{aligned}
              \mathbb{E}[ R_n |B] &\le \frac{Cn}{M}\sum_{i\le n-1} \Big(\frac{n}{i(n-i)}\Big)^{\frac{d}{2}}+C\frac{n^2}{N}.
        \end{aligned}
    \end{equation}
   At last,  apply the Markov inequality and  we get 
    \begin{equation}
        \mathbb{P}(J|B) \le  \mathbb{E}[ R_n |B] \le O(\frac{n^2}{N})+\begin{cases}
		O\big(\frac{n^{\frac{3}{2}}}{M}\big), & d=1;\\
		O\big(\frac{n\log n}{M}\big), & d=2;\\
		O\big(\frac{n}{M}\big), & d>2,
	\end{cases}
    \end{equation}
   which is the desired result since $M=(1+o(1))W^d$. 
\end{proof}

\subsection{Splitting vertices of degree greater than 3}\label{Section:diagram_d>3}

For the weighted diagram with degree greater than 3 vertices, we need to grid rid of these vertices  by splitting  operation, see  Figure \ref{fig:degree_operation}. The following useful proposition shows that 
we can split the vertices with degree greater than 3    in the loss of a unit of   weight by taking  $n_3=1$.

\begin{lemma}[\textbf{Vertex splitting}]\label{proposition:degree_operation} If $(n_1+n_2+n_3)e^{-0.5c_\Sigma W^2}\ll 1$ as $W\to \infty$, 
   then  for  $n_1, n_2\ge n_3\geq 1$ there exists a constant $C>0$ such that  
    \begin{equation}
        p_{n_1}(x_1,x_3)p_{n_2}(x_2,x_3)\le C\sum_{x\in \Lambda_L} p_{n_1}(x_1,x)p_{n_2}(x_2,x)p_{n_3}(x,x_3).
    \end{equation}
In particular, $n_3=1$ is usually chosen.
   \end{lemma}
\begin{proof}
    We first prove the  Gaussian part  and  give an  estimate about Gaussian integrals. For this, 
let 
    \begin{equation}
    \Tilde{p}_{n}(x,y)=\frac{1}{(\sqrt{2\pi})^{d}\sqrt{\text{det}(nW^2\Sigma)}}e^{-\frac{1}{2}{x^T(nW^2\Sigma)^{-1}}x}, \quad  x\in \mathbb{Z}^d.\end{equation}
        By   local limit theorem  as in  Theorem \ref{prop:gaussian_llt}, we have
    \begin{equation}
        p_{n}(x,y)=\frac{1}{M}\big(1+O(ne^{-0.5c_{\Sigma}W^2})\big)\sum_{m\in \mathbb{Z}^d}\Tilde{p}_n(x,y+mL).
    \end{equation}

    For simplicity,  without loss of generality we consider $\Sigma=\mathbf{I}$ only. When  $\sigma_1, \sigma_2\ge \sigma_3$, let $ A=\sigma_1^2\sigma_2^2+\sigma_2^2\sigma_3^2+\sigma_3^2\sigma_1^2$,  noting  the  simple inequalities    $A\leq 3 \sigma_1^2\sigma_2^2$  and 
     \begin{equation}
    \begin{aligned}
        \frac{ax^2+by^2+c(x+y)^2}{ab+bc+ca}\le \frac{x^2}{b}+\frac{y^2}{a},  \quad x,y\in \mathbb{R},\ a,b,c>0,
           \end{aligned}
    \end{equation} 
    we arrive at  
    \begin{equation}
        \begin{aligned}
             \int \prod_{i=1}^3\frac{1}{          (\sqrt{2\pi}\sigma_{i})^{d}}e^{-\frac{1}{2\sigma_i^2}\|x_i-x\|^2} dx &=\frac{1}{(2\pi \sqrt{A})^{d}}e^{-\frac{1}{2A}(\sigma_1^2\|x_2-x_3\|^2+\sigma_2^2\|x_3-x_1\|^2+\sigma_3^2\|x_1-x_2\|^2)}\\
            &\ge \frac{1}{(2\sqrt{3}\pi \sigma_1 \sigma_2)^{d}}
            e^{-\frac{1}{2\sigma_1^2}\|x_3-x_1\|^2-\frac{1}{2\sigma_2^2}\|x_2-x_3\|^2}
                    \end{aligned}
    \end{equation}

    By Poisson summation formula, with $n=n_1+n_2+n_3$, it's easy to see that 
    \begin{equation}
        \sum_{x\in \mathbb{Z^d}} \Tilde{p}_{n_1}(x_1,x)\Tilde{p}_{n_2}(x_2,x)\Tilde{p}_{n_3}(x,x_3)=(1+O(ne^{-0.5c_{\Sigma}W^2})) \int \prod_{i=1}^3\frac{1}{      (\sqrt{2\pi}\sigma_{i})^{d}}e^{-\frac{1}{2\sigma_i^2}\|x_i-x\|^2} dx, 
    \end{equation}
   from which  
        \begin{equation}
        \Tilde{p}_{n_1}(x_1,x_3)\Tilde{p}_{n_2}(x_2,x_3)\le C \sum_{x\in \mathbb{Z^d}} \Tilde{p}_{n_1}(x_1,x)\Tilde{p}_{n_2}(x_2,x)\Tilde{p}_{n_3}(x,x_3).
    \end{equation}
Change $x_1,x_2$   to  $x_1+m_1L$ and $x_2+m_2L$ and   take a  summation over  integers $m_1,m_2$,  we obtain  
    \begin{equation}
    \begin{aligned}
        {p}_{n_1}(x_1,x_3){p}_{n_2}(x_2,x_3)&\le C(1+O(ne^{-0.5c_{\Sigma}W^2})) \sum_{x\in \mathbb{Z^d}} {p}_{n_1}(x_1,x){p}_{n_2}(x_2,x)\Tilde{p}_{n_3}(x,x_3)\\
        &\le C\sum_{x\in \Lambda_L} {p}_{n_1}(x_1,x){p}_{n_2}(x_2,x){p}_{n_3}(x,x_3),
    \end{aligned}
    \end{equation}
    where $x_1, x_2, x_3 \in \Lambda_L$.

    This thus completes the proof.
\end{proof}

\section{Polynomial moment method: analytical aspects}\label{sec:poly_moment_analysis}

\subsection{Properties of renormalization polynomials}\label{sec:re_poly}
The  family of polynomials  $\widetilde{\mathcal{P}}_n$ can be regarded   as a  perturbation of the classical  Chebyshev polynomials,  since  
  all the recursion coefficients $a_{2l}$  decay as  $W^{-d}$.  
  A key question is how this perturbation influences the asymptotic behavior of the Chebyshev polynomials $U_n$, depending on the interplay between 
    the degree  $n$, the bandwidth parameter $W$ and the dimension $d$. 

Our starting point is the generating function of the polynomials $\widetilde{\mathcal{P}}_n$
\begin{equation}
    \widetilde{\mathcal{G}}(z;x)=\sum_{n=0}^\infty \widetilde{\mathcal{P}}_n(2x)z^n,
\end{equation}
  which provides a powerful tool for analyzing the  asymptotic behavior under the given perturbation.
Indeed, by the recursion relation  defined in Proposition \ref{Prop:4.2}
we have
\begin{equation} \label{Gcal}
    \widetilde{\mathcal{G}}(z;x)=\frac{1}{1-2xz+z^2-a_4z^4+\sum_{l=3}^{3R}a_{2l}z^{2l}}.
\end{equation}

Similarly,   the generating functions for polynomials  $\mathcal{P}_n$ and   $U_n$ read 
\begin{equation}
    \mathcal{G}(z;x):=\sum_{n=0}^\infty\mathcal{P}_n(2x)z^n=\frac{1}{1-2xz+z^2-a_4z^4},
\end{equation}
and \begin{equation}
    G(z;x):=\sum_{n=0}^\infty U_n(x)z^n=\frac{1}{1-2xz+z^2}.
\end{equation}

 First,  we  need to derive the asymptotic behaviors for  $a_{2l}$.

\begin{lemma}\label{lem:a_2l}
    Under the Gaussian profile function with $\Sigma=I_{d}$ and with  $\mathcal{E}(R)$  given in \ref{ENNotation}, 
    let $R\geq 2$ be an integer such that  
    $\mathcal{E}(R), RW^{-d}\rightarrow 0$. 
    Then  
    \begin{equation}
        a_{2l}=\Big(1+O\big(RW^{-d}+\mathcal{E}(R)\big)\Big)(2\pi l)^{-\frac{d}{2}}W^{-d},\quad 2\leq l\leq R,
    \end{equation}
and  
    \begin{equation}
        a_{2l}=O\Big(\big(\frac{R}{W^d}\big)^{l-R}W^{-d}\Big), \quad R< l\le  3R.
    \end{equation}
\end{lemma}

\begin{proof}[Proof of Lemma \ref{lem:a_2l}] 
In  the case $l\le R$,  we see from     \eqref{equ:4.13} that 
the sum for $a_{2l}$ can be rewritten as two parts, respectively corresponding  to $s=l$ and $s<l$,
\begin{multline}
    a_{2l}=a_{2l}^{(0)}+a_{2l}^{(1)}:=\\  \Big(\sum_{s=l}+\sum_{3\leq s<l}\Big)
     \sum_{|\mathbf{b}^{(2s)}|=2l,b_1=1}\sum_{\{x_i\}_{i=0}^{2s}}\delta_{x_0,x}\mathbb{1}_{0,2s}\prod_{i=3}^{s-1}(1-\mathbb{1}_{0,2i})\prod_{i=0}^{2l-2}\mathbb{1}(x_i\ne x_{i+2})\prod_{i=1}^{2s}(\Phi_{b_i})_{x_ix_{i+1}}.
\end{multline}

For the leading part, that is,   when  $s=l$ we have $b_i=1$ for all $1\le i\le 2l$ and 
\begin{equation}\label{equ:4.51}
\begin{aligned}
    a_{2l}^{(0)}=\sum_{\{x_i\}_{i=0}^{2l}}\delta_{x_0,x}\mathbb{1}_{0,2l}\prod_{i=3}^{l-1}(1-\mathbb{1}_{0,2i})\prod_{i=0}^{2l-2}\mathbb{1}(x_i\ne x_{i+2})\prod_{i=1}^{2l}(\Phi_1)_{x_ix_{i+1}}.
\end{aligned}
\end{equation}
Note that 
    $\mathbb{1}_{0,2l}=\prod_{t=0}^{l}\delta_{x_t,x_{t+l}}$,
it's easy to see 
\begin{equation}
   \mathbb{1}\{x_1, x_2,\ldots, x_{l-1}~\text{distinct}\} \le  \mathbb{1}_{0,2l}\prod_{i=3}^{l-1}(1-\mathbb{1}_{0,2i})\prod_{i=0}^{2l-2}\mathbb{1}(x_i\ne x_{i+2})\le 1,
\end{equation}
so this implies  
\begin{equation}
    \sum_{x_1,\ldots, x_{l-1}~\text{distinct}}H_{x_0x_1}^2H_{x_1x_2}^2\cdots H_{x_{l-1}x_0}^2\le a_{2l}^{(0)}\le \sum_{x_1,\ldots, x_{l-1}}H_{x_0x_1}^2H_{x_1x_2}^2\cdots H_{x_{l-1}x_0}^2.
\end{equation}
Since the right-hand side is equal to $p_{l}(x_0,x_0)$, use of  Proposition \ref{proposition:Gaussian_self_avoiding} gives  
\begin{equation}
     (1+\mathcal{E}(l))p_{l}(x_0,x_0) \le a_{2l}^{(0)}\le p_{l}(x_0,x_0).
\end{equation}

On the other hand,  we have 
\begin{equation}\label{equ:4.56}
\begin{aligned}
     |a_{2l}^{(1)}|\le \sum_{3\leq s<l}\sum_{|\mathbf{b}^{(2s)}|=2l,b_1=1}\sum_{\{x_i\}_{i=0}^{2s}}\delta_{x_0,x}\mathbb{1}_{0,2s}\prod_{i=1}^{2s}|(\Phi_{b_i})_{x_ix_{i+1}}|. 
\end{aligned}
\end{equation}
For every  $s$, the   number of appearance of  $\Phi_3$ is $l-s$, so there are at most  $\frac{(2s)^{l-s}}{(l-s)!}$ ways to pick the position of $\Phi_3$. Hence, we know from   $|(\Phi_3)_{xy}|\le C_1 W^{-d}|(\Phi_1)_{xy}|$ that  
\begin{equation}
    \begin{aligned}
        |a_{2l}^{(1)}|&\le \sum_{3\leq s<l}\frac{1}{(l-s)!}(\frac{2C_1 s}{W^d})^{l-s}\sum_{\{x_i\}_{i=0}^{2s}}\delta_{x_0,x}\mathbb{1}_{0,2s}\prod_{i=1}^{2s}|(\Phi_1)_{x_ix_{i+1}}|\\
        &=\sum_{3\leq s<l}\frac{1}{(l-s)!}(\frac{2C_1 s}{W^d})^{l-s}p_{s}(x_0,x_0)\\
        &\le  C_2 \sum_{3\leq s<l}   \big(\frac{2C_1 l}{W^d}\big)^{l-s}W^{-d}s^{-\frac{d}{2}}=O\big(\frac{l}{W^d}p_{l}(x_0,x_0)\big),
    \end{aligned}
\end{equation} 
where we use the assumption $lW^{-d}\to 0$ and the sum can be controlled by the last term.

In   the case $l>R$, there must be   many $\Phi_3$ terms in the product of   \eqref{equ:4.13}.  Moreover,  the number of appearance of  $\Phi_3$ is $ l-s$. This implies that  the number to choose the positions of $\Phi_3$ can be bounded by $(2R)^{l-s}$. Note that    each $\Phi_3$ contributes   a $W^{-d}$ factor and the remaining  terms are  bounded by the probability  $p_{s}(x_0,x_0)$, we obtain an upper bound  
\begin{equation}
    |a_{2l}|\le \sum_{3\leq s\leq R} \big(\frac{2C_3 R}{W^d}\big)^{l-s} p_{s}(0,0)
    \le  C_4 W^{-d}\sum_{s\le R}s^{-\frac{d}{2}} \big(\frac{2C_3 R}{W^d}\big)^{l-s}=O\Big((\frac{R}{W^d})^{l-R}W^{-d}\Big).
\end{equation} Here  we have used the assumption $RW^{-d}\to 0$ from which the sum can be controlled by the last term.

Combine the above two case and we thus complete the proof.
\end{proof}

When $d\ge 2$, introduce a function 
\begin{equation}\label{equ:def_A(z)}
    A(z)=W^d(-a_4z^4+\sum_{l=3}^{3R}a_{2l}z^{2l}).
\end{equation}
Then Lemma \ref{lem:a_2l} shows that   $A(z)$ has  similar asymptotic behavior  as the polylogarithm $\text{Li}_{{d}/{2}}(z)$, which is defined by
\begin{equation}
\operatorname{Li}_{\frac{d}{2}}(z) = \sum_{k=1}^\infty \frac{z^k}{k^{\frac{d}{2}}}, \quad |z|<1.
\label{eq:polylog}
\end{equation}
Indeed,  for integer $d\geq 2$,  as $z\to 1$  we have
\begin{equation}\label{equ:polylogarithm}
    \text{Li}_{\frac{d}{2}}(z)= \begin{cases}
        -\log (1-z),&d=2,\\
        \zeta(\frac{3}{2})-2\sqrt{\pi}\sqrt{1-z}+O(|1-z|),&d=3,\\
        \zeta(2)+(1-z)\log(1-z)- (1-z)+O(|(1-z)^2\log (1-z)|),&d=4,\\
        \zeta(\frac{d}{2})+\zeta(\frac{d}{2}-1)(1-z)+O(|z-1|^{\frac{3}{2}}),&d \geq 5.
    \end{cases}
\end{equation}

Since  the generating function  \ref{Gcal}  gives an  integral representation   
\begin{equation}\label{equ:integral}
    \widetilde{\mathcal{P}}_n(2x)=\frac{1}{2\pi i}\oint_{|z|=r}\frac{z^{-n-1}}{1-2xz+z^2+W^{-d} A(z)}dz,
\end{equation} where  some $r\in (0,1)$, 
we need the following lemma to  bound the  coefficient.
\begin{lemma}\label{lem:4.6series}
    Consider two power series
    \begin{equation}
        h(z)^{-1}=\sum_{i=0}^{\infty} h_iz^i, \quad 
        g(z)^{-1}=\sum_{i=0}^{\infty} g_iz^i,
    \end{equation}
    with convergence radius $r\ge 1-\epsilon$, $\epsilon\in (0,1)$. If there are two  fixed constants  $\alpha, \beta$, $0\le \beta\le 1$ and  $\beta< 2\alpha-1$, such that 
    \begin{equation}
        |h|,|g|\ge \delta_1|z-1|^\alpha, ~ |h-g|\le \delta_2 |z-1|^\beta, \quad \forall \ |z|\le 1-\epsilon,
    \end{equation}
    for some $\delta_1,\delta_2>0$, then  
    \begin{equation}
        |h_n-g_n|\le   \frac{\delta_2}{(2\alpha-\beta-1)\delta_{1}^2}2^{\alpha-1-\frac{\beta}{2}}   (1-\epsilon)^{-n-\frac{1}{2}}  
        \epsilon^{\beta-2\alpha+1}.
    \end{equation}
\end{lemma}
\begin{proof}
   Note that  \begin{equation}
        h_n-g_n=\frac{1}{2\pi i}\oint_{|z|=1-\epsilon}\frac{z^{-n-1}(g(z)-h(z))}{h(z)g(z)}dz,
    \end{equation}
  we use   the simple fact 
    \begin{equation} 
        \left|\frac{z^{-n-1}(g(z)-h(z))}{h(z)g(z)}\right|\le \delta_{1}^{-2}\delta_2 (1-\epsilon)^{-n-1}|z-1|^{\beta-2\alpha}
    \end{equation}
   to obtain     
    \begin{equation}\label{b24}
        |h_n-g_n|\le \frac{\delta_2}{2\pi \delta_{1}^2} (1-\epsilon)^{-n-1}\oint_{|z|=1-\epsilon}|z-1|^{\beta-2\alpha}|dz|.
    \end{equation}

      In order to derive an upper estimate for  the above integral, let $z=(1-\epsilon)e^{i\theta}$, then we have
    \begin{equation} \label{b25}
        \oint_{|z|=1-\epsilon}|z-1|^{\beta-2\alpha}|dz|=(1-\epsilon)\int_{-\pi}^{\pi}\left|(1-\epsilon)e^{i\theta}-1\right|^{\beta-2\alpha}d\theta.
    \end{equation}
Simple manipulation  gives   
    \begin{equation}
     \begin{aligned}
        \left|(1-\epsilon)e^{i\theta}-1\right|&=
        \sqrt{\epsilon^2+4(1-\epsilon)\sin^2\frac{\theta}{2}}\ge \frac{1}{\sqrt{2}} \big( \epsilon+2\sqrt{1-\epsilon} |\sin\frac{\theta}{2}| \big) \\
        &\ge\frac{1}{\sqrt{2}} \big( \epsilon+2\sqrt{1-\epsilon} \frac{|\theta|}{\pi} \big),
         \end{aligned} 
    \end{equation}
    whenever $|\theta|\le \pi$.  Noting  the assumption  $\beta-2\alpha<-1$, we  get 
     \begin{equation}
    \begin{aligned}
    \int_{-\pi}^{\pi}\left|(1-\epsilon)e^{i\theta}-1\right|^{\beta-2\alpha}d\theta
    &\leq 2  \big(\frac{1}{\sqrt{2}}\big)^{\beta-2\alpha}  \int_{0}^{\pi}
    \big( \epsilon+2\sqrt{1-\epsilon} \frac{|\theta|}{\pi} \big)^{\beta-2\alpha}d\theta\\
    &\leq  2^{\alpha-\frac{\beta}{2}} \frac{\pi}{2\alpha-\beta-1} \frac{\epsilon^{\beta-2\alpha+1}}{\sqrt{1-\epsilon}}.
     \end{aligned} 
      \end{equation}

      Combining  \eqref{b24} and  \eqref{b25}, we complete the proof. 
\end{proof}

\begin{lemma}\label{lem:5.8}
 For the \textbf{super-critical case}  where $W\gg L^{1-\frac{d}{6}}$ for $d< 4$ or $W\gg L^{\frac{1}{3}}(1+\delta_{d,4}\log L)$ for $d\ge 4$, chose  $R=(\frac{L}{W})^2$. Given any two fixed positive constants $C_1,C_2>0$,  assume that  $n\le  C_1 L^{\frac{d}{3}}$ and 
 \begin{equation}
    x= \sqrt{1+W^{-d}A(1)}+C_2n^{-2}.
\end{equation}
Then  
\begin{equation}\label{equ:4.44}
 \begin{aligned}
\widetilde{\mathcal{P}}_n(2x)
&=\frac{1}{2\pi i}\oint_0\frac{z^{-n-1}}{1-2xz+z^2+W^{-d} A(1)}dz+o(n)\\
&=\big(1+W^{-d} A(1)\big)^{\frac{n}{2}}U_{n}\Big(\frac{x}{\sqrt{1+W^{-d}A(1)}}\Big)+o(n).\\
\end{aligned}
\end{equation}
\end{lemma}
\begin{proof}
Rewrite 
\begin{equation}
 \begin{aligned}
    &1-2xz+z^2+W^{-d} A(z)\\
    &=\big(\sqrt{1+W^{-d}A(1)}-z\big)^2+2\big(\sqrt{1+W^{-d}A(1)}-x\big)z+W^{-d}(A(z)-A(1)).
     \end{aligned}
\end{equation}
Take the contour  $\mathcal{C}:|z|=1-\frac{C_2}{100n}$, then we immediately have
\begin{equation}\label{equ:f_lower_bound}
    \Big|\big(\sqrt{1+W^{-d}A(1)}-z\big)^2+2\big(\sqrt{1+W^{-d}A(1)}-x\big)z\Big|>C_2|z-1|^2.
\end{equation}

We also claim that 
    \begin{equation}\label{equ:Az-A1}
        |A(z)-A(1)|=\begin{cases}
            O(R|z-1|),~~~~&d=2,\\
            O(|z-1|^{\frac{1}{2}}),~~~~&d=3,\\
            O(|z-1|\log (1-z)),~~~~&d=4,\\
            O(|z-1|),~~~~&d>4.
        \end{cases}
    \end{equation} 
Since
\begin{equation}
\begin{aligned}
    |A(z)-A(1)|&=\Big|W^d\big(-a_4(z^4-1)+\sum_{l=3}^{3R}a_{2l}(z^{2l}-1)\big)\Big|\\
    &= O(|z-1|)+|z^2-1|\Big|\sum_{l=3}^{3R}W^da_{2l}(z^{2l-2}+z^{2l-4}+\cdots +1)\Big|\\
    &=|z^2-1|\Big|\sum_{l=0}^{3R}z^{2l}\sum_{i=l+1}^{3R}W^da_{2i}\Big|+O(|z-1|),
\end{aligned}
\end{equation}
  we have  for  $d>2$ 
\begin{equation}
    W^d\sum_{i=l+1}^{3R}a_{2i}\le Cl^{1-\frac{d}{2}},
\end{equation}
while for  $d=2$ 
\begin{equation}
    |z^2-1|\left|\sum_{l=3}^{3R}W^da_{2l}(z^{2l-2}+z^{2l-4}+\cdots +1)\right|\le |z^2-1|\sum_{l=3}^{3R}W^da_{2l}l=O(R|z-1|).
\end{equation}

Now we take $h(z)=1-2xz+z^2+W^{-d} A(z)$, $g(z)=1-2xz+z^2+W^{-d} A(1)$. Noting that $|z-1|\ge \frac{1}{n}$ on the contour, we have  $W^{-d}|A(z)-A(1)|=o(n^{-2})=o(|z-1|^2)$, which 
together with \eqref{equ:f_lower_bound} gives us 
\begin{equation}
    |h|,|g|\ge C|z-1|^2.
\end{equation}
Hence, by Lemma \ref{lem:4.6series} with $\beta=1/2$ for $d=2$ and with $\beta=1$ for $d>2$,  the error term $f_n-g_n$ can be   bounded by
    \begin{equation}
        \begin{cases}
             n^2 W^{-d}R,&d=2,\\
             n^{\frac{5}{2}} W^{-d}, &d=3,\\
           {n^2W^{-d}\log R}, &d=4,\\
           n^2 W^{-d}, &d\geq 5,
        \end{cases}
    \end{equation}
    all of which has order $o(n)$.

    This shows 
    \begin{equation}
        \widetilde{\mathcal{P}}_n(2x)=\frac{1}{2\pi i}\oint_C\frac{z^{-n-1}}{1-2xz+z^2+W^{-d} A(1)}dz+o(n),
    \end{equation}
    from which the second estimate follows. 
\end{proof}

\begin{lemma}\label{lem:P_sub}
    For \textbf{the subcritical case}
    , where $d< 4$ and   $W\ll L^{1-\frac{d}{6}}$, let   $n=[t W^{\frac{2d}{6-d}}]$ and  $R=n$. 
    Set  
       \begin{equation}
    x=\begin{cases}
        1+\frac{-2-\log2 +\log W}{4\pi W^2}-\frac{\hat{x}}{2W^{2}},&d=2,\\
        1+\frac{1}{2(\sqrt{2\pi})^3 W^3}\big(\zeta(\frac{3}{2})-1-\frac{1}{\sqrt{2}}\big)-\frac{\hat{x}}{2W^{4}}, &d=3,
    \end{cases}
\end{equation}
then we have 
\begin{equation}
\frac{1}{n}\widetilde{\mathcal{P}}_n(2x) 
=o(1)+
 \frac{1}{2\pi it} \begin{cases}
       \int_{-i\infty}^{i\infty}\frac{e^{-ty}dy}{y^2+\hat{x}-\frac{1}{2\pi}\log(- y) },&d=2,\\
       \int_{-i\infty}^{i\infty}\frac{e^{-ty}dy}{y^2+\hat{x}-\frac{1}{\pi}\sqrt{-y}},&d=3.
\end{cases}
\end{equation}
\end{lemma}
\begin{proof}
    The value $a_{2l}$ of $2l>n$ do not affect the evaluation of $\widetilde{\mathcal{P}}_n(2x)$, so we can take $\Tilde{a}_{2l}=(2\pi l)^{-\frac{d}{2}}W^{-d}$ for $2l>n$ and $\Tilde{a}_{2l}=a_{2l}$ for $2l\le n$. Let
    \begin{equation}
        \widetilde{A}(z)=W^d(-a_4z^4+\sum_{l\ge 3}\Tilde{a}_{2l}z^{2l}),
    \end{equation}
    then we have 
    \begin{equation}
        \widetilde{\mathcal{P}}_n(2x)=\frac{1}{2\pi i}\oint_{\mathcal{C}}\frac{z^{-n-1}}{1-2xz+z^2+W^{-d} \widetilde{A}(z)}dz.
    \end{equation}
    Take the contour  $|z|=1-W^{-\frac{2d}{6-d}}$, then as $z\to 1$ we see from  \eqref{equ:polylogarithm} that    
    \begin{equation}
        \widetilde{A}(z)=\begin{cases} -\frac{1}{\pi}
            - \frac{1}{2\pi}\log (1-z^2) +o(1), &~~d=2,\\
            (2\pi)^{-\frac{3}{2}}\big(\zeta(\frac{3}{2})-1-\frac{1}{\sqrt{2}}\big)- \frac{1}{\sqrt{2}\pi}  \sqrt{1-z^2}+O(|1-z|),&~~d=3.
        \end{cases}
    \end{equation}
   
   In Lemma \ref{lem:4.6series}, take $h(z)=1-2xz+z^2+W^{-d} \widetilde{A}(z)$ and 
    \begin{equation}
         g(z)=1-2xz+z^2+W^{-d}\begin{cases}
            -\frac{1}{\pi}
            - \frac{1}{2\pi}\log (1-z^2),&~~d=2,\\
            (2\pi)^{-\frac{3}{2}}\big(\zeta(\frac{3}{2})-1-\frac{1}{\sqrt{2}}\big)- \frac{1}{\sqrt{2}\pi}  \sqrt{1-z^2},&~~d=3.
        \end{cases}
    \end{equation}
    Then for the integral associated with  the function $g$, make a change of   variables  $z=W^{-\frac{2d}{6-d}}y-1$, and we know that the leading  contribution changes to an integral over the imaginary axis.  
    
    Thus  we complete  the proof.
\end{proof}

\begin{lemma}\label{lem:B.7_outside_upperbound}
    Let   $x=\sqrt{1-a_4+\delta}$ 
     where $\delta> 0$ is sufficiently small and      $\epsilon^{1-2^{-k}}\ll \sqrt{\delta}$  for  any  given integer $k\ge 1$. Then 
\begin{equation}
   \frac{1}{n} \left|\mathcal{P}_n(2x)\right|\le  C_1\frac{e^{C_2n\sqrt{\delta}}}{n\sqrt{\delta}}
\end{equation} 
for  some  constants $C_1,C_2>0$.  
Moreover, If   $n\sqrt{\delta}\ge 1$, then there are  some constants $C_3,C_4>0$  such that
\begin{equation}\label{equ:B.50}
    \frac{1}{n}\mathcal{P}_n(2x)\ge  C_3e^{C_4n\sqrt{\delta}}.
\end{equation}
\end{lemma}
\begin{proof} By  the   four-term recursion  for $\mathcal{P}_n(z)$ in Definition \eqref{a4Polynomial}, 
      we have  a contour representation of 
      \begin{equation}\label{contourrep}
\begin{aligned}
    \mathcal{P}_n(2x)&=\frac{1}{2\pi i}\oint_{\mathcal{C}}\frac{z^{-n-1}}{1-2xz+z^2-a_4z^4}dz,
\end{aligned}
    \end{equation}
    where $\mathcal{C}$ is a small contour encircling the origin.

Put  $\epsilon=a_4$ and introduce two functions 
    \begin{equation}
        g(z):=1-2xz+z^2-\epsilon z^4, \quad  g_0(z)=1-2xz+z^2.
    \end{equation}
On the contour $\mathcal{C}_2=\{z\in \mathbb{C}:|z|=2\}$,  for $\epsilon$ small enough we have
\begin{equation}
    |g(z)-g_0(z)|\le 16\epsilon < |g(z)|,
\end{equation}
Hence by Rouch\'e's Theorem,   $g(z)$ and $g_0(z)$ have  the same number of zeros    inside $\mathcal{C}_2$, which  is exactly 2.  Moreover, let    $z_1, z_2$     be the two zeros of the equation $g(z)=0$ such that $|z_1|, |z_2|<2$.

  By the  residue theorem, we get  
\begin{equation}\label{equ:B.46}
\begin{aligned}
    \mathcal{P}_n(2x) 
    &=\frac{1}{2\pi i}\oint_{\mathcal{C}_2}\frac{z^{-n-1}}{g(z)}dz-z_1^{-n-1}(g'(z_1))^{-1}-z_2^{-n-1}(g'(z_2))^{-1}\\
    & =-z_1^{-n-1}(g'(z_1))^{-1}-z_2^{-n-1}(g'(z_2))^{-1}+O(2^{-n}).
\end{aligned}
    \end{equation}
where the simple fact has been used  
\begin{equation}\label{equ:B.47}
    \left|\frac{z^{-n-1}}{1-2xz+z^2-\epsilon z^4}\right|\le C2^{-n}, \quad   |z|=2.
\end{equation}

 We just consider the  asymptotic estimate at one point    $z_1$ since the other is similar.  Note that 
    \begin{equation}
        z_1 =x+ \sqrt{x^2-1+\epsilon+\epsilon(z_1^4-1)},
    \end{equation}
we use   the  bound  of $|z_1|\le 2$ to obtain 
\begin{equation}\label{equ:B.54}
\begin{aligned}
        z_1&=\sqrt{1-\epsilon+\delta}+\sqrt{\delta+\epsilon(z_1^4-1)}=1+O(\sqrt{\epsilon}+\sqrt{\delta}).
\end{aligned}
\end{equation}
Thus, the replacement of this bound  leads to 
\begin{equation}
    \sqrt{\delta+\epsilon(z_1^4-1)}=\sqrt{\delta+O(\epsilon\sqrt{\delta}+\epsilon^{\frac{3}{2}})}=O(\sqrt{\delta}+\epsilon^{\frac{3}{4}}).
\end{equation}
substitution of  which into  the first equality of \eqref{equ:B.54}   implies  that 
\begin{equation}\label{equ:B.59}
\begin{aligned}
        z_1&=1+O(\delta+\epsilon)\pm \sqrt{\delta+\epsilon(z_1^4-1)}
        =1+O(\sqrt{\delta}+\epsilon^{\frac{3}{4}}).
\end{aligned}
\end{equation}
Repeat the procedure for $k$ times and  we arrive at 
\begin{equation} \label{z1bound}
    z_1=1+O_k(\sqrt{\delta}+\epsilon^{1-2^{-k}})=1+O_k(\sqrt{\delta}).
\end{equation}

On the other hand,  it's easy to see from  \eqref{z1bound} that 
\begin{equation}
    g'(z_1)=-2x+2z_1-4\epsilon z_1^3=2\sqrt{\delta+\epsilon(z_1^4-1)}-4\epsilon z_1^3=2\sqrt{\delta}+O(\epsilon).
\end{equation}
So  we have 
\begin{equation}\label{equ:B.54f'}
    c\delta^{-\frac{1}{2}}\le |g'(z_1)|^{-1}\le C\delta^{-\frac{1}{2}}.
\end{equation}
for some constant $C>0$.

Combination of  \eqref{equ:B.46},    \eqref{equ:B.59} and \eqref{equ:B.54f'} leads to  
\begin{equation}
   \frac{1}{n} \left|\mathcal{P}_n(2x)\right|\le  C_1\frac{e^{C_2n\sqrt{\delta}}}{n\sqrt{\delta}}.
\end{equation}

Similarly we can obtain the lower bound \eqref{equ:B.50}
 and thus we finish the proof.
\end{proof}

\begin{lemma}\label{lem:B.6_inside_upperbound}
   Assume that   
    $ n \ll \sqrt{\epsilon} W^d$ 
    as $W\to \infty$.  Then 
  there is  an absolute constant $C$ such that
\begin{equation}
|\mathcal{P}_n(2x)|\le \frac{C}{\sqrt{1-x^2}}, \quad  0\le x\le1-\epsilon.
\end{equation}
\end{lemma}
\begin{proof}
By the same argument as in the proof of  Lemma \ref{lem:B.7_outside_upperbound},  we   have
\begin{equation}\label{equ:B.61}
\begin{aligned}
    \mathcal{P}_n(2x)
    &=-z_1^{-n-1}(g'(z_1))^{-1}-z_2^{-n-1}(g'(z_2))^{-1}+O(2^{-n}),
\end{aligned}
\end{equation}
where $z_1$ and $ z_2$ are the two zeros of $g(z):=1-2xz+z^2-\epsilon z^4 $  inside $|z|=2$. 

Using 
\begin{equation}
    z_{1}=x+\sqrt{x^2-1+a_4z_1^4}=x\pm\sqrt{x^2-1}+O\big(\frac{a_4}{\sqrt{x^2-1}}\big), 
\end{equation}
and   $\big|x\pm\sqrt{x^2-1}\big|=1$ for  $-1\le x\le 1$,  together with the assumption $ n \ll \sqrt{\epsilon} W^d$ we derive that   
\begin{equation}\label{equ:B.63}
    |z_1|^{-n-1} =1+O(\frac{n+1}{\sqrt{\epsilon}W^{d}})=1+o(1).
\end{equation}

On the other hand, 
\begin{equation}
    g'(z_1)=-2x+2z_1-4a_4z_1^3=2\sqrt{x^2-1}+O(W^{-d}).
\end{equation}
So we see from  $|\sqrt{x^2-1}|\ge \sqrt{\epsilon}\gg W^{-d}$ that  
\begin{equation}\label{equ:B.65}
    |g'(z_1)|^{-1}\le \frac{C}{\sqrt{1-x^2}}.
\end{equation}

As to $z_2$, similar estimates with \eqref{equ:B.63} and   \eqref{equ:B.65} hold true.  Therefore,  the desired result follows from  \eqref{equ:B.61}.
\end{proof}

Similarly, we can also obtain  the  asymptotic behavior  of $\mathcal{P}_n$ via the well-known Chebyshev polynomial.
\begin{lemma}\label{lem:B.5}
Given any fixed constant $C_0>0$,  
 if $n\ll  W^d$, then 
\begin{equation}
\begin{aligned}
\mathcal{P}_n(2x)&=\big(1-a_4\big)^{\frac{n}{2}}U_{n}\big(\frac{x}{\sqrt{1-a_4}}\big)+o(n), \quad 0\le x\le1-a_4+\frac{C_0} {n^{2}}.
\end{aligned}
\end{equation}
\end{lemma}

Finally, we  just state  some  properties of Chebyshev polynomials whose proofs are well-known.
\begin{lemma}[{\cite[Lemma 11.2]{erdHos2011quantum2}}]\label{lem:B.8}
 There is  a constant $C>0$ such that  
\begin{equation}
U_n(1+x)\ge e^{C n\sqrt{x}}, \quad x\in [0,0.1),
\end{equation}
and  
\begin{equation}
    |U_n(1+x)|\le 2n,\quad -1\le x\le 0,
\end{equation}
uniformly for  $n\ge 1$. 
\end{lemma}
\begin{lemma}\label{lem:B.9_Chebyshev}
    Let $n=[tM]$,  for any   $y\in \mathbb{R}$ we have
    \begin{equation}
        \lim\limits_{M\rightarrow\infty}\frac{1}{n}U_n\big(1+\frac{y}{2M^2}\big)=\frac{\sin(t\sqrt{-y})}{t\sqrt{-y}}.
    \end{equation}
\end{lemma}

\subsection{The continuity theorem}\label{sec:cont_thm}
\begin{theorem}\label{prop:continuity_thm}
Given  a sequence of    measures $\{\sigma_N\}$, for integer $j\geq 1$ define    
\begin{equation}\label{equ:sin4_}
    \widehat{\sigma}_{N}^{(2j)}(x):=\int_{-\infty}^{\infty} \frac{\sin^4 (x\sqrt{-\lambda})}{(x\sqrt{-\lambda})^4}d \sigma_N(\lambda), \quad  x\in (0,\infty).
\end{equation}
If    for  some constant $C_0$  and    for any  $x_0\in [a,b]$ with any given  $b>a>0$,   
 \begin{equation}
      |\widehat{\sigma}_{N}^{(4)}(x_0)|\leq C  
  \end{equation}
 and also    as $N\rightarrow\infty$ 
\begin{equation}
    \widehat{\sigma}_{N}^{(8)}(x) \to \widehat{\sigma}^{(8)}(x),  \ \  \widehat{\sigma}_{N}^{(10)}(x)  \to \widehat{\sigma}^{(10)}(x), \ \forall x>0, 
\end{equation}  for some functions $\widehat{\sigma}^{(8)}(x)$ and $\widehat{\sigma}^{(10)}(x)$, then there exists a measure $\sigma$ such that 
\begin{equation}\label{equ:f=f_N}
    \widehat{\sigma}^{(8)}(x)=\int_{-\infty}^{\infty} \frac{\sin^8 (x\sqrt{-\lambda})}{(x\sqrt{-\lambda})^8}d \sigma(\lambda),\quad   \widehat{\sigma}^{(10)}(x)=\int_{-\infty}^{\infty} \frac{\sin^{10} (x\sqrt{-\lambda})}{(x\sqrt{-\lambda})^{10}}d \sigma(\lambda).
\end{equation}

 Moreover, if for any  fixed $\delta\in (0,2)$ and for some $C>0$, \begin{equation}
     \widehat{\sigma}^{(8)}(x)\le e^{Cx^{2-\delta}}, \quad \widehat{\sigma}^{(10)}(x)\le e^{Cx^{2-\delta}},
 \end{equation}  
  then $\sigma(x)$ is   unique up to some constant and $\sigma_N$ converges vaguely to $\sigma$.
\end{theorem}
\begin{proof} {\bf Step 1: Existence.} 
Choose $x_0,x_1\in [a,b]$ such that
\begin{equation}
    \left|\frac{\sin (x_0\sqrt{-\lambda})}{x_0\sqrt{-\lambda}}\right|^4+\left|\frac{\sin (x_1\sqrt{-\lambda})}{x_1\sqrt{-\lambda}}\right|^4=\begin{cases}
        \Theta_{x_0,x_1}(\frac{1}{1+\lambda^2}),~&\lambda<0,\\
        e^{\Theta_{x_0,x_1}(\sqrt{\lambda})},~&\lambda>0.
    \end{cases}
\end{equation}
Here the notation $f=\Theta(g)$ means that there exist constants $c,C>0$ such that $cg\le f\le Cg$.

By the boundedness assumption in  \eqref{equ:sin4_}, so there exists a constant $C_1,C_2>0$ such that for $\lambda>0$
\begin{equation}\label{equ:negative_bound}
    \sigma_N(0)-\sigma_N(-\lambda)\le C_1(1+\lambda^2), \quad \sigma_N(\infty)-\sigma_N(\lambda)\le C_1e^{-C_2\sqrt{\lambda}}.
\end{equation}
  Introduce a family of new measures  
    \begin{equation}
        dv_N(\lambda)=\begin{cases}
             (1+\lambda^2)^{-1-\epsilon}d\sigma_N(\lambda),&\lambda\le 0,\\
            d\sigma_N(\lambda),&\lambda>0,
        \end{cases}
    \end{equation}
    where $\epsilon$ is a fixed positive number. Then 
    by \eqref{equ:negative_bound}, 
      for $\lambda>0$
    \begin{equation}\label{equ:v_negative_decay}
        v_N(-\lambda)-v_N(-\infty)\le  C_1 (1+\lambda^2)^{-\epsilon},\quad v_{N}(\infty)-v_{N}(\lambda)\le C_1e^{-C_2\sqrt{\lambda}}.
    \end{equation}
  This shows that the sequence    $\{v_N\}$ is tight. Thus by Prokhorov's theorem, $\{v_N\}$ is compact and there is a subsequence $\{v_{N_j}\}$  that converges to some measure $v$. 
        
        Now for any $x<0.01 C_2$, as $j\to \infty$ we have 
    \begin{equation}
    \begin{aligned}
        &\int_{0}^{\infty} \frac{\sin^8 (x\sqrt{-\lambda})}{(x\sqrt{-\lambda})^8}d v_{N_j}(\lambda)=\int_{0}^{t}\frac{\sin^8 (x\sqrt{-\lambda})}{(x\sqrt{-\lambda})^8}d v_{N_j}(\lambda)+O(e^{-C_3t})\\
        &=\int_{0}^{t}\frac{\sin^8 (x\sqrt{-\lambda})}{(x\sqrt{-\lambda})^8}d v(\lambda)+o(1)+O(e^{-C_3t})
      =\int_{0}^{\infty}\frac{\sin^8 (x\sqrt{-\lambda})}{(x\sqrt{-\lambda})^8}d v(\lambda)+o(1)+O(e^{-C_3t}).
    \end{aligned}
    \end{equation}
    Take $t\rightarrow \infty$, as $j\to \infty$ we obtain
    \begin{equation}
        \int_{0}^{\infty} \frac{\sin^8 (x\sqrt{-\lambda})}{(x\sqrt{-\lambda})^8}d \sigma_{N_j}(\lambda)=\int_{0}^{\infty} \frac{\sin^8 (x\sqrt{-\lambda})}{(x\sqrt{-\lambda})^8}d v_{N_j}(\lambda)\rightarrow \int_{0}^{\infty}\frac{\sin^8 (x\sqrt{-\lambda})}{(x\sqrt{-\lambda})^8}d \sigma(\lambda).
    \end{equation}

Similarly, by the tail decay  \eqref{equ:v_negative_decay} we have
    \begin{equation}
    \begin{aligned}
        &\int_{-\infty}^{0} \frac{\sin^8 (x\sqrt{-\lambda})}{(x\sqrt{-\lambda})^8}d \sigma_{N_j}(\lambda) =\int_{-t}^{0}\frac{\sin^8 (x\sqrt{-\lambda})}{(x\sqrt{-\lambda})^8}(1+\lambda^2)^{1+\epsilon}d v_{N_j}(\lambda)+O(t^{-2})\\
        &=\int_{-t}^{0}\frac{\sin^8 (x\sqrt{-\lambda})}{(x\sqrt{-\lambda})^8}(1+\lambda^2)^{1+\epsilon}d v(\lambda)+o(1)+O(t^{-2})\\
        &=\int_{-\infty}^{0}\frac{\sin^8 (x\sqrt{-\lambda})}{(x\sqrt{-\lambda})^8}d \sigma(\lambda)+o(1)+O(t^{-2}).
    \end{aligned}
    \end{equation}
   
  Also, we can prove the similar  convergence of $\widehat{\sigma}_{N_j}^{(10)}(x)$.   Thus, these  ensure     the existence as shown in \eqref{equ:f=f_N}. 

    {\bf Step 2: Uniqueness.} 
    Assume there is another positive measure $\widetilde{\sigma}(\lambda)$ such that
    \begin{equation}\label{equ:bound_sigma}
        \int_{-\infty}^{0}\frac{d\widetilde{\sigma}(\lambda)}{1+\lambda^2}<\infty, \quad \int_{0}^{\infty}e^{x\sqrt{|\lambda|}}d\widetilde{\sigma}(\lambda)\le C_3e^{C_4|x|^{2-\delta}},
    \end{equation} 
    and
    \begin{equation}
        \int_{-\infty}^{\infty} \frac{\sin^8 (x\sqrt{-\lambda})}{(x\sqrt{-\lambda})^8}d \sigma(\lambda)=\int_{-\infty}^{\infty} \frac{\sin^8 (x\sqrt{-\lambda})}{(x\sqrt{-\lambda})^8}d \widetilde{\sigma}(\lambda).
    \end{equation}
  Introduce  a signed measure   $\mu:=\sigma-\widetilde{\sigma}$, then  
    \begin{equation}\label{equ:g=A=B}
        g(x)=\int_{-\infty}^{0-} \frac{\sin^8 (x\sqrt{-\lambda})}{(x\sqrt{-\lambda})^8}d \mu(\lambda)=-\int_{0}^{\infty} \frac{\sin^8 (x\sqrt{-\lambda})}{(x\sqrt{-\lambda})^8}d \mu(\lambda),\quad \forall x\in \mathbb{R},
    \end{equation}
    where the range for  the first integral is chosen to be $(-\infty,0)$ to avoid the possible atom measure $\mu(\{0\})$. 

     Moreover,  we can extend $g(x)$ to a holomorphic function on $\mathbb{C}$ as
    \begin{equation}
        g(z):=-\int_{0}^{\infty} \frac{\sin^8 (z\sqrt{-\lambda})}{(z\sqrt{-\lambda})^8}d \mu(\lambda),
    \end{equation}
    and see from  \eqref{equ:bound_sigma} that 
    \begin{equation}
        |g(z)|\le C_3e^{C_4|z|^{2-\delta}}.
    \end{equation}
    
  We claim    that   $g(z)$ is bounded when $z\in [0,+\infty)\cup i[0, \infty)$.   Indeed,  if $z\in [1,+\infty)$, we have
    \begin{equation}
        |g(z)|=\left|\int_{-\infty}^{0-} \frac{\sin^8 (z\sqrt{-\lambda})}{(z\sqrt{-\lambda})^8}d \mu(\lambda)\right|\le C_5
    \end{equation}
    while  $z\in i[1,+\infty)$,  
    \begin{equation}
        |g(z)|=\left|\int_{0}^{\infty} \frac{\sin^8 (z\sqrt{-\lambda})}{(z\sqrt{-\lambda})^8}d \mu(\lambda)\right|\le C_6\left|\int_{0}^{\infty} \frac{1}{1+\lambda^4}d \mu(\lambda)\right|\le C_7.
    \end{equation}
   This shows that  $g(z)$ is bounded in $[0,1]\cup i[0,1]$ as $g(z)$ is a holomorphic function. So $g(z)$ is bounded on the boundary of the first quadrant. 

By Phragm\'en–Lindel\"of principle and the assumption  condition \eqref{equ:bound_sigma},   $g(z)$ is bounded in the first quadrant.  The same argument applies to all other 3 quadrants. Thus $g(z)$ is bounded on the whole plane and must be a constant. Moreover, 
    \begin{equation}
        g(z)\equiv 0, \label{g=0}
    \end{equation}
 since
    \begin{equation}
        \lim_{x\rightarrow\infty }\int_{-\infty}^{0-} \frac{\sin^8 (x\sqrt{-\lambda})}{(x\sqrt{-\lambda})^8}d \mu(\lambda)=0.
    \end{equation}

    Next, we show that $\mu\equiv 0$.   Firstly,   we treat the uniqueness of $\lambda>0$ part.  Put 
    \begin{equation}
        h(x):=-\int_{0}^{\infty} \cos(x\sqrt{-\lambda}) d \mu(\lambda).
    \end{equation}

Using  the expansion of $(\sin \sqrt{-\lambda}x)^8$  in terms of the  cosine functions of multiple angles
     and differentiating it 
repeatedly,   
    we  see from \eqref{g=0} that 
    \begin{equation}
        0=\frac{d^8 }{dx^8}\big(x^8g(x)\big)= A_1 h(2x)+A_2 h(4x)+A_3 h(6x)+A_4 h(8x)
    \end{equation}
    for some constants $A_j$.    By \eqref{equ:bound_sigma}, we know that $h(z)$ is also analytic function on $\mathbb{C}$, so we can do Taylor's expansion and check the coefficients, which show  that $h$ must be a polynomial. Hence after taking  finite times of derivatives,
    \begin{equation}
        h^{(4k)}=-\int_{0}^{\infty} \cos(x\sqrt{-\lambda}) \lambda^{2k}d \mu(\lambda)=0.
    \end{equation}
    By the uniqueness theorem as in   \cite[Theorem 1.1]{sodin2021uniqueness}    due to   Levitan, Meiman and Vul,  the measure  $d \mu(\lambda)\equiv 0$ for  $\lambda>0$.

    Secondly,  we treat the uniqueness of $\lambda<0$ part. By \eqref{equ:g=A=B} and \eqref{g=0}, we know that
    \begin{equation}
        \int_{-\infty}^{0-} \frac{\sin^8 (x\sqrt{-\lambda})}{(x\sqrt{-\lambda})^8}d \mu(\lambda)=0,
    \end{equation}
    and similarly
    \begin{equation}\label{equ:sinc_10}
        \int_{-\infty}^{0-} \frac{\sin^{10} (x\sqrt{-\lambda})}{(x\sqrt{-\lambda})^{10}}d \mu(\lambda)=0.
    \end{equation}
    By the assumption condition \eqref{equ:bound_sigma}, we can take twice derivatives on the left-hand side of \eqref{equ:sinc_10}
    \begin{equation}
        0=\frac{d^2}{dx^2}\int_{-\infty}^{0-} \frac{\sin^{10} (x\sqrt{-\lambda})}{(\sqrt{-\lambda})^{10}}d \mu(\lambda)=\int_{-\infty}^{0-} \frac{\sin^{8} (x\sqrt{-\lambda})(80-90\sin^2(x\sqrt{-\lambda}))}{(\sqrt{-\lambda})^{8}}d \mu(\lambda), 
    \end{equation}
    from which 
 \begin{equation}
         \int_{-\infty}^{0-} \frac{\sin^{10} (x\sqrt{-\lambda})}{(\sqrt{-\lambda})^{8}}d \mu(\lambda)=0.
    \end{equation}
    
   Let   $dv_1(t)=- t^2(1+t^6)^{-1}d\mu(-t^2)$, it's    easy to  see  
    \begin{equation}
    \begin{aligned}
        0&=\int_{-\infty}^{0-} \frac{\sin^{10} (x\sqrt{-\lambda})}{(\sqrt{-\lambda})^{8}}d \mu(\lambda)
        =\int_{0}^{\infty}\frac{\sin^{10}(xt)}{t^{10}}(1+t^6)dv_1(t).
    \end{aligned}
    \end{equation}
   By \eqref{equ:bound_sigma}, we see that  the Mellin transform  
\begin{equation}
    \mathcal{M}[v](s):=\int_{0}^{\infty}t^{s-1}dv(t)
\end{equation}
is well-defined for $-1\le \Re s\le 1$. 

 Noting that  the Mellin transform  of  $\sin^{10} t$ reads 
    \begin{equation}
        \mathcal{M}[\sin^{10} t](s)=\Big(- \frac{105}{256}2^{-s} + \frac{15}{64} 4^{-s} - \frac{45}{512} 6^{-s} + \frac{5}{256} 8^{-s} - \frac{1}{512} 10^{-s}\Big)\cos(\frac{\pi s}{2})\Gamma(s)
    \end{equation}
for $-10<\Re s<0$, 
  by Plancherel's theorem we see that 
    \begin{equation}
        0=\int_{0}^{\infty}\frac{\sin^{10}(xt)}{ t^{10}}(1+t^6)dv(t)=\frac{1}{2\pi i}\int_{c-i\infty}^{c+i\infty}x^{-s}\big(\mathcal{M}[\sin^{10} t](s-10)+\mathcal{M}[\sin^{10} t](s-4)\big)\cdot \mathcal{M}[v_1](1-s)ds
    \end{equation}
holds for $0<c<2$ and for any $x$. 
 Thus $\mathcal{M}[v_1](1-s)$ does vanish on the line $c+i(-\infty,\infty)$ for some constant $0<c<2$, except for discrete points from zeros of the above factor of the integrand and hence $\mathcal{M}[v_1](1-s)$ vanishes on the whole line.

   So $dv_1(t)\equiv0$ for $t\ge 0$, which implies    $d\mu(\lambda)\equiv0$ for $\lambda<0$. We thus complete  the uniqueness for $\lambda<0$ part.

    At last, we also need to show that there is no atom measure when $\lambda=0$, that is, $\mu(\{0\})=0$. Since when $\lambda=0$, $ (x\sqrt{-\lambda})^{-8}\sin^8(x\sqrt{-\lambda}) =1$, and for the other    part   $d\mu=0$,   we have
    \begin{equation}
        0=\int_{-\infty}^{\infty}\frac{\sin^8(x\sqrt{-\lambda})}{(x\sqrt{-\lambda})^8}d\mu(\lambda)=\mu(\{0\}).
    \end{equation}
    
    Hence,  we complete the whole proof of the desired results.
\end{proof}
\begin{remark} Theorem \ref{prop:continuity_thm} 
  above    also  holds true if we change  the transform 
   \begin{equation}
        \int_{-\infty}^{\infty} \frac{\sin^{2j} (x\sqrt{-\lambda})}{(x\sqrt{-\lambda})^{2j}}d \sigma(\lambda).
    \end{equation}
    from    $j= 4,5$ to any integer $j\geq 4$. The condition $\widehat{\sigma}^{(2j)}(x)\le e^{Cx^{2-\delta}}$ can potentially be relaxed to $\widehat{\sigma}^{(8)}(x)\le e^{Cx^{2}}$ through an application of the techniques developed in \cite[Lemma 2.8.2]{Levi87}.
\end{remark}

\let\oldthebibliography\thebibliography
\let\endoldthebibliography\endthebibliography
\renewenvironment{thebibliography}[1]{
  \begin{oldthebibliography}{#1}
    \setlength{\itemsep}{0.5em}
    \setlength{\parskip}{0em}
}
{
  \end{oldthebibliography}
}

\begin{spacing}{0}
\newcommand{\etalchar}[1]{$^{#1}$}

\end{spacing}
\end{document}